%% file: main.tex
\documentclass[11pt, a4paper]{article}

\usepackage[utf8]{inputenc}
\usepackage[english]{babel}
\usepackage[T1]{fontenc}
\usepackage[colorlinks = true,
            linkcolor = blue,
            urlcolor  = blue,
            citecolor = blue,
            anchorcolor = blue,
            hyperfootnotes=false]{hyperref}
\usepackage[a4paper,width=150mm,top=25mm,bottom=25mm]{geometry}
\usepackage{fancyhdr}
\usepackage[symbol]{footmisc}

\usepackage{amsmath,amsthm,amssymb, enumitem, mathtools, mathrsfs}
\usepackage{siunitx}
\setitemize{noitemsep}

\usepackage{tikz}
\usepackage{tikz-cd}
\usepackage{caption}
\usepackage[all]{xy}
\usepackage{graphicx}
\usepackage{chngcntr}
\usepackage{booktabs}
\usepackage[labelformat=simple]{subcaption}

\usepackage{pdfpages}

\usepackage{csquotes}
\usepackage[
backend=biber,
sorting=anyt,
minbibnames=4,
maxbibnames=4
]{biblatex}
\graphicspath{ {images/} }

\theoremstyle{plain}
\newtheorem{theorem}{Theorem}[section]
\newtheorem*{theorem*}{Theorem}
\newtheorem*{satz*}{Satz}
\newtheorem{proposition}[theorem]{Proposition}
\newtheorem{lemma}[theorem]{Lemma}
\newtheorem{corollary}[theorem]{Corollary}
\theoremstyle{definition}
\newtheorem{definition}[theorem]{Definition}
\newtheorem{assumption}[theorem]{Assumption}

\theoremstyle{remark}
\newtheorem{remark}[theorem]{Remark}

\numberwithin{equation}{section}

\DeclareMathOperator{\id}{id}

\DeclareMathOperator{\trace}{tr}

\newcommand{\ev}{\mathbb{E}}
\newcommand{\pr}{\mathbb{P}}

\newcommand{\R}{\mathbb{R}}
\renewcommand{\P}{\mathcal{P}}
\newcommand{\F}{\mathcal{F}}
\newcommand{\B}{\mathcal{B}}

\newcommand{\leb}{\mathrm{Leb}}
\renewcommand{\L}{\mathcal{L}}
\newcommand{\M}{\mathcal{M}}
\renewcommand{\d}{\mathrm{d}}

\newcommand{\define}{\mathpunct{:}}
\newcommand{\bb}[1]{\mathbb{#1}}
\renewcommand{\bf}[1]{\mathbf{#1}}
\renewcommand{\cal}[1]{\mathcal{#1}}

\title{Title}
\author{Philipp Jettkant}

\begin{document}



\noindent
\begin{center}
    \Large
    \textbf{Optimal Control of the \\
    Nonlinear Stochastic Fokker--Planck Equation}

    \vspace{1em}

    \normalsize
    Ben Hambly\footnote[1]{Mathematical Institute, University of Oxford, United Kingdom.} \& Philipp Jettkant\footnote[2]{Mathematical Institute, University of Oxford, United Kingdom and Department of Mathematics, Imperial College London, United Kingdom, Corresponding Author, \href{mailto:p.jettkant@imperial.ac.uk}{p.jettkant@imperial.ac.uk}.}

\end{center}

\vspace{1em}

\begin{abstract}
\thispagestyle{plain}
\small
We consider a control problem for the nonlinear stochastic Fokker--Planck equation. This equation describes the evolution of the distribution of nonlocally interacting particles affected by a common source of noise. The system is directed by a controller that acts on the drift term with the goal of minimising a cost functional. We establish the well-posedness of the state equation, prove the existence of optimal controls, and formulate a stochastic maximum principle (SMP) that provides necessary and sufficient optimality conditions for the control problem. The adjoint process arising in the SMP is characterised by a nonlocal (semi)linear backward SPDE for which we study existence and uniqueness. We also rigorously connect the control problem for the nonlinear stochastic Fokker--Planck equation to the control of the corresponding McKean--Vlasov SDE that describes the motion of a representative particle. Our work extends existing results for the control of the Fokker--Planck equation to nonlinear and stochastic dynamics. In particular, the sufficient SMP, which we obtain by exploiting the special structure of the Fokker--Planck equation, seems to be novel even in the linear deterministic setting. We illustrate our results with an application to a model of government interventions in financial systems, supplemented by numerical illustrations.
\end{abstract}

\normalsize

\section{Introduction} \label{sec:introduction}

\input{1_introduction/introduction}

\input{1_introduction/notation}

\section{Main Results} \label{sec:main_results}

\input{2_main_results/main_results}

\section{Semilinear BSPDEs with C\`adl\`ag Noise} \label{sec:bspde}

\input{3_bspde/hilbert_space_martingales}

\input{3_bspde/ito_formula}

\input{3_bspde/bspde}

\section{Optimal Control for the Random Fokker--Planck Equation} \label{sec:control}

\input{4_control/fpe}

\input{4_control/optimal}

\input{4_control/smp}

\section{Application to a Model of Government Interventions in Financial Systems} \label{sec:application}

\input{5_application/application}

\appendix

\section{Appendix} \label{sec:appendix}

\subsection{Equivalence between the Mean-Field Control Problem and the Control of the Nonlinear Stochastic Fokker--Planck Equation} \label{sec:equivalence}

\input{6_appendix/equivalence}

\subsection{Some Properties of the Linear Functional Derivative}

\input{6_appendix/ldf}

\subsection{Auxiliary Results for Section \ref{sec:bspde}}

\input{6_appendix/stoch_int_conv}

\subsection{Auxiliary Results for Section \ref{sec:control}}

\input{6_appendix/int_conv}

\section*{Acknowledgement}
This research has been supported by the EPSRC Centre for Doctoral Training in Mathematics of Random Systems: Analysis, Modelling and Simulation (EP/S023925/1).

\sloppypar
\printbibliography

\end{document}

%% file: 1_introduction/introduction.tex
In this article we study the optimal control of the nonlinear stochastic Fokker--Planck equation
\begin{equation} \label{eq:fpe_intro}
    \d \langle \nu_t, \varphi\rangle = \bigl\langle \nu_t, \L\varphi(t, \cdot, \nu_t, \gamma_t)\bigr\rangle \, \d t + \bigl\langle \nu_t, (\nabla \varphi)^{\top} \sigma_0\bigr\rangle \, \d W_t
\end{equation}
for $\varphi \in C^2_c(\R^d)$ and $t \in [0, T]$. Here the second-order differential operator $\L$ is the generator of a nonlinear diffusion process, the control $\gamma$ is a space-dependent progressively measurable process, $W$ is a $d_W$-dimensional Brownian motion, and $\sigma_0 \in \R^{d \times d_W}$. A controller chooses the control $\gamma$ with the objective of minimising the cost functional
\begin{equation*}
    J(\gamma) = \ev\biggl[\int_0^T \langle \nu_t, f(t, \cdot, \nu_t, \gamma_t)\rangle \, \d t + \psi(\nu_T)\biggr]
\end{equation*}
over $\gamma$ in some admissible set. The measure-valued solution $\nu$ of the Fokker--Planck equation corresponds to the conditional law of a representative particle whose motion is described by a controlled McKean--Vlasov SDE with generator $\L$ and a common noise, modelled by the Brownian motion $W$. The latter problem is referred to as McKean--Vlasov or mean-field control (MFC) and has received widespread attention in recent years \cite{andersson_smp_mfc_2011, lauriere_dp_mfc_2014, carmona_fbsde_smp_2015, pham_dynamic_programming_mkv_2017, djete_mkv_control_limit_2022, cardaliaguet_rate_mfc_2023}. While most of the literature studies MFC problems from the point of view of the McKean--Vlasov SDE, in our analysis we emphasise the stochastic Fokker--Planck equation \eqref{eq:fpe_intro}. This perspective is particularly appropriate when the generator $\L$ includes zeroth-order terms, in which case the measure $\nu_t$ gains or loses mass over time. This change in mass can be interpreted as the insertion or killing of particles at a state-dependent rate in the associated McKean--Vlasov SDE. Models featuring such mechanisms were recently introduced in \cite{hambly_mvcp_arxiv_2023} and \cite{carmona_nssc_2023}. The insertion and killing mechanism manifests itself as an additional process in the McKean--Vlasov SDE, appended to the particle's state, that tracks the cumulative insertion or killing intensity. As we prove in Theorem \ref{thm:equivalence}, knowledge of this process is not necessary to optimally control the McKean--Vlasov SDE. Nonetheless, we cannot dispense with it in the specification of the representative particle. On the other hand, in the Fokker--Planck formulation of the control problem, the intensity is encoded as a zeroth-order term in the generator $\L$. No additional equations are required, leading to a cleaner description. 

As indicated above, the goal of the present paper is to provide a comprehensive analysis of the control problem for the nonlinear stochastic Fokker--Planck equation \eqref{eq:fpe_intro}. First, we establish the well-posedness of this equation for any admissible control $\gamma$. Our next objective is to prove the existence of optimal controls through the use of tightness arguments. Then, we formulate both a necessary and a sufficient stochastic maximum principle (SMP) for the control problem. Lastly, we apply our results to the MFC model for government bailouts from \cite{hambly_mvcp_arxiv_2023}, illuminating the optimal range for capital injections by the government. We also discuss the equivalence between the optimal control of the nonlinear stochastic Fokker--Planck equation and the optimal control of the associated McKean--Vlasov SDE to justify our approach (see Appendix \ref{sec:equivalence}). In our analysis we exploit that the stochasticity in the Fokker--Planck equation can be removed by shifting the solution $\nu$ of this stochastic PDE (SPDE) along the noise $W$. This requires a constant (or, more generally, time-dependent) diffusion coefficient $\sigma_0$. Through the shift, the stochastic Fokker--Planck equation is transformed into a PDE with random coefficients -- a nonlinear random Fokker--Planck equation. This PDE captures the evolution of the distribution of a living representative particle whose state dynamics follow a McKean--Vlasov SDE with random coefficients. The same method was applied in \cite{cardaliaguet_master_2019} and, more recently, in \cite{cardaliaguet_fo_mgf_2022, cardaliaguet_mfg_common_2022} to study the existence of mean-field game (MFG) equilibria. (In a MFG, the particles described by the Fokker--Planck equation are in competition as opposed to being controlled simultaneously by a single agent.) 

It is convenient to study the random Fokker--Planck equation as a stochastic evolution equation in $L^2(\R^d)$, assuming the initial measure $\nu_0$ admits a density $\rho_0 \in L^2(\R^d)$ with respect to the Lebesgue measure. A slight complication arises because, inspired by the correspondence with McKean--Vlasov SDEs, the nonlinearities in the measure argument, i.e.\@ with respect to $\nu_t$, of the generator $\L$ are assumed to be Lipschitz continuous with respect to the bounded Lipschitz distance on the space of measures. In general, such nonlinearities, when viewed as a function of the density of a measure, will not be continuous with respect to the $L^2$-norm. This leads us to consider a weighted $L^2$-space whose norm dominates the bounded Lipschitz distance in a suitable way. In this framework, the adjoint process arising in the necessary and the sufficient SMP is naturally an element of the dual of this weighted $L^2$-space. This dual can be identified as the weighted $L^2$-space whose weight is the multiplicative inverse of the original weight. While the shifted state process satisfies a random PDE, the adjoint process is nonetheless characterised through a stochastic equation. As for classical control problems, this equation runs backwards and to ensure adaptedness of the adjoint process, a martingale is introduced in the adjoint PDE, which removes randomness as one steps backwards in time. In other words, the adjoint process satisfies a linear backward SPDE (BSPDE). Since the tightness arguments employed to prove the existence of an optimal control force us to work in a general filtration, which may extend that of the Brownian motion $W$, the martingale driving the BSPDE may have merely c\`adl\`ag as opposed to continuous trajectories. In addition, the nonlinearities of the Fokker--Planck equation translate into nonlocalities for this BSPDE. In summary, the adjoint process is characterised by a nonlocal BSPDE with c\`adl\`ag martingale driver. 

By inserting the necessary optimality condition from the SMP into the adjoint BSPDE, we derive a nonlocal semilinear BSPDE of Hamilton--Jacobi--Bellman (HJB) type. We provide a general existence and uniqueness analysis of such nonlocal (semi)linear BSPDEs with c\`adl\`ag martingale driver. Together with the random Fokker--Planck equation, the stochastic HJB equation satisfied by the adjoint process forms a coupled forward-backward system of stochastic PDEs reminiscent of the classical MFG systems appearing in the pioneering work of Lasry \& Lions \cite{lasry_mfg_st_2006, lasry_mfg_fh_2006, lasry_mfg_2007}. Indeed, our forward-backward SPDE, deduced from the optimality condition of the necessary SMP, characterises a MFG equilibrium with extended running and terminal cost functions deriving from $f$ and $\psi$. This connection between MFC and MFGs was already established in \cite{lasry_mfg_2007}, but we also refer to the account in \cite[Section 6.2.5]{carmona_mfg_2018} which resembles our setup more closely. Under a rather restrictive convexity assumption and if the minimiser of the Hamiltonian associated with the control of SPDE \eqref{eq:fpe_intro} is unique, it follows from the sufficient SMP that this forward-backward SPDE has a unique solution. This in turn implies that the associated MFG equilibrium is unique. 

Let us note that the convexity assumptions required for the sufficient SMP imply linearity of the Fokker--Planck equation, though not that of the cost functions. However, even for the control of the linear Fokker--Planck equation our sufficient optimality condition seems to be new in the literature. This is likely to be due to the fact that the structure of the Fokker--Planck equation generically precludes joint convexity in the state and the control variable (except in pathological situations), which is crucially used to deduce the sufficient SMP for standard (stochastic) optimal control problems. Fortunately, this same structure can be exploited to establish sufficiency in the absence of joint convexity; an observation that seems to have been gone unnoticed so far. In particular, as soon as both the Fokker--Planck equation and the cost functionals are linear, the sufficient SMP holds.

\subsection{Related Literature} \label{sec:related_literature}

The interest in controlling the Fokker--Planck equation stems from the fact that it describes the law of a particle whose motion is described by an SDE. Hence, controlling the Fokker--Planck equation corresponds to the control of the particle's trajectories. A number of papers have studied the optimal control of deterministic Fokker--Planck equations, in the linear case \cite{barbu_ana_contr_1993, fleig_oc_fpe_2017, breitenbach_mp_fpe_2020, anita_sde_fpe_2021, daudin_oc_const_2023} or where the equation has local \cite{barbu_nonlin_fp_ctrl_2023, anita_nonlin_fp_2024} or nonlocal nonlinearities \cite{bensoussan_mfg_mfc_2013, carrillo_mfc_2020}. As in our case, the objective of many of the cited works, particularly those in the linear setting, is to establish the existence of optimal controls as well as the necessary SMP. In this article, we borrow from the exposition in \cite{anita_sde_fpe_2021} to construct an optimal control. However, due to the stochasticity present in our setup, the proof becomes more involved. Indeed, the tightness arguments we employ force us to introduce a weak formulation of the control problem, where the filtration underlying the admissible controls may be larger than that generated by the noise $W$. We must show that this relaxation does not change the optimal value of the control problem. We achieve this by adapting ideas of Djete, Possama\"i \& Tan \cite{djete_mkv_control_limit_2022} developed in the context of McKean--Vlasov control problems with common noise. Since our state equation only has random coefficients as opposed to a stochastic driver, we can simplify their arguments considerably. 

For linear equations, the necessary SMP can be derived using the spike variation technique, without a need for a second-order adjoint process (cf.\@ \cite{breitenbach_mp_fpe_2020, anita_sde_fpe_2021}). In our nonlinear setting, changing the control on a set of measure $\epsilon$ leads to a variation in the state of order $\epsilon^{1/2}$, so a second-order adjoint would become necessary. To avoid this, we opted for a formulation of the necessary SMP with only a first-order adjoint at the price of convexity assumptions on the domain of the control and the Hamiltonian. Let us stress once more that we also establish a sufficient optimality condition for control of the Fokker--Planck equation. This result cannot be found in the aforementioned literature, even in the linear and deterministic setting.

Linear BSPDEs and their precursors were initially explored in the context of the optimal control of SPDEs by Pardoux \cite{pardoux_bspde_1980}, Bensoussan \cite{bensoussan_smp_spde_1983, benoussan_mp_dp_1983}, Zhou \cite{zhou_nec_spde_1993, zhou_duality_1992}, among others.  Their particular focus was on the Zakai equation, which arises in the filtering of partially observed diffusions. As in our case, the BSPDEs characterise the adjoint processes appearing in the necessary optimality condition. However, all of the cited articles consider linear SPDEs (such as the Zakai equation) and work in a Brownian filtration (in the context of filtering that would be the filtration of the Brownian motion driving the observation process). Consequently, they do not have to deal with nonlocalities or a c\`adl\`ag martingale driver. For the latter, we crucially draw on a generalisation of It\^o's formula for c\`adl\`ag semimartingales in $L^2$-spaces, see Theorem \ref{thm:ito}.

The semilinear BSPDEs analysed in \cite{hu_semilinear_bspde_1991, peng_stochastic_hjb_1992, cardaliaguet_mfg_common_2022} are structurally closest to those in our work. Peng \cite{peng_stochastic_hjb_1992} studies a local stochastic HJB equation arising in the control of SDEs with random coefficients. The McKean--Vlasov SDE associated with the Fokker--Planck equation \eqref{eq:fpe_intro} takes this form if we view the dependence on the conditional law $\nu_t$ in its coefficients as a fixed random input. However, unlike in \cite{peng_stochastic_hjb_1992}, this source of randomness $\nu_t$ is not fixed and varies under perturbations of the control $\gamma_t$, which is the cause of the nonlocalities present in the BSPDEs in our setting. Cardaliaguet, Seeger \& Souganidis \cite{cardaliaguet_mfg_common_2022} analyse a local stochastic HJB equation using an alternative solution theory based on the propagation of certain semiconcavity estimates, which does not require uniform parabolicity. Their arguments, however, cannot be straightforwardly adapted to nonlocal equations. Finally, let us mention that to the best of our knowledge the only articles that consider a similar class of BSPDEs in a non-Brownian filtration are \cite{al_hussein_bspde_2006, al_hussein_bspde_2009}. However, these works explicitly assume that the filtration only generates continuous martingales in order to avoid jumps in the dynamics.


MFGs and -- to a lesser extent -- MFC have seen a surge of interest in the last decade and a half and have been studied both from a probabilistic and analytic perspective. Particularly for MFGs, the analytic approach, which leads to a coupled forward-backward PDE consisting of the Fokker--Planck equation and the HJB equation, has received considerable attention \cite{lasry_mfg_st_2006, lasry_mfg_2007, cardaliaguet_mfg_2014, gomes_mfg_2016, cardaliaguet_master_2019, cardaliaguet_fo_mgf_2022, cardaliaguet_mfg_common_2022}. The corresponding system for MFC problems, see e.g.\@ \cite{bensoussan_mfg_mfc_2013, achdou_mfc_pde_2015, carmona_mfg_2018}, features less prominently in the literature. Its derivation as the necessary optimality condition of the SMP is rather different from the derivation of the MFG system, which proceeds via dynamic programming. Since the proof of the necessary SMP relies on the stability of the state equation under small perturbation of the control, it requires strict regularity conditions, such as uniform parabolicity of the generator $\L$. The derivation of the MFG system is less burdensome.
Moreover, in contrast to the MFG system, the HJB equation of the MFC system exhibits nonlocalities, so one cannot rely on the theory of viscosity solutions or the comparison principle to analyse the equation, unless the nonlocalities are restricted to zeroth-order terms as e.g.\@ in \cite{carmona_nssc_2023}. 

Several authors have established necessary and sufficient SMPs for the control of the McKean--Vlasov SDE underlying the stochastic Fokker--Planck equation \eqref{eq:fpe_intro} in the open-loop formulation. Andersson \& Djehiche \cite{andersson_smp_mfc_2011} treat coefficients that depend on moments of the population's distribution and derive a necessary SMP for convex control spaces. A similar result was obtained by Li \cite{li_smp_mfc_2012} for scalar interactions. Buckdahn, Djehiche \& Li \cite{buckdahn_mfc_smp_2011}  
derive the first SMP for McKean--Vlasov control with a general action space, requiring the application of spike variation techniques. General dependence on the population's law are treated in the monographs \cite{bensoussan_mfg_mfc_2013} and \cite{carmona_mfg_2018} and by Carmona \& Delarue \cite{carmona_fbsde_smp_2015}, who also provide a sufficient condition. Extended mean-field control, where interaction is through the joint law of the population and the control is covered in \cite{acciaio_smp_mfc_2019}. As for the control of standard SDEs, all of the above works require differentiability of the coefficients with respect to the state variable, which is not needed for our approach. Moreover, while in our approach the regularity of the coefficients and cost functions with respect to the measure argument is captured through the so-called linear functional derivative (cf.\@ Definition \ref{def:linear_functional_derivatve}), when proceeding via the McKean--Vlasov SDE, the Lions derivative (see \cite[Definition 5.22]{carmona_mfg_2018}) is the appropriate concept. For sufficiently regular function, \cite[Proposition 5.51]{carmona_mfg_2018} shows that the Lions derivative coincides with the spatial derivative of the linear function derivative. So again, we require less regularity. Note, however, that our SMP characterises closed-loop as opposed to open-loop controls.

We should remark that our framework does not cover models where the killing of particles occurs at a fixed boundary. These have received increasing attention recently, both in an uncontrolled setting \cite{lions_cond_proc_2016, hambly_mckean_vlasov_absorbing_2017, hambly_mckean_vlasov_blow_up_2019, hambly_spde_model_2019, nadtochiy_ps_singular_inter_2019} and in the context of MFGs \cite{campi_mfg_absorption_2018, campi_mfg_hitting_2021, burzoni_mean_field_absorption_2023, nadtochiy_mean_field_network_2020}. Instead of a zeroth-order term, killing at a boundary leads to a Dirichlet boundary condition for SPDE \eqref{eq:fpe_intro}. Such nonlinear stochastic Fokker--Planck equations with Dirichlet boundary condition were studied in \cite{hambly_mckean_vlasov_absorbing_2017, hambly_spde_model_2019}. To establish uniqueness, one has to perform a careful analysis of the regularity of the solution near the boundary. Embedding a control problem into this framework is an interesting topic for future research.

Finally, let us note here that the proof of Theorem \ref{thm:equivalence}, which establishes the equivalence of the McKean--Vlasov control problem (in open-loop formulation) and the control problem for the corresponding stochastic Fokker--Planck equation with space-dependent control processes $\gamma$ as introduced above, relies on a variation of the mimicking theorem from \cite{lacker_mimicking_2020} combined with classical measurable selection arguments. The main difference to standard McKean--Vlasov control problems is that we must apply the measurable selection to an open-loop control of the McKean--Vlasov control problem in such a way that the resulting feedback control $\gamma$ does not depend on the auxiliary cumulative intensity process. For otherwise, we cannot derive a stochastic Fokker--Planck equation for the conditional distribution $\nu_t$ of the living particles. If the Hamiltonian of the McKean--Vlasov control problem is convex, such a selection is indeed possible. This yields a rather straightforward proof of the equivalence between both formulations of the control problem. A similar result was concurrently obtained by Carmona, Lauri\`ere \& Lions \cite{carmona_nssc_2023} in the related context of the control of conditional processes. Their proof proceeds with a careful analysis of a nonlocal PDE of HJB type that characterises the optimal controls of the McKean--Vlasov control problem. Our approach simplifies the proof considerably.

\subsection{Main Contributions and Structure of the Paper}

In Section \ref{sec:main_results} we present the main results of the paper. We discuss the transformation of the stochastic Fokker--Planck equation to a Fokker--Planck equation with random coefficients, state the results concerning the well-posedness of the state equation and existence of optimal controls, and formulate the necessary and the sufficient SMP.

We present a general theory of nonlocal semilinear BSPDEs with c\`adl\`ag drivers in Section \ref{sec:bspde}. We first establish a generalisation of It\^o's formula for c\`adl\`ag semimartingales with values in $L^2(\R^d)$. This extension of It\^o's formula is a key tool that is used frequently throughout this work and also seems to be novel in this generality. Subsequently, we prove existence and uniqueness for the BSPDE based on a Galerkin approximation and fixed-point arguments. 

The analysis of the control problem is the topic of Section \ref{sec:control}. We begin by establishing the equivalence between the stochastic Fokker--Planck equation and its transformation -- the Fokker--Planck equation with random coefficients. For the latter, we prove existence and uniqueness given any control input, so the control problem is well-posed. Then, we construct an optimal control by combining compactness theorems from PDE theory with tightness arguments from stochastic analysis. Since the tightness arguments require us to change or, alternatively, enlarge the underlying probability space, we must verify that the optimal cost of the control problem does not depend on the underlying probability space or its filtration. We achieve this by adapting ideas from \cite{djete_mkv_control_limit_2022}. The next goal is to prove the necessary SMP. This requires an analysis of the infinitesimal variations of the state process under perturbations of the control, which is complicated by the nonlinearities of the Fokker--Planck equation. We combine these variations with the adjoint process to deduce an explicit expression for the Gateaux derivative of the cost functional. From this expression and standard convexity assumptions, we deduce the necessary condition of the SMP. We conclude the chapter by establishing the novel sufficient SMP. We conclude Section \ref{sec:control} by establishing the sufficient SMP.

In Section \ref{sec:application}, we give a concrete application of the results of this paper in the context of a MFC model of government bailouts in financial systems from \cite{hambly_mvcp_arxiv_2023}. In particular, we discuss the shape of the adjoint process, supplemented by numerical simulations for the noiseless case (i.e.\@ $\sigma_0 = 0$), and the implications that can be drawn for the control problem.

%% file: 1_introduction/notation.tex
\subsection{Notation}

We conclude this section by introducing notation that will be used frequently throughout this work. 

\begin{itemize}
    \item We use $\nabla$ and $\nabla^2$ to denote the gradient and the Hessian, and write $a \cdot b = a^{\top} b$ for $a$, $b \in \R^d$ as well as $A : B = \trace(A^{\top}B)$ for $A$, $B \in \R^{d \times d}$, where $(\cdot)^{\top}$ is the transpose. For two matrices $A$, $B \in \bb{S}^d$, we write $A \geq B$ if $A - B$ is positive semidefinite. Here $\bb{S}^d$ is the space of symmetric $d \times d$ matrices. Lastly, $I_d$ is the identity matrix in $d$ dimensions.

    \item We denote by $\M(\R^d)$ the space of finite measures on $\R^d$ equipped with the bounded Lipschitz distance $d_0$. It is defined by
    \begin{equation} \label{eq:blip}
        d_0(v_1, v_2) = \sup_{\lVert \varphi\rVert_{\text{bLip}} \leq 1} \langle v_1 - v_2, \varphi\rangle \quad \text{for } v_1, v_2 \in \M(\R^d),
    \end{equation}
    where for a bounded Lipschitz continuous function $\varphi \define \R^d \to \R$ we denote by $\lVert \varphi\rVert_{\text{bLip}}$ the maximum between the Lipschitz constant of $\varphi$ and $\lVert \varphi \rVert_{\infty}$. The space $(\M(\R^d), d_0)$ is a complete separable metric space by \cite[Theorem 8.3.2]{bogachev_measure_theory_vol_2_2007} and the topology induced by $d_0$ coincides with the topology of weak convergence of measures.

    \item For a Polish space $E$ and $n \geq 1$, we let $\cal{B}_b(E; \R^n)$ denote the space of Borel measurable and bounded functions $E \to \R^n$. We set $\cal{B}_b(E) = \cal{B}_b(E; \R)$.

    \item Let $d \geq 1$ and $\eta \define \R^d \to \R$ be a measurable function. Then for $n \geq1$, $L_{\eta}^2(\R^d; \R^n)$ denotes the space of measurable functions $\varphi \define \R^d \to \R^n$ such that $\int_{\R^d} \lvert \varphi\rvert^2 e^{\eta(x)} \, \d x < \infty$. Similarly, for any integer $k$, we define the weighted Sobolev space $H_{\eta}^k(\R^d; \R^n)$ and for simplicity set $L_{\eta}^2(\R^d) = L_{\eta}^2(\R^d; \R)$ and $H_{\eta}^k(\R^d) = H_{\eta}^k(\R^d; \R)$. We denote the inner product and norm of $L_{\eta}^2(\R^d; \R^n)$ by $\langle \cdot, \cdot\rangle_{\eta}$ and $\lVert \cdot \rVert_{\eta}$, respectively. 
    
    \item We let $L^2_{\eta, \text{loc}}(\R^d; \R^n)$ be the space of measurable functions $\varphi \define \R^d \to \R^n$ such that $\varphi \bf{1}_{B_R} \in L_{\eta}^2(\R^d; \R^n)$ for any $R > 0$, where $B_R$ denotes the centred ball of radius $R$. We equip $L^2_{\eta, \text{loc}}(\R^d; \R^n)$ with the distance
    \begin{equation*}
        (\varphi_1, \varphi_2) \mapsto \sum_{k = 1}^{\infty} 2^{-k} \bigl(\lVert (\varphi_1 - \varphi_2) \bf{1}_{B_k}\rVert_{\eta} \land 1\bigr),
    \end{equation*}
    which turns $L^2_{\eta, \text{loc}}(\R^d; \R^n)$ into a complete separable metric space and a topological vector space. We set $L^2_{\eta, \text{loc}}(\R^d) = L^2_{\eta, \text{loc}}(\R^d; \R)$. Note that a sequence $(\varphi_k)_k$ converges weakly in $L^2_{\eta, \text{loc}}(\R^d; \R^n)$ to some limit $\varphi$ if and only if for all $h \in L^2_{\eta}(\R^d; \R^n)$ and all $R > 0$ it holds that $\langle \varphi_k, \bf{1}_{B_R}h\rangle_{\eta} \to \langle \varphi, \bf{1}_{B_R}h\rangle_{\eta}$. Whenever $\eta$ is the constant zero function we drop it from the subscript in the designation of the respective function space.

    \item Let $\cal{B}$ be a separable Banach space with norm $\lVert \cdot \rVert_{\cal{B}}$. We denote the space of bounded linear operators from $\cal{B}$ to itself by $B(\cal{B})$ and equip it with the operator norm. For $p \in [1, \infty]$ we let $L_{\bb{F}}^p([0, T]; \cal{B})$ denote the Banach space of $\bb{F}$-progressively measurable $\cal{B}$-valued processes $h$ such that $\ev \int_0^T \lVert h_t\rVert_{\cal{B}}^p \, \d t < \infty$ if $p \in [1, \infty)$ and $\lVert h_t\rVert_{\cal{B}}$ is $\leb \otimes \pr$-essentially bounded if $p = \infty$. Clearly, if $\cal{B}$ is a separable Hilbert space, then $L_{\bb{F}}^2([0, T]; \cal{B})$ is a Hilbert space as well. Next, we denote by $C_{\bb{F}}^2([0, T]; \cal{B})$ (respectively by $D_{\bb{F}}^2([0, T]; \cal{B})$) the space of $\bb{F}$-adapted continuous processes $h$ (respectively c\`adl\`ag processes) with values in $\cal{B}$ such that $\ev \sup_{0 \leq t \leq T} \lVert h_t\rVert_{\cal{B}}^2 < \infty$. Lastly, assuming again that $\cal{B}$ is a separable Hilbert space with inner product $\langle \cdot, \cdot\rangle_{\cal{B}}$, we let $\M_{\bb{F}}^2([0, T]; \cal{B})$ be the space of $\cal{B}$-valued $\bb{F}$-martingales $m = (m_t)_{0 \leq t \leq T}$ with a.s.\@ c\`adl\`ag trajectories for which $\ev[\langle m_T, m_T\rangle_{\cal{B}}] < \infty$. We equip $\M_{\bb{F}}^2([0, T]; \cal{B})$ with the inner product $(m^1, m^2) \mapsto \ev \langle m^1_T, m^2_T\rangle_{\cal{B}}$ which turns $\M_{\bb{F}}^2([0, T]; \cal{B})$ into a Hilbert space. 
\end{itemize}

%% file: 2_main_results/main_results.tex
In this section we state the main results of the paper. The proofs will be provided in the subsequent sections.

\subsection{The Control Problem}

Let $(\Omega, \F, \pr)$ be a complete probability space equipped with a complete filtration $\bb{F} = (\F_t)_{0 \leq t \leq T}$ and a $d_W$-dimensional $\bb{F}$-Brownian motion $W = (W_t)_{0 \leq t \leq T}$ for $d_W \geq 1$. We consider the \textit{nonlinear stochastic Fokker--Planck equation}
\begin{equation} \label{eq:sfpe}
    \d \langle \nu_t, \varphi\rangle = \bigl\langle \nu_t, \L\varphi(t, \cdot, \nu_t, \gamma_t)\bigr\rangle \, \d t + \bigl\langle \nu_t, (\nabla \varphi)^{\top} \sigma_0\bigr\rangle \, \d W_t
\end{equation}
for $\varphi \in C^2_c(\R^d)$ and $t \in (0, T]$ with initial condition $\nu_0 = v_0 \in \cal{P}(\R^d)$. The \textit{generator} $\L$ acts on $\varphi \in C^2_c(\R^d)$ by
\begin{equation} \label{eq:generator}
    \L\varphi(t, x, v, g) = \lambda(t, x, v) \varphi(x) + b(t, x, v, g) \cdot \nabla \varphi + \bigl(a(t, x) + \tfrac{1}{2}\sigma_0 \sigma_0^{\top}\bigr) \colon \nabla^2\varphi(x)
\end{equation}
for $(t, x, v, g) \in [0, T] \times \R^d \times \M(\R^d) \times G$, where $G$ is a nonempty, compact, and convex subset of $\R^{d_G}$ with $d_G \geq 1$. The \textit{coefficients} are measurable functions $\lambda \define [0, T] \times \R^d \times \M(\R^d) \to \R$, $b \define [0, T] \times \R^d \times \M(\R^d) \times G \to \R^d$, and $a \define [0, T] \times \R^d \to \R^{d \times d}$, while $\sigma_0 \in \R^{d \times d_W}$. We will make suitable assumptions on the coefficients below. The \textit{control} $\gamma \define [0, T] \times \Omega \times \R^d \to G$ is an $\bb{F}$-progressively measurable random function with values in $G$. We denote the collection of such random functions by $\bb{G}$. Next, we introduce the \textit{cost functional}
\begin{equation} \label{eq:cost_functional}
    J(\gamma) = \ev\biggl[\int_0^T \langle \nu_t, f(t, \cdot, \nu_t, \gamma_t)\rangle \, \d t + \psi(\nu_T)\biggr]
\end{equation}
and denote the \textit{value} of the control problem, i.e.\@ the infimum of $J$ over $\gamma \in \bb{G}$, by $V$. The \textit{running} and \textit{terminal cost} are measurable functions $f \define [0, T] \times \R^d \times \M(\R^d) \times G \to \R$ and $\psi \define \M(\R^d) \to \R$. 

Throughout this work, we make the following assumptions on the coefficients and cost functions.

\begin{assumption} \label{ass:fpe}
We assume there exist constants $C > 0$ and $c > 0$ such that
\begin{enumerate}[noitemsep, label = (\roman*)]
    \item \label{it:continuity_fpe} for all $t$ the map $x \mapsto a(t, x)$ has a weak derivative, for all $(t, x)$ and uniformly in $g$ the maps $v \mapsto f(t, x, v, g)$ and $v \mapsto \psi(v)$ are continuous, and for all $(t, x, v, v', g)$ we have
    \begin{equation*}
        \lvert b(t, x, v, g) - b(t, x, v', g)\rvert + \lvert \lambda(t, x, v) - \lambda(t, x, v')\rvert \leq Cd_0(v, v');
    \end{equation*}
    \item \label{it:growth_fpe} the coefficients $\lambda$, $b$, $a$, and $\partial_x a$ as well as the costs $f$ and $\psi$ are bounded;
    \item \label{it:nondeg_fpe} for all $(t, x)$ we have $a(t, x) \geq c I_d$.
\end{enumerate}
\end{assumption}

Our goal is twofold: first, we will show that for a suitably weakened version of the control problem, an optimal control is guaranteed to exist. Secondly, we derive a necessary and a sufficient stochastic maximum principle (SMP). For both purposes, it is advantageous to transform the stochastic Fokker--Planck equation \eqref{eq:sfpe} to a Fokker--Planck equation with random coefficients by applying an appropriate shift along the noise $W$. More precisely, we set
\begin{equation*}
    \tilde{\lambda}(t, x, v) = \lambda\bigl(t, x + \sigma_0 W_t, (\id + \sigma_0 W_t)^{\#}v\bigr)
\end{equation*}
for $(t, x, v) \in [0, T] \times \R^d \times \M(\R^d)$ and similarly define the shifted coefficient $\tilde{b}$ and $\tilde{a}$ as well as the shifted cost functions $\tilde{f}$ and $\tilde{\psi}$. Here $F^{\#}v$ denotes the pushforward of $v \in \M(\R^d)$ under a measurable functions $F \define \R^d \to \R^d$ and by $\id + \sigma_0 W_t$ we mean the random function $\R^d \to \R^d$ given by $x \mapsto x + \sigma_0 W_t$. No shift is applied to the control argument. Then we consider the \textit{nonlinear random Fokker--Planck equation}
\begin{equation} \label{eq:rfpe}
    \d \langle \mu_t, \varphi\rangle = \bigl\langle \mu_t, \tilde{\L}\varphi(t, \cdot, \mu_t, \gamma_t)\bigr\rangle \, \d t
\end{equation}
for $t \in (0, T]$ and $\varphi \in C^2_c(\R^d)$ with initial condition $\mu_0 = v_0$. The generator $\tilde{\L}$ is obtained by replacing $\lambda$, $b$, $a$, and $\sigma_0$ in the definition of $\L$ in \eqref{eq:generator} by $\tilde{\lambda}$, $\tilde{b}$, $\tilde{a}$, and $\tilde{\sigma}_0 = 0$. 

\begin{remark}
If we interpret $\nu_t$ as the conditional subprobability of the solution $X = (X_t)_{0 \leq t \leq T}$ of a McKean--Vlasov SDE (e.g.\@ SDE \eqref{eq:sde}) given the common information $\F_t$, then the performed shift corresponds to subtracting $\sigma_0 W_t$ from the state $X_t$. The resulting process $Y_t = X_t - \sigma_0 W_t$ solves a McKean--Vlasov SDE without common noise, but where the coefficients depend on $\sigma_0 W_t$. For example, if the drift coefficient of the original SDE was given by $b$, then the transformed coefficient is precisely provided by the shifted function $\tilde{b}$ defined above. Moreover, the conditional law of the process $Y_t$ is simply $\mu_t$.
\end{remark}

The \textit{cost functional} $\tilde{J}$ for the control of PDE \eqref{eq:rfpe} is given by
\begin{equation} \label{eq:cost_rfpe}
    \tilde{J}(\gamma) = \ev\biggl[\int_0^T \bigl\langle \mu_t, \tilde{f}(t, \cdot, \mu_t, \gamma_t)\bigr\rangle \, \d t + \tilde{\psi}(\mu_T)\biggr]
\end{equation}
for $\gamma \in \bb{G}$.

\begin{proposition} \label{prop:sfpe_to_rfpe}
Let Assumption \ref{ass:fpe} be satisfied, fix $\gamma \in \bb{G}$, and define $\tilde{\gamma} \in \bb{G}$ by $\tilde{\gamma}_t(x) = \gamma_t(x + \sigma_0 W_t)$. If $\nu = (\nu_t)_{0 \leq t \leq T}$ is a continuous $\M(\R^d)$-valued process that satisfies SPDE \eqref{eq:sfpe} with control $\gamma$, then $\mu = (\mu_t)_{0 \leq t \leq T}$ defined by $\mu_t = (\id - \sigma_0 W_t)^{\#}\nu_t$ solves PDE \eqref{eq:rfpe} with control $\tilde{\gamma}$ and $\tilde{J}(\tilde{\gamma}) = J(\gamma)$.

Conversely, if $\mu = (\mu_t)_{0 \leq t \leq T}$ is a continuous $\M(\R^d)$-valued process that satisfies PDE \eqref{eq:rfpe} with control $\tilde{\gamma}$, then $\nu = (\nu_t)_{0 \leq t \leq T}$ defined by $\nu_t = (\id + \sigma_0 W_t)^{\#}\mu_t$ solves SPDE \eqref{eq:sfpe} with control $\gamma$ and $J(\gamma) = \tilde{J}(\tilde{\gamma})$.
\end{proposition}

Proposition \ref{prop:sfpe_to_rfpe} justifies studying the control problem for SPDE \eqref{eq:sfpe} from the point of view of PDE \eqref{eq:rfpe}, and that is the approach we follow from now on. In our analysis of the control problem for PDE \eqref{eq:rfpe} we do not explicitly use that $W$ is a Brownian motion and we could in principle replace it by an arbitrary $\bb{F}$-progressively measurable stochastic process. For simplicity, we will not pursue this level of generality here.

To begin with, we will state a well-posedness result for PDE \eqref{eq:rfpe}. We would like to analyse PDE \eqref{eq:rfpe} in a standard $L^2$-framework, studying the density of $\mu_t$ as an element of $L^2(\R^d)$, as this facilitates the derivation of the SMP. However, in order to deal with the nonlinearities in the coefficients $\lambda$ and $b$ in the measure argument, which, in general, will not be continuous with respect to the $L^2$-norm of the density of $\mu_t$, we have to consider weighted $L^2$-spaces. Let us, therefore, fix the function $\eta \define \R^d \to \R$ defined by $\eta(x) = \eta_0 \sqrt{1 + \lvert x\rvert^2}$ for some $\eta_0 > 0$ and set $\bar{\eta} = -\eta$. Then the weighted space $L_{\eta}^2(\R^d; [0, \infty))$ is embedded in $\cal{M}(\R^d)$. Indeed, for any $v \in L_{\eta}^2(\R^d; [0, \infty))$ we have $\int_{\R^d} v(x) \, \d x \leq \lVert e^{-\eta/2} \rVert_{L^2} \lVert v\rVert_{\eta} < \infty$
and if $v'$ is another element of $L_{\eta}^2(\R^d; [0, \infty))$, then
\begin{equation*}
    d_0(v, v') \leq \sup_{\lVert \varphi\rVert_{\text{bLip}} \leq 1} \int_{\R^d} \lvert \varphi(x)\rvert \lvert v(x) - v'(x)\rvert \, 
    \d x \leq \lVert e^{-\eta/2} \rVert_{L^2} \lVert v - v'\rVert_{\eta}.
\end{equation*}
Here we identified $v$ and $v'$ with the respective measures they induce. This embedding makes $L_{\eta}^2(\R^d; [0, \infty))$ a convenient state space for our purpose. Note that we will be working with the specific choice for $\eta$ detailed above throughout most of this article, we consider a more general $\eta$ only in Section \ref{sec:bspde}.

\begin{proposition} \label{prop:rfpe}
Let Assumption \ref{ass:fpe} be satisfied and suppose that $v_0$ has a density $\rho_0 \in L_{\eta}^2(\R^d)$. Then for any $\gamma \in \bb{G}$, PDE \eqref{eq:rfpe} has a unique solution $\mu = (\mu_t)_{0 \leq t \leq T}$ with values in $C([0, T]; \M(\R^d))$. Moreover, $\pr$-a.s.\@ the measure $\mu_t$ has a density $\rho_t$ for all $t \in [0, T]$ such that $(\rho_t)_{0 \leq t \leq T} \in C_{\bb{F}}^2([0, T]; L_{\eta}^2(\R^d)) \cap L_{\bb{F}}^2([0, T]; H_{\eta}^1(\R^d))$ with
\begin{equation} \label{eq:rfpe_estimate}
    \sup_{0 \leq t \leq T} \lVert \rho_t \rVert_{\eta}^2 + \int_0^T \lVert \nabla \rho_t\rVert_{\eta}^2 \, \d t \leq C \lVert \rho_0\rVert_{\eta}^2
\end{equation}
for a constant $C > 0$ depending only on the coefficients $\lambda$, $b$, $a$, and $\sigma$, the parameter $\eta_0$ as well as the time horizon $T$.
\end{proposition}

From now on we will always assume that $v_0$ has a density in $L_{\eta}^2(\R^d)$ and no longer distinguish between $\mu_t$ and its density.

Before we state the existence result for optimal controls, let us make the following remark: it is possible to introduce a control problem analogous to the one discussed so far on any (complete) filtered probability space that carries a $d_W$-dimensional Brownian motion. In that case, the space of controls simply consists of the $G$-valued random functions that are progressively measurable with respect to the underlying filtration and the coefficients are shifted along the corresponding Brownian motion. If we wish to emphasise the \textit{probabilistic setup} $\bb{W} = (\Omega, \F, \bb{F}, \pr, W)$, we add it as a subscript to the symbols denoting the generator, the space of controls, and the cost functional. That is, we write $\tilde{\L}_{\bb{W}}$, $\bb{G}_{\bb{W}}$ and $\tilde{J}_{\bb{W}}$ instead of $\tilde{\L}$, $\bb{G}$ and $\tilde{J}$, respectively. 

\begin{theorem} \label{thm:optimal}
Let Assumption \ref{ass:fpe} be satisfied. If for all $(t, x, v) \in [0, T] \times \R^d \times \M(\R^d)$ the map $G \ni g \mapsto b(t, x, v, g)$ is continuous and the map $G \ni g \mapsto f(t, x, v, g)$ is upper semicontinuous, then the value $\inf_{\gamma \in \bb{G}_{\bb{W}}} \tilde{J}_{\bb{W}}(\gamma)$ does not depend on the probabilistic setup $\bb{W}$. 

Next, if for all $(t, x, v) \in [0, T] \times \R^d \times \M(\R^d)$ the map $G \ni g \mapsto b(t, x, v, g)$ is affine and the map $G \ni g \mapsto f(t, x, v, g)$ is convex and continuous, then on some probabilistic setup $\bb{W}$ there exists an optimal control.
\end{theorem}

Note that by $b$ being affine in the control argument, we mean that there exist functions $b_0 \define [0, T] \times \R^d \times \M(\R^d) \to \R^d$ and $b_1 \define [0, T] \times \R^d \times \M(\R^d) \to \R^{d_G \times d}$ such that $b(t, x, v, g) = b_0(t, x, v) + g^{\top} b_1(t, x, v)$. Theorem \ref{thm:optimal} tells us that under suitable assumptions on the coefficients and cost functions, we can always move to a probabilistic setup on which there exists at least one optimal control. 

\subsection{The Necessary Stochastic Maximum Principle}

Let us now move on to the stochastic maximum principle (SMP). In order to formulate the SMP, we have to consider an appropriate adjoint process. The adjoint will solve a \textit{backward SPDE} (BSPDE), involving the derivative of the Hamiltonian of the control problem with respect to the state, i.e.\@ the measure $\mu_t$. We shall introduce these two objects next. The \textit{Hamiltonian} $\cal{H} \define [0, T] \times \Omega \times \M(\R^d) \times \cal{B}_b(\R^d) \times \cal{B}_b(\R^d; \R^d) \times \cal{G} \to \R$ is given by
\begin{equation} \label{eq:adjoint_coefficient}
    \cal{H}(t, v, r, p, g) = \Bigl\langle v, \tilde{\lambda}(t, \cdot, v)r + \tilde{b}(t, \cdot, v, g) \cdot p + \tilde{f}(t, \cdot, v, g)\Bigr\rangle
\end{equation}
for $(t, \omega, v, r, p, g) \in [0, T] \times \Omega \times \M(\R^d) \times \cal{B}_b(\R^d) \times \cal{B}_b(\R^d; \R^d) \times \cal{G}$. Here $\cal{G}$ denotes the space of measurable functions $\R^d \to G$. Next, we define an appropriate notion of derivative for functions on the space of measures $\M(\R^d)$.

\begin{definition} \label{def:linear_functional_derivatve}
A map $F \define \M(\R^d) \to \R$ is said to have a \textit{linear functional derivative} if there exists a measurable and bounded function $DF \define \M(\R^d) \times \R^d \to \R$ such that for any $v$, $v' \in \M(\R^d)$ it holds that
    \begin{equation}
        F(v) - F(v') = \int_0^1 \int_{\R^d} DF\bigl(\theta v + (1 - \theta) v'\bigr)(x) \, \d (v - v')(x) \, \d \theta.
    \end{equation}
\end{definition}

The linear functional derivative satisfies the following type of product rule (see Lemma \ref{lem:ldf_product} for a proof): if $F \define \R^d \times \M(\R^d) \to \R$ is a measurable and bounded function such that $\M(\R^d) \ni v \mapsto F(x, v)$ has a linear functional derivative $DF(x, \cdot)$ for all $x \in \R^d$, so that $DF \define \R^d \times \M(\R^d) \times \R^d \to \R$ is measurable, bounded, and continuous in the measure argument, then the map $\M(\R^d) \ni v \mapsto \langle v, F(\cdot, v)\rangle$ possesses the linear functional derivative $F(x, v) + \langle v, DF(\cdot, v)(x)\rangle$. Consequently, assuming that the coefficients $\lambda$ and $b$ as well as the running cost $f$ have linear functional derivatives that are bounded and continuous in the measure argument, we can compute
\begin{align} \label{eq:coeff_lfd}
\begin{split}
    D\cal{H}(t, v, r, p, g)(x) &= \tilde{\lambda}(t, x, v)r(x) + \tilde{b}\bigl(t, x, v, g(x)\bigr) \cdot p(x) + \tilde{f}\bigl(t, x, v, g(x)\bigr) \\
    &\ \ \ + \Bigl\langle v, D\tilde{\lambda}(t, \cdot, v)(x) r + D\tilde{b}(t, \cdot, v, g)(x) \cdot p + D\tilde{f}(t, \cdot, v, g)(x)\Bigr\rangle.
\end{split}
\end{align}
Note that the expression above as well as $\cal{H}$ itself are also defined for $v \in L_{\eta}^2(\R^d; [0, \infty))$, $r \in L_{\bar{\eta}}(\R^d)$, and $p \in L_{\bar{\eta}}(\R^d; \R^d)$, and we shall use this extension in what follows. 

Now, let $\gamma \in \bb{G}$ and denote the corresponding solution to PDE \eqref{eq:rfpe} by $\mu$. Then the \textit{adjoint process} solves the linear BSPDE
\begin{align} \label{eq:adjoint}
\begin{split}
    \d u_t(x) = -\Bigl(D\cal{H}(t, \mu_t, u_t, \nabla u_t, \gamma_t)(x) + \tilde{a}(t, x) \colon \nabla^2 u_t(x)\Bigr) \, \d t + \d m_t(x)
\end{split}
\end{align}
for $(t, x) \in [0, T) \times \R^d$ with terminal condition $u_T(x) = D\tilde{\psi}(\mu_T)(x)$. The process $m$ is a c\`adl\`ag martingale, which takes values in the Hilbert space $L_{\bar{\eta}}^2(\R^d)$. It ensures that the process $u$ stays $\bb{F}$-adapted as we step backwards in time.

\begin{definition} \label{def:adjoint}
A \textit{solution} of BSPDE \eqref{eq:adjoint} is a pair $(u, m) \in L^2_{\bb{F}}([0, T]; H^1_{\bar{\eta}}(\R^d)) \times \M_{\bb{F}}^2([0, T]; L^2_{\bar{\eta}}(\R^d))$ such that $m$ is started from zero and
\begin{align*}
    \langle u_t, \varphi\rangle &= \bigl\langle D\tilde{\psi}(\mu_T), \varphi\bigr\rangle + \int_t^T \Bigl(\bigl\langle D\cal{H}(t, \mu_t, u_t, \nabla u_t, \gamma_t), \varphi\bigr\rangle - \bigl\langle \nabla u_s, \nabla \cdot \bigl(\tilde{a}(t, \cdot)\varphi\bigr)\bigr\rangle\Bigr) \, \d s \\
    &\ \ \ - \langle m_T - m_t, \varphi\rangle
\end{align*}
for all $\varphi \in H^1_{\eta}(\R^d)$ and $\leb \otimes \pr$-a.e.\@ $(t, \omega) \in [0, T] \times \Omega$.
\end{definition}

Note that $u$ and $m$ take values in function spaces which are weighted according to $\bar{\eta} = -\eta$ instead of $\eta$ itself. This ensures that the adjoint $u_t$ can be integrated against the state $\mu_t$. This is of crucial importance for the SMP.

Naturally, we require additional assumptions to study the existence and uniqueness of solutions to BSPDE \eqref{eq:adjoint}, in particular concerning the existence of linear functional derivatives for the coefficients $\lambda$ and $b$ as well as the cost functions $f$ and $\psi$. We shall state them together with the assumptions for the SMP.

\begin{assumption} \label{ass:smp}
We assume that
\begin{enumerate}[noitemsep, label = (\roman*)]
    \item \label{it:lfd} for all $(t, x, g)$ the maps $v \mapsto \lambda(t, x, v)$, $v \mapsto h(t, x, v, g)$ for $h \in \{b, f\}$, and $v \mapsto \psi(v)$ have measurable and bounded linear functionals derivatives $D\lambda$, $Dh$, and $D\psi$ which are continuous in the measure argument;
    \item \label{it:cost_cntrl_diff} there exists a measurable and bounded function $\nabla_g f \define [0, T] \times \R^d \times \M(\R^d) \times G \to \R^{d_G}$ such that for all $(t, x)$ the map $(v, g) \mapsto \nabla_g f(t, x, v, g)$ is continuous and for all $(t, x, v, g, g')$ we have
    \begin{equation*}
        f(t, x, v, g) - f(t, x, v, g) = \int_0^1 \nabla_g f\bigl(t, x, v, \theta g + (1 - \theta)g'\bigr) \cdot (g - g') \, \d \theta.
    \end{equation*}
    \item \label{it:convexity} for all $(t, x, v)$ the map $g \mapsto b(t, x, v, g)$ is affine and the map $g \mapsto f(t, x, v, g)$ is convex;
\end{enumerate}
\end{assumption}

The slightly awkward phrasing of the differentiability condition for $f$ in the control argument in Assumption \ref{ass:smp} \ref{it:cost_cntrl_diff} is needed because the set $G$ is not an open subset of $\R^{d_G}$.

\begin{proposition} \label{prop:adjoint}
Let Assumptions \ref{ass:fpe} and \ref{ass:smp} \ref{it:lfd} be satisfied. Then for any $\gamma \in \bb{G}$ and corresponding state $\mu$, the BSPDE \eqref{eq:adjoint} has a unique solution.
\end{proposition}

We are finally in a position to state the necessary SMP.

\begin{theorem} \label{thm:nsmp}
Let Assumptions \ref{ass:fpe} and \ref{ass:smp} be satisfied. Suppose further that $\gamma \in \bb{G}$ is an optimal control and let $\mu$ and $(u, m)$ be the corresponding solutions to PDE \eqref{eq:rfpe} and BSPDE \eqref{eq:adjoint}, respectively. Then it holds that
\begin{equation} \label{eq:nsmp}
    \tilde{b}\bigl(t, x, \mu_t, \gamma_t(x)\bigr) \cdot \nabla u_t(x) + \tilde{f}\bigl(t, x, \mu_t, \gamma_t(x)\bigr) = \inf_{g \in G} \Bigl(\tilde{b}(t, x, \mu_t, g) \cdot \nabla u_t(x) + \tilde{f}(t, x, \mu_t, g)\Bigr)
\end{equation}
for $\mu_t$-a.e.\@ $x \in \R^d$ and $\leb \otimes \pr$-a.e.\@ $(t, \omega) \in [0, T] \times \Omega$. 
\end{theorem}

Note that owing to Theorem \ref{thm:optimal}, by changing the probabilistic setup if necessary, optimal controls are guaranteed to exist.

Next, we wish to use the necessary SMP to derive a semilinear BSPDE from the linear BSPDE \eqref{eq:adjoint} satisfied by the adjoint process. Fix an optimal control $\gamma \in \bb{G}$ and let $\mu$ denote the corresponding solution to PDE \eqref{eq:rfpe}. From Equation \eqref{eq:nsmp}, we obtain the equality
\begin{equation*}
    \cal{H}(t, \mu_t, u_t, \nabla u_t, \gamma_t) = \inf_{g \in \cal{G}}\cal{H}(t, \mu_t, u_t, \nabla u_t, g),
\end{equation*}
which motivates us to introduce the function $\cal{H}_{\ast} \define [0, T] \times \Omega \times \M(\R^d) \times \cal{B}_b(\R^d) \times \cal{B}_b(\R^d; \R^d) \to \R$ given by
\begin{equation*}
    \cal{H}_{\ast}(t, v, r, p) = \inf_{g \in \cal{G}} \cal{H}(t, v, r, p, g)
\end{equation*}
for $(t, \omega, v, r, p) \in [0, T] \times \Omega \times \M(\R^d) \times \cal{B}_b(\R^d) \times \cal{B}_b(\R^d; \R^d)$. As for $\cal{H}$, the functions $\cal{H}_{\ast}$ and $D\cal{H}_{\ast}$ are also defined for $v \in L_{\eta}^2(\R^d; [0, \infty))$, $r \in L_{\bar{\eta}}^2(\R^d)$, and $p \in L^2_{\bar{\eta}}(\R^d; \R^d)$. Note that to guarantee that $D\cal{H}_{\ast}$ exists, we have to make an additional assumption (see Assumption \ref{ass:separability} below). Then, we may consider the semilinear BSPDE
\begin{equation} \label{eq:hjb}
    \d u_t(x) = -\Bigl(D\cal{H}_{\ast}(t, \mu_t, u_t, \nabla u_t)(x) + a(t, x) \colon \nabla^2 u_t(x)\Bigr) \, \d t + \d m_t(x)
\end{equation}
with terminal condition $u_T(x) = D\tilde{\psi}(\mu_T)(x)$. We understand this BSPDE in the sense of Definition \ref{def:adjoint}.

\begin{assumption} \label{ass:separability}
We assume there exist $(b_0, f_0) \define [0, T] \times \R^d \times \M(\R^d) \to \R^d \times \R$ and $(b_1, f_1) \define [0, T] \times \R^d \times G \to \R^d \times \R$ such that for all $(t, x, v, g)$ we have
\begin{equation*}
    b(t, x, v, g) = b_0(t, x, v) + b_1(t, x, g) \quad \text{and} \quad f(t, x, v, g) = f_0(t, x, v) + f_1(t, x, g).
\end{equation*}
\end{assumption}

In what follows Assumption \ref{ass:separability} will typically be in force in conjunction with Assumption \ref{ass:smp}, in which case the properties of $b$ and $f$ assumed in the latter transfer to the functions $b_0$, $b_1$, $f_0$, and $f_1$ from the former. As for the remaining coefficients and cost functions, we will indicate the shifted versions of $b_0$, $b_1$, $f_0$, and $f_1$ with a tilde. Under Assumption \ref{ass:separability}, we can write
\begin{align*}
    \cal{H}_{\ast}(t, v, r, p) &= \inf_{g \in \cal{G}} \Bigl\langle v, \tilde{\lambda}(t, \cdot, v) + \tilde{b}(t, \cdot, v, g) \cdot p + \tilde{f}(t, \cdot, v, g) \Bigr\rangle \\
    &= \cal{H}_0(t, v, r, p) + \Bigl\langle v, \inf_{g \in G} H_1(t, \cdot, p, g) \Bigr\rangle,
\end{align*}
where $\cal{H}_0 \define [0, T] \times \Omega \times \M(\R^d) \times \cal{B}_b(\R^d) \times \cal{B}_b(\R^d; \R^d) \to \R$ is given by
\begin{equation*}
    \cal{H}_0(t, v, r, p) = \Bigl\langle v, \tilde{\lambda}(t, \cdot, v)r + \tilde{b}_0(t, \cdot, v) \cdot p + \tilde{f}_0(t, \cdot, v)\Bigr\rangle
\end{equation*}
for $(t, \omega, v, r, p) \in [0, T] \times \Omega \times \M(\R^d) \times \cal{B}_b(\R^d) \times \cal{B}_b(\R^d; \R^d) \to \R$ and $H_1 \define [0, T] \times \Omega \times \R^d \times \R^d \times G \to \R$ is defined as
\begin{equation*}
    H_1(t, x, p, g) = \tilde{b}_1(t, x, g) \cdot p + \tilde{f}_1(t, x, g)
\end{equation*}
for $(t, \omega, x, p, g) \in [0, T] \times \Omega \times \R^d \times \R^d \times G$. Note that $\cal{H}_0$ and $D\cal{H}_0$ can be extended to $v \in L_{\eta}^2(\R^d; [0, \infty))$, $r \in L_{\bar{\eta}}(\R^d)$, and $p \in L_{\bar{\eta}}(\R^d; \R^d)$. Hence, it follows that $\cal{H}_{\ast}$ possesses the linear functional derivative 
\begin{align*}
    D\cal{H}_{\ast}(t, v, r, p)(x) = D\cal{H}_0(t, v, r, p)(x) + \inf_{g \in G} H_1\bigl(t, x, p(x), g\bigr).
\end{align*}

\begin{proposition} \label{prop:hjb}
Let Assumptions \ref{ass:fpe}, \ref{ass:smp} \ref{it:lfd}, and \ref{ass:separability} be satisfied. Then for any $\mu \in L_{\bb{F}}^{\infty}([0, T]; L_{\eta}(\R^d))$, BSPDE \eqref{eq:hjb} has a unique solution.
\end{proposition} 

One might expect that if $\gamma \in \bb{G}$ is an optimal control with state $\mu$, then the corresponding adjoint process $(\tilde{u}, \tilde{m})$ solves BSPDE \eqref{eq:hjb}. However, there is a slight technicality: Equation \eqref{eq:nsmp} only implies that
\begin{equation*}
    D\cal{H}_{\ast}(t, \mu_t, \tilde{u}_t, \nabla \tilde{u}_t)(x) = D\cal{H}(t, \mu_t, \tilde{u}_t, \nabla \tilde{u}_t, \gamma_t)(x)
\end{equation*}
for $\mu_t$-a.e.\@ $x \in \R^d$ and a priori it is not clear whether $\mu_t$ is equivalent to the Lebesgue measure on $\R^d$ for $\leb \otimes \pr$-a.e.\@ $(t, \omega) \in [0, T] \times \Omega$. However, it is possible to circumvent this issue by modifying $\gamma$ on null sets of $\mu_t$. Indeed, given the solution $(u, m)$ to the semilinear BSPDE \eqref{eq:hjb}, we alter $\gamma$ as follows: let $g_{\ast} \define [0, T] \times \R^d \times \R^d \to G$ denote a measurable minimiser of the map $g \mapsto H_1(t, x, p, g)$ for $(t, x, p) \in [0, T] \times \R^d \times \R^d$ (which is guaranteed to exist under Assumptions \ref{ass:fpe} and \ref{ass:smp}) and let $\tilde{g}_{\ast}$ be the shift of $g_{\ast}$ along the noise $W$. Then, we define the new control $\gamma^{\ast} \in \bb{G}$ by
\begin{equation*}
    \gamma^{\ast}_t(x) = \bf{1}_{\mu_t > 0} \gamma_t(x) + \bf{1}_{\mu_t = 0}\tilde{g}_{\ast}(t, x, \nabla u_t(x)).
\end{equation*}
Note that replacing $\gamma$ with $\gamma^{\ast}$ in PDE \eqref{eq:rfpe} does not change the solution $\mu$. In particular, $\gamma^{\ast}$ is still an optimal control. Consequently, the SMP applies with $\gamma^{\ast}$ and the corresponding adjoint process $(u^{\ast}, m^{\ast})$. Plugging the equality
\begin{equation*}
    D\cal{H}_{\ast}(t, \mu_t, u^{\ast}_t, \nabla u^{\ast}_t)(x) = D\cal{H}(t, \mu_t, u^{\ast}_t, \nabla u^{\ast}_t, \gamma^{\ast}_t)(x)
\end{equation*}
for $\leb$-a.e.\@ $x \in \R^d$, implied by \eqref{eq:nsmp} and our modification, into the linear BSPDE \eqref{eq:adjoint} satisfied by the adjoint process $(u^{\ast}, m^{\ast})$, shows that $(u^{\ast}, m^{\ast})$ in fact solves the semilinear BSPDE \eqref{eq:hjb}. Uniqueness of the latter, guaranteed by Proposition \ref{prop:hjb}, then yields $(u, m) = (u^{\ast}, m^{\ast})$. In other words, $(\mu, u, m)$ is a solution to the \textit{forward-backward SPDE} (FBSPDE) consisting of the nonlinear random Fokker--Planck equation \eqref{eq:rfpe} and the \textit{stochastic Hamilton--Jacobi--Bellman equation} \eqref{eq:hjb}:
\begin{align} \label{eq:fbspde}
\begin{split}
    \d \langle \mu_t, \varphi\rangle &= \bigl\langle \mu_t, \tilde{\L}\varphi\bigl(t, \cdot, \mu_t, \tilde{g}_{\ast}(t, \cdot, \nabla u_t)\bigr)\bigr\rangle \, \d t \\
    \d u_t(x) &= -\Bigl(D\cal{H}_{\ast}(t, \mu_t, u_t, \nabla u_t)(x) + a(t, x) \colon \nabla^2 u_t(x)\Bigr) \, \d t + \d m_t(x)
\end{split}
\end{align}
for $\varphi \in C_c^2(\R^d)$ with initial condition $\mu_0 = v_0$ and terminal condition $u_T(x) = D\tilde{\psi}(\mu_T)(x)$. 

\begin{corollary}
Let Assumptions \ref{ass:fpe}, \ref{ass:smp}, and \ref{ass:separability} be satisfied and suppose there exists an optimal control $\gamma \in \bb{G}$. Then FBSPDE \eqref{eq:fbspde} has a solution $(\mu, u, m)$.
\end{corollary}

\subsection{The Sufficient Stochastic Maximum Principle}

Next, we will formulate a sufficient SMP which together with uniqueness of the minimum on the right-hand side of Equation \eqref{eq:nsmp} implies uniqueness for FBSPDE \eqref{eq:fbspde}. 

\begin{theorem} \label{thm:ssmp}
Let Assumptions \ref{ass:fpe}, \ref{ass:smp} \ref{it:lfd}, and \ref{ass:separability} be satisfied. Fix $\gamma \in \bb{G}$ and let $\mu$ and $(u, m)$ be the corresponding solutions to PDE \eqref{eq:rfpe} and BSPDE \eqref{eq:adjoint}, respectively. Suppose $L_{\eta}^2(\R^d; [0, \infty)) \ni v \mapsto \psi(v)$ is convex and $L_{\eta}^2(\R^d; [0, \infty)) \ni v \mapsto \cal{H}_0(t, v, u_t, \nabla u_t)$ is convex for $\leb \otimes \pr$-a.e.\@ $(t, \omega) \in [0, T] \times \Omega$. Then, if
\begin{equation} \label{eq:ssmp}
    H_1\bigl(t, x, \nabla u_t(x), \gamma_t(x)\bigr) = \inf_{g \in G} H_1\bigl(t, x, \nabla u_t(x), g\bigr)
\end{equation}
for $\mu_t$-a.e.\@ $x \in \R^d$ and $\leb \otimes \pr$-a.e.\@ $(t, \omega) \in [0, T] \times \Omega$, the control $\gamma$ is optimal. 

Moreover, if the map $g \mapsto H_1(t, x, \nabla u_t(x), g)$ has a unique minimiser for $\leb$-a.e.\@ $x \in \R^d$ and $\leb \otimes \pr$-a.e.\@ $(t, \omega) \in [0, T] \times \Omega$, then $\gamma$ is the (up to modifications on null sets of $\mu_t$) unique optimal control.
\end{theorem}

In stark contrast to standard formulations of the sufficient SMP no joint convexity in the state and the control variable is required. This is due to the specific structure of the Fokker--Planck equation. Regarding the uniqueness part of Theorem \ref{thm:ssmp}, a very natural situation in which the map $g \mapsto H_1(t, x, \nabla u_t(x), g)$ has a unique minimum is if Assumption \ref{ass:smp} \ref{it:convexity} is in place with strict convexity of the map $g \mapsto f_1(t, x, g)$.

\begin{remark}
Note that if $\lambda$, $b_0$, and $f_0$ do not depend on $v$ and $v \mapsto \psi(v)$ is simply the integration of a bounded and measurable function $\R^d \to \R$ against $v$, then the convexity requirements in Theorem \ref{thm:ssmp} are automatically satisfied. Unfortunately, $\lambda$ and $b_0$ being independent of the measure argument also seems to be necessary in order for the convexity assumptions to hold. For example, suppose that $\lambda$ only depends on the measure argument. Then the map $v \mapsto \langle v, \lambda(v) r\rangle$, which is one of the summands in $\cal{H}_0$, is convex if and only if $D\lambda(v)(x) = r(x)$ for all $x \in \R^d$. Hence, the convexity condition from Theorem \ref{thm:ssmp} would imply that $D\lambda(v)(x) = u_t(x)$, which in general can of course not be satisfied. Thus, there is little hope that $v \mapsto \cal{H}_0(t, v, u_t, \nabla u_t)$ is convex if $\lambda$ or $b_0$ depend on $v$. Consequently, Theorem \ref{thm:ssmp} should be understood as a sufficient optimality condition for the linear Fokker--Planck equation \eqref{eq:rfpe} with possibly nonlinear cost functions.
\end{remark}

If $(\mu, u, m)$ is a solution to the FBSPDE \eqref{eq:fbspde} and the assumptions of Theorem \ref{thm:ssmp} are satisfied, then the sufficient SMP implies that the control $\gamma$ defined by $\gamma_t(x) = \tilde{g}_{\ast}(t, x, \nabla u_t(x))$ is optimal. Moreover, under the additional assumption that the minimum of $g \mapsto H_1(t, x, \nabla u_t(x), g)$ is unique, if $(\mu', q', m')$ is another solution to FBSPDE \eqref{eq:fbspde} and we define $\gamma' \in \bb{G}$ by $\gamma'_t(x) = \tilde{g}_{\ast}(t, x, \nabla u'_t(x))$, uniqueness of optimal controls implies that $\gamma' = \gamma$. Consequently, it follows that $\mu' = \mu$ by Proposition \ref{prop:rfpe} and, finally, by Proposition \ref{prop:hjb}, we get $(u', m') = (u, m)$. We have proved the following corollary.

\begin{corollary}
Let Assumptions \ref{ass:fpe}, \ref{ass:smp}, and \ref{ass:separability} be satisfied and let $(\mu, u, m)$ be a solution to FBSPDE \eqref{eq:fbspde}. If the convexity assumptions from Theorem \ref{thm:ssmp} hold for $(\mu, u, m)$ and the map $g \mapsto H_1(t, x, \nabla u_t(x), g)$ has a unique minimum for $\mu_t$-a.e.\@ $x \in \R^d$ and $\leb \otimes \pr$-a.e.\@ $(t, \omega) \in [0, T] \times \Omega$, then $(\mu, u, m)$ is the unique solution of FBSPDE \eqref{eq:fbspde}.
\end{corollary}

\subsection{Practical Implementation of the Criticality Condition}

Before we conclude this section, let us briefly comment on the practical implementation of the criticality condition of the necessary SMP. This condition gives rise to the forward-backward system \eqref{eq:fbspde}. In general, this system may have multiple solutions $(\mu, u, m)$, not all of which necessarily determine an optimal control $\gamma_t$ via
\begin{equation*}
    \gamma_t(x) = \tilde{g}_{\ast}(t, x, \nabla u_t(x)).
\end{equation*}
Optimality of $\gamma_t$ can be guaranteed, however, under the additional convexity assumptions from the sufficient SMP, Theorem 2.14. In other words, solving FBSPDE \eqref{eq:fbspde} provides a candidate control, which is optimal under the conditions of Theorem 2.14. In practice, FBSPDE \eqref{eq:fbspde} will not be analytically solvable, except perhaps in certain special cases, so numerical methods must be deployed. One possibility is to discretise both the Fokker--Planck equation and the stochastic Hamilton--Jacobi--Bellman equation from \eqref{eq:fbspde} by a finite-difference or Galerkin scheme. The resulting high-dimensional forward-backward SDE can be approximated, for example, by using the deep BSDE method of Jiequn, Jentzen \& Weinan \cite{jiequn_deep_bsde_2018}. A numerical implementation of this procedure is beyond the scope of this paper.

%% file: 3_bspde/hilbert_space_martingales.tex
In this section, we will study a more general version of the BSPDEs \eqref{eq:adjoint} and \eqref{eq:hjb} appearing in Section \ref{sec:main_results}. In preparation, we will introduce an extension of It\^o's formula to $L^2$-valued semimartingales that will serve us throughout this paper.

\subsection{C\`adl\`ag Martingales with Values in Hilbert Spaces} \label{sec:hilbert_space_martingale}

We recall a few notions from the theory of martingales with values in infinite-dimensional Hilbert spaces. For more information we refer to \cite[Section 26]{metivier_semimartingales_1982}. We fix a separable Hilbert space $\bb{H}$ with inner product $\langle \cdot, \cdot \rangle_{\bb{H}}$ and induced norm $\lVert \cdot \rVert_{\bb{H}}$. Recall that $\cal{M}_{\bb{F}}^2([0, T]; \bb{H})$ denotes the Hilbert space of $\bb{H}$-valued $\bb{F}$-martingales $m$ with a.s.\@ c\`adl\`ag trajectories for which $\ev \lVert m_T\rVert_{\bb{H}}^2 < \infty$. We denote the \textit{quadratic variation} of a martingale $m \in \cal{M}_{\bb{F}}^2([0, T]; \bb{H})$ by $[m]$ and its \textit{tensor quadratic variation} by $[[m]]$. The tensor quadratic variation is a process with values in the trace class operators from $\bb{H}$ to itself and is symmetric as such. Moreover, it holds that $[m]_t = \trace [[m]]_t$. We can define the \textit{quadratic covariation} $[m^1, m^2]$ and \textit{tensor quadratic covariation} $[[m^1, m^2]]$ of two martingales $m^1$, $m^2 \in \cal{M}_{\bb{F}}^2([0, T]; \bb{H})$ through polarisation. 

We say that two martingales $m^1 \in \M_{\bb{F}}^2([0, T]; \bb{H}_1)$, $m^2 \in \M_{\bb{F}}^2([0, T]; \bb{H}_2)$ with values in separable Hilbert spaces $\bb{H}_1$ and $\bb{H}_2$, respectively, are \textit{very strongly orthogonal} if for all elements $h_1 \in \bb{H}_1$, $h_2 \in \bb{H}_2$ the real-valued c\`adl\`ag martingales $\langle m^1, h_1\rangle_{\bb{H}_1}$ and $\langle m^2, h_2\rangle_{\bb{H}_2}$ are strongly orthogonal, i.e.\@ $\bigl[\langle m^1, h_1\rangle_{\bb{H}_1}, \langle m^2, h_2\rangle_{\bb{H}_2}\bigr]_t = 0$ for all $t \in [0, T]$. If $\bb{H}_1 = \bb{H}_2$, this is equivalent to saying that $[[m^1, m^2]] = 0$.

For $m \in \cal{M}_{\bb{F}}^2([0, T]; \bb{H})$, let $\P_m^2$ denote the space of $\bb{F}$-predictable $\bb{H}$-valued processes such that $\int_0^T \lVert h_t\rVert_{\bb{H}}^2 \, \d [m]_t < \infty$ almost surely. For $h \in \P_m^2$ we let $\int_0^t \langle h_s, \d m_s\rangle_{\bb{H}}$ be the stochastic integral of $h$ against $m$ over the interval $(0, t]$, i.e.\@ we exclude the left endpoint. The indefinite stochastic integral $\int_0^{\cdot} \langle h_s, \d m_s\rangle_{\bb{H}}$ is a real-valued local martingale and its quadratic variation is given by $\int_0^{\cdot} \langle h_s, h_s \cdot \d [[m]]_s\rangle_{\bb{H}}$. Here again the integral excludes the left endpoint, and for a bounded operator $A \define \bb{H} \to \bb{H}$ and an element $h \in \bb{H}$ we write $h \cdot A$ for $Ah$. 

Since the operator norm of a trace class operator $A$ is bounded from above by its trace, it holds that $\langle h, h \cdot A\rangle_{\bb{H}} \leq \lVert h\rVert_{\bb{H}}^2 \trace A$. We can apply this inequality to simple integrands $h \in \P_m^2$ of the form $h = \sum_{i = 0}^{n - 1} h_i \bf{1}_{(t_i, t_{i + 1}]}$ for $h_i$ bounded and $\bb{F}_{t_i}$-measurable and $0 = t_0 < t_1 < \dots < t_n = T$ to see $\int_0^t \langle h_s, h_s \cdot \d [[m]]_s\rangle_{\bb{H}} \leq \int_0^t \lVert h_s\rVert_{\bb{H}}^2 \, \d [m]_s$. Next, we extend this inequality to arbitrary $h \in \P_m^2$ by density of the simple integrands. This implies the following Burkholder--Davis--Gundy type inequality.

\begin{proposition} \label{prop:bdg}
Let $m \in \cal{M}_{\bb{F}}^2([0, T]; \bb{H})$ and $h \in \cal{P}^2_m$. Then for all $p \geq 1$ and any stopping time $\tau$ with values in $[0, T]$ we have
\begin{equation}
    \ev \sup_{0 \leq t \leq \tau} \biggl\lvert\int_0^t \langle h_s, \d m_s\rangle_{\bb{H}}\biggr\rvert^p \leq C_p \ev\biggl(\int_0^{\tau} \lVert h_t\rVert_{\bb{H}}^2 \, \d [m]_t\biggr)^{p/2}
\end{equation}
for a constant $C_p > 0$ independent of $m$, $h$, and $\tau$.
\end{proposition}

\begin{proof}
By the classical Burkholder--Davis--Gundy inequality we have
\begin{equation*}
    \ev \sup_{0 \leq t \leq \tau} \biggl\lvert\int_0^t \langle h_s, \d m_s\rangle_{\bb{H}}\biggr\rvert^p \leq C_p \ev\biggl(\int_0^{\tau} \langle h_t, h_t \cdot \d [[m]]_t\rangle_{\bb{H}}\biggr)^{p/2}.
\end{equation*}
Then the bound $\int_0^t \langle h_s, h_s \cdot \d [[m]]_s\rangle_{\bb{H}} \leq \int_0^t \lVert h_s\rVert_{\bb{H}}^2 \, \d [m]_s$ allows us to conclude the proof.
\end{proof}

%% file: 3_bspde/ito_formula.tex
\subsection{It\^o's Formula for C\`adl\`ag Hilbert Space-Valued Semimartingales}

Our next goal is to derive a generalisation of It\^o's formula which is suitable for our analysis of c\`adl\`ag semimartingales with values in Hilbert spaces.

\begin{theorem} \label{thm:ito}
For $i = 1$,~\ldots, $k$, let $u^i \in L_{\bb{F}}^2([0, T]; H^1(\R^d))$, $\zeta^i \in L_{\bb{F}}^2([0, T]; H^{-1}(\R^d))$, and $m^i \in \cal{M}_{\bb{F}}^2([0, T]; L^2(\R^d))$ such that
\begin{equation*}
    \langle u^i_t, \varphi\rangle = \int_0^t \zeta^i_s(\varphi) \, \d s + \langle m^i_t, \varphi\rangle
\end{equation*}
for all $\varphi \in H^1(\R^d)$. Then $u^1$,~\ldots, $u^k$ lie in $D_{\bb{F}}^2([0, T]; L^2(\R^d))$ (or $C_{\bb{F}}^2([0, T]; L^2(\R))$ if $m^1$,~\ldots, $m^k$ have continuous trajectories).

Next, let $F \define L^2(\R^d; \R^k) \to \R$ be a twice continuously Fr\'echet differentiable function and assume that for all $1 \leq i, j \leq k$ it holds that
\begin{enumerate}[noitemsep, label = \textup{(\roman*)}]
    \item \label{it:bounded_on_h_1} for each $h \in H^1(\R^d; \R^k)$ the derivative $D_i F(h)$ is an element of $H^1(\R^d)$ and the map $H^1(\R^d; \R^k) \to H^1(\R^d)$, $h \mapsto D_i F(h)$ is uniformly continuous on bounded subsets of $H^1(\R^d; \R^k)$;
    \item \label{it:hs_bounded} for each $h \in L^2(\R^d; \R^k)$ the second derivative $D_{ij}^2 F(h)$ is an element of $B(L^2(\R^d))$ and the map $L^2(\R^d; \R^k) \to B(L^2(\R^d))$, $h \mapsto D_{ij}^2 F(h)$ is uniformly continuous on bounded subsets of $L^2(\R^d; \R^k)$.
\end{enumerate}
Then for any $t \in [0, T]$ it holds that
\begin{align} \label{eq:ito_formula}
\begin{split}
    F(u_t) &= F(m_0) + \sum_{i = 1}^k \biggl(\int_0^t \zeta^i_s(D_i F(u_s)) \, \d s + \int_0^t \langle D_i F(u_{s-}), \, \d m^i_s\rangle \biggr) \\
    &\ \ \ + \sum_{i, j = 1}^k \frac{1}{2}\int_0^t \trace\bigl(D^2_{ij}F(u_{s-}) 
    \, \d [[m^i, m^j]]^c_s\bigr) \\
    &\ \ \ + \sum_{0 < s \leq t} \biggl(F(u_{s-} + \Delta m_s) - F(u_{s-}) - \sum_{i = 1}^k \langle D_i F(u_{s-}), \Delta m^i_s\rangle\biggr).
\end{split}
\end{align}
Here $\Delta m_s$ denotes the jump of $m$ at time $s$.
\end{theorem}

\begin{proof}
We proceed in two steps.

\textit{Step 1}: First, we assume that $\zeta^i$ is an element of $L_{\bb{F}}^2([0, T]; L^2(\R^d))$ for $i = 1$,~\ldots, $k$, so that $u^i_t = \int_0^t \zeta^i_s \, \d s + m^i_t$, where $\int_0^t \zeta^i_s \, \d s$ is understood as a Bochner integral. Then $u^1$,~\ldots, $u^k$ are $L^2(\R^d)$-valued c\`adl\`ag semimartingales, so we simply apply \cite[Theorem 27.2]{metivier_semimartingales_1982}, which gives precisely \eqref{eq:ito_formula}. Note that \cite[Theorem 27.2]{metivier_semimartingales_1982} is formulated for the case that the second derivative $D^2_{ij}F$ is a uniformly continuous map from $L^2(\R^d; \R^k)$ to the space of Hilbert--Schmidt operators on $L^2(\R^d)$ equipped with the Hilbert--Schmidt norm. However, since the tensor quadratic variation $[[m^i, m^j]]$ is a finite variation process with values in the trace class operators, the theorem can easily be extended to the situation where $D^2_{ij}F$ only satisfies Assumption \ref{it:hs_bounded} above.

\textit{Step 2}: Now, we let $\zeta^i \in L_{\bb{F}}^2([0, T]; H^{-1}(\R^d))$. We will reduce to the previous case by mollification. Let $\rho \define \R^d \to \R$ be the standard mollifier defined by $\rho(x) = c_d e^{-1/(1 - \lvert x\rvert^2)}$ if $x$ is in the interior of the centred ball of radius $1$ in $\R^d$ and zero otherwise, and set $\rho_{\epsilon}(x) = \frac{1}{\epsilon^d} \rho(x/\epsilon)$. Here $c_d > 0$ is a normalisation constant which ensures that $\int_{\R^d} \rho(x) \, \d x = 1$. For a function $\varphi \in L^2(\R^d)$ we let $(\rho_{\epsilon} \ast \varphi)(x) = \int_{\R^d} \rho_{\epsilon}(x - y) \varphi(x) \, \d y$. Then we set $u^{i, \epsilon}_t = \rho_{\epsilon} \ast u^i_t$, and $m^{i, \epsilon}_t = \rho_{\epsilon} \ast m^i_t$, and define $\zeta^{i, \epsilon}$ via $\zeta^{i, \epsilon}_t(\varphi) = \zeta^i_t(\rho_{\epsilon} \ast \varphi)$. With this we obtain
\begin{equation} \label{eq:mollified_differential}
    \d \langle u^{i, \epsilon}_t, \varphi\rangle = \d \langle u^i_t, \rho_{\epsilon} \ast \varphi\rangle = \zeta^i_t(\rho_{\epsilon} \ast \varphi) \, \d t + \d \langle m^i_t, \rho_{\epsilon} \ast \varphi\rangle = \zeta^{i, \epsilon}_t(\varphi) \, \d t + \d \langle m^{i, \epsilon}_t, \varphi\rangle
\end{equation}
for $\varphi \in L^2(\R^d)$. Note that for any $\xi \in H^{-1}(\R^d)$ it holds that $\lvert \xi(\rho_{\epsilon} \ast \varphi) \rvert \leq \lVert \xi\rVert_{H^{-1}} \lVert \rho_{\epsilon} \ast \varphi \rVert_{H^1} \leq C_{\epsilon, d} \lVert \rho_{\epsilon}\rVert_{W^{1, \infty}} \lVert \xi\rVert_{H^{-1}} \lVert \varphi\rVert_{L^2}$ for some constant $C_{\epsilon, d} > 0$. Hence, the mollification of $\xi$ is a bounded linear functional on $L^2(\R^d)$, so we may identify it with an element of $L^2(\R^d)$. Consequently, we can view $\zeta^{i, \epsilon}$ as a process in $L_{\bb{F}}^2([0, T]; L^2(\R^d))$. Thus, by Step 1 it holds that
\begin{align} \label{eq:ito_formula_epsilon}
\begin{split}
    F(u^{\epsilon}_t) &= F(m^{\epsilon}_0) + \sum_{i = 1}^k \biggl(\int_0^t \zeta^{i, \epsilon}_s(D_i F(u^{\epsilon}_s)) \, \d s + \int_0^t \langle D_i F(u^{\epsilon}_{s-}), \, \d m^{i, \epsilon}_s\rangle \biggr) \\
    &\ \ \ + \sum_{i, j = 1}^k \frac{1}{2}\int_0^t \trace\bigl(D^2_{ij}F(u^{\epsilon}_{s-}) 
    \, \d [[m^{i, \epsilon}, m^{j, \epsilon}]]^c_s\bigr) \\
    &\ \ \ + \sum_{0 < s \leq t} \biggl(F(u^{\epsilon}_{s-} + \Delta m^{\epsilon}_s) - F(u^{\epsilon}_{s-}) - \sum_{i = 1}^k \langle D_i F(u^{\epsilon}_{s-}), \Delta m^{i, \epsilon}_s\rangle\biggr).
\end{split}
\end{align}
Our goal is to pass to the limit $\epsilon \to 0$ in Equation \eqref{eq:ito_formula_epsilon}. However, while the sequences $(u^{i, \epsilon})_{\epsilon > 0}$, $(\zeta^{i, \epsilon})_{\epsilon > 0}$, and $(m^{i, \epsilon})_{\epsilon > 0}$, $i = 1$,~\ldots, $k$, converge in their respective Hilbert spaces $L_{\bb{F}}^2([0, T]; H^1(\R^d))$, $L_{\bb{F}}^2([0, T]; H^{-1}(\R^d))$, and $\cal{M}_{\bb{F}}^2([0, T]; L^2(\R^d))$, we do not yet know whether $(u^{i, \epsilon})_{\epsilon > 0}$ also converges as a sequence in $D_{\bb{F}}^2([0, T]; L^2(\R^d))$. But this is certainly necessary to pass to the limit in the stochastic integral term $\int_0^t \langle D_i F(u^{\epsilon}_{s-}), \, \d m^{i, \epsilon}_s\rangle$ and to make sense of its limit $\int_0^t \langle D_iF(u_{s-}), \d m^i_s\rangle$. 

To deduce the convergence in $D_{\bb{F}}^2([0, T]; L^2(\R^d))$, we first apply \eqref{eq:ito_formula_epsilon} to $u^{i, \epsilon, \delta} = u^{i, \epsilon} - u^{i, \delta}$ for the specific choice $F(\varphi) = \sum_{i = 1}^k\lVert \varphi^i\Vert_{L^2}^2$. This implies that for any $\epsilon$, $\delta > 0$ we have
\begin{equation} \label{eq:l_2_app_ito}
    \sum_{i = 1}^k \lVert u^{i, \epsilon, \delta}_t \Vert_{L^2}^2 = \sum_{i = 1}^k \biggl(\lVert m^{i, \epsilon, \delta}_0 \Vert_{L^2}^2 + \int_0^t 2\zeta^{i, \epsilon, \delta}_s(u^{i, \epsilon, \delta}_s) \, \d s + \int_0^t 2\bigl\langle u^{i, \epsilon, \delta}_{s-}, \, \d m^{i, \epsilon, \delta}_s\bigr\rangle + [m^{i, \epsilon, \delta}]_t\biggr),
\end{equation}
where $\zeta^{i, \epsilon, \delta}$ and $m^{i, \epsilon, \delta}$ are defined in the obvious way. We take the supremum over $t \in [0, T]$ and then expectation on both sides. Further, we estimate with Proposition \ref{prop:bdg},
\begin{align*}
    \ev \sup_{t \in [0, T]} 2\biggl\lvert \int_0^t \bigl\langle u^{i, \epsilon, \delta}_{s-}, \, \d m^{i, \epsilon, \delta}_s\bigr\rangle\biggr\rvert &\leq 2C_1 \ev\biggl(\int_0^T \lVert u^{i, \epsilon, \delta}_{t-}\rVert_{L^2}^2 \, \d [m^{i, \epsilon, \delta}]_t\biggr)^{1/2} \\
    &\leq 2 C_1 \ev\biggl(\sup_{0 < t \leq T} \lVert u^{i, \epsilon, \delta}_{t-}\rVert_{L^2}^2 [m^{i, \epsilon, \delta}]_T\biggr)^{1/2} \\
    &\leq \frac{1}{2} \ev \sup_{0 \leq t \leq T} \lVert u^{i, \epsilon, \delta}_t\rVert_{L^2}^2 + 8 C_1^2 \ev [m^{i, \epsilon, \delta}]_T.
\end{align*}
Inserting this into \eqref{eq:l_2_app_ito} and rearranging gives
\begin{align*}
    \sum_{i = 1}^k \frac{1}{2} \ev \sup_{0 \leq t \leq T} \lVert u^{i, \epsilon, \delta}_t\rVert_{L^2}^2 \leq \sum_{i = 1}^k \biggl(&\lVert m^{i, \epsilon, \delta}_0 \Vert_{L^2}^2 + \ev\int_0^T 2 \lVert \zeta^{i, \epsilon, \delta}_t\rVert_{H^{-1}} \lVert u^{i, \epsilon, \delta}_t\rVert_{H^1} \, \d t \\
    &+ (2C_1 + 1) \ev[m^{i, \epsilon, \delta}]_T\biggr).
\end{align*}
The right-hand side clearly tends to zero as $\epsilon$, $\delta \to 0$, which means that $(u^{i, \epsilon})_{\epsilon > 0}$ is a Cauchy sequence in the space $D_{\bb{F}}^2([0, T]; L^2(\R^d))$ for $i = 1$,~\ldots, $k$. Hence, it has a unique limit, which must coincide with $u^i$, since $u^{i, \epsilon} \to u^i$ in $L_{\bb{F}}^2([0, T]; L^2(\R^d))$. Consequently, we have $\ev \sup_{0 \leq t \leq T} \lVert u^{i, \epsilon}_t - u^i_t\rVert_{L^2}^2 \to 0$ and $u^i$ lies in $D_{\bb{F}}^2([0, T]; L^2(\R^d))$.

Let us return to Equation \eqref{eq:ito_formula_epsilon}. We can clearly pass to the limit in probability in the terms $F(u^{\epsilon}_t)$, $F(m^{\epsilon}_0)$, and $\int_0^t \zeta^{i, \epsilon}_s(D_i F(u^{\epsilon}_s)) \, \d s$ resulting in the expressions $F(u_t)$, $F(m_0)$, and $\int_0^t \zeta^i_s(D_iF(u_s)) \, \d s$. Here we make use of Assumption \ref{it:bounded_on_h_1}. Next, by Lemma \ref{lem:stoch_int_conv}, the stochastic integrals and the quadratic variation terms converge to $\int_0^t \langle D_i F(u_{s-}), \d m^i_s\rangle$ and $\frac{1}{2}\int_0^t \trace\bigl(D_{ij}^2 F(u_{s-}) \, \d [[m^i, m^j]]^c_s\bigr)$, respectively, in probability as $\epsilon \to 0$. This leaves the jump terms. First note that the jumps of $m^{\epsilon}$ and $m$ align perfectly and, moreover, it holds that $\lVert \Delta m^{\epsilon}_s\rVert_{L^2} \leq \lVert \Delta m_s\rVert_{L^2}$. We can use Taylor's theorem to bound $F(u^{\epsilon}_{s-} + \Delta m^{\epsilon}_s) - F(u^{\epsilon}_{s-}) - \sum_{i = 1}^k \langle D_i F(u^{\epsilon}_{s-}), \Delta m^{i, \epsilon}_s\rangle$ by
\begin{align*}
    \sum_{i, j = 1}^k \bigl\langle \Delta m^{i, \epsilon}_s, D_{ij}^2 F(u^{\epsilon}_{s-}) \Delta m^{j, \epsilon}_s\bigr\rangle &\leq \sum_{i, j = 1}^k \frac{\lVert D_{ij}^2F(u^{\epsilon}_{s-})\rVert}{2} \bigl(\lVert \Delta m^{i, \epsilon}_s\rVert_{L^2}^2 + \lVert \Delta m^{j, \epsilon}_s\rVert_{L^2}^2\bigr) \\
    &\leq \sum_{i, j = 1}^k \frac{\lVert D_{ij}^2F(u^{\epsilon}_{s-})\rVert}{2} \bigl(\lVert \Delta m^i_s\rVert_{L^2}^2 + \lVert \Delta m^j_s\rVert_{L^2}^2\bigr),
\end{align*}
where $\lVert \cdot \rVert$ denotes the operator norm for elements of $B(L^2(\R^d))$. From above we know that $\ev \sup_{0 < s \leq t} \lVert u^{\epsilon}_{s-} - u_{s-}\rVert_{L^2}^2 \to 0$, so since $D_{ij}^2 F$ is bounded on bounded sets of $L^2(\R^d)$ with respect to the operator norm by Assumption \ref{it:hs_bounded}, we deduce that $\lVert D_{ij}^2F(u^{\epsilon}_{s-})\rVert$ is bounded by some random constant $C > 0$ uniformly in $\epsilon$ and $s \in [0, T]$. Hence, we can upper bound $F(u^{\epsilon}_{s-} + \Delta m^{\epsilon}_s) - F(u^{\epsilon}_{s-}) - \sum_{i = 1}^k \langle D_i F(u^{\epsilon}_{s-}), \Delta m^{i, \epsilon}_s\rangle$ by the a.s.\@ summable expression $Ck \sum_{i = 1}^k \lVert \Delta m^i_s\rVert_{L^2}^2$. Thus, we can apply the dominated convergence theorem to pass to the a.s.\@ limit in the sum in the third line of Equation \eqref{eq:ito_formula_epsilon}. Putting all convergences together yields the desired formula \eqref{eq:ito_formula}.
\end{proof}

We have the following corollary of Theorem \ref{thm:ito} for weighted function spaces, where the exponent of the weight $\eta \in C^1(\R^d)$ has uniformly bounded derivative. Note that $\eta$ is not necessarily equal to the exponent from Section \ref{sec:main_results}.

\begin{corollary} \label{cor:ito}
Let $u^{\pm} \in L_{\bb{F}}^2([0, T]; H_{\pm \eta}^1(\R^d))$, $\zeta^{\pm} \in L_{\bb{F}}^2([0, T]; (H_{\mp \eta}^1(\R^d))^{\ast})$, and $m^{\pm} \in \cal{M}_{\bb{F}}^2([0, T]; L^2_{\pm \eta}(\R^d))$ such that
\begin{equation*}
    \langle u^{\pm}_t, \varphi\rangle = \int_0^t \zeta^{\pm}_s(\varphi) \, \d s + \langle m^{\pm}_t, \varphi\rangle
\end{equation*}
for all $\varphi \in H^1_{\mp \eta}(\R^d)$. Then $u^{\pm}$ lies in $D_{\bb{F}}^2([0, T]; L^2_{\pm \eta}(\R^d))$ (or $C_{\bb{F}}^2([0, T]; L^2_{\pm \eta}(\R))$ if $m^{\pm}$ has continuous trajectories) and for any $t \in [0, T]$ it holds that
\begin{equation} \label{eq:ito_formula_2}
    \langle u^+_t, u^-_t\rangle = \int_0^t \zeta^+_s(u^-_s) + \zeta^-_s(u^+_s) \, \d s + \int_0^t \bigl(\langle u^-_s, \d m^+_s\rangle + \langle u^+_s, \d m^-_s\rangle\bigr) + [m^+, m^-]_t.
\end{equation}
\end{corollary}

\begin{proof}
To see that $u^{\pm}$ lies in $D_{\bb{F}}^2([0, T]; L^2_{\pm \eta}(\R^d))$ (or $C_{\bb{F}}^2([0, T]; L^2_{\pm \eta}(\R))$ if $m^{\pm}$ has continuous trajectories), we simply apply Theorem \ref{thm:ito} to $\tilde{u}^{\pm} \in L_{\bb{F}}^2([0, T]; H^1(\R^d))$ defined by $\tilde{u}^{\pm}_t(x) = e^{\pm \eta(x)/2} u^{\pm}_t(x)$, which satisfies
\begin{equation*}
    \langle \tilde{u}^{\pm}_t, \varphi\rangle = \int_0^t \tilde{\zeta}^{\pm}_s(\varphi) \, \d s + \langle \tilde{m}^{\pm}_t, \varphi\rangle
\end{equation*}
for $\varphi \in H^1(\R^d)$, where $\tilde{\zeta}^{\pm} \in L_{\bb{F}}^2([0, T]; H^{-1}(\R^d))$ and $\tilde{m}^{\pm} \in \cal{M}_{\bb{F}}^2([0, T]; L^2(\R^d))$ are given by $\tilde{\zeta}^{\pm}_t(\varphi) = \zeta^{\pm}(e^{\pm \eta/2}\varphi)$ for $\varphi \in H^1(\R^d)$ and $\tilde{m}^{\pm}_t(x) = e^{\pm \eta(x)/2} m^{\pm}_t(x)$, respectively. Next, we apply It\^o's formula from \eqref{eq:ito_formula} to $\langle \tilde{u}^+_t, \tilde{u}^-_t\rangle$, which implies that
\begin{align*}
    \langle u^+_t, u^-_t\rangle &= \langle \tilde{u}^+_t, \tilde{u}^-_t\rangle \\
    &= \int_0^t \tilde{\zeta}^+_s(\tilde{u}^-_s) + \tilde{\zeta}^-_s(\tilde{u}^+_s) \, \d s + \int_0^t \bigl(\langle \tilde{u}^-_s, \d \tilde{m}^+_s\rangle + \langle \tilde{u}^+_s, \d \tilde{m}^-_s\rangle\bigr) + [\tilde{m}^+, \tilde{m}^-]_t \\
    &= \int_0^t \zeta^+_s(u^-_s) + \zeta^-_s(u^+_s) \, \d s + \int_0^t \bigl(\langle u^-_s, \d m^+_s\rangle + \langle u^+_s, \d m^-_s\rangle\bigr) + [m^+, m^-]_t
\end{align*}
as required.
\end{proof}

%% file: 3_bspde/bspde.tex
\subsection{Semilinear BSPDEs}

In this subsection we study semilinear BSPDEs of a slightly more general type than BSPDEs \eqref{eq:adjoint} and \eqref{eq:hjb} introduced in Section \ref{sec:main_results}. The existence and uniqueness of the latter BSPDEs (cf.\@ Propositions \ref{prop:adjoint} and \ref{prop:hjb}) then follows as a special case. This is further discussed at the end of this subsection. The BSPDE we consider here takes the form
\begin{align} \label{eq:bspde}
\begin{split}
    \d u_t(x) = &-\Bigl(F_t(x, u_t, \nabla u_t, q_t) + \alpha_t(x) \colon \nabla^2 u_t(x) + \beta_t(x) \colon \nabla q_t(x)\Bigr) \, \d t \\
    &+ q_t(x) \cdot \d W_t + \d m_t(x)
\end{split}
\end{align}
for $(t, x) \in [0, T) \times \R^d$ with terminal condition $u_T(x) = \Psi(x)$. Here $m$ is a c\`adl\`ag $\bb{F}$-martingale with values in a suitable weighted function space and together with $q$ ensures that the solution $(u, q, m)$ is $\bb{F}$-progressively measurable. The coefficients $\alpha$ and $\beta$ are $\bb{F}$-progressively measurable random functions on $[0, T] \times \Omega \times \R^d$ with values in $\bb{S}^d$ and $\R^{d \times d_W}$, respectively. The map $F$ is an $\bb{F}$-progressively measurable random function $[0, T] \times \Omega \times \R^d \times L^2_{\eta}(\R^d) \times L^2_{\eta}(\R^d; \R^d) \times L^2_{\eta}(\R^d; \R^{d_W}) \to \R$, where $\eta \in C^1(\R^d)$ has a uniformly bounded derivative. Again we do not assume that $\eta$ is the exponent from Section \ref{sec:main_results}. We define the function space
\begin{equation*}
    \cal{S}_{\eta} = L^2_{\bb{F}}([0, T]; H^1_{\eta}(\R^d)) \times L^2_{\bb{F}}([0, T]; L^2_{\eta}(\R^d; \R^{d_W})) \times \M_{\bb{F}}^2([0, T]; L^2_{\eta}(\R^d)).
\end{equation*}

\begin{definition} \label{def:bspde}
A \textit{solution} of BSPDE \eqref{eq:bspde} consists of a triple $(u, q, m) \in \cal{S}_{\eta}$ such that $m$ is started from zero and is very strongly orthogonal to $W$, and
\begin{align*}
    \langle u_t, \varphi\rangle = \langle \Psi, \varphi\rangle &+ \int_t^T \Bigl(\bigl\langle F_s(\cdot, u_s, \nabla u_s, q_s), \varphi\bigr\rangle - \langle \nabla u_s, \nabla \cdot (\alpha_s\varphi)\rangle - \bigl\langle q_s, \nabla \cdot (\beta_s \varphi)\bigr\rangle\Bigr) \, \d s \\
    &- \int_t^T \langle q_s, \varphi\rangle \cdot \d W_s - \langle m_T - m_t, \varphi\rangle
\end{align*}
for all $\varphi \in H^1_{\bar{\eta}}(\R^d)$ and $\leb \otimes \pr$-a.e.\@ $(t, \omega) \in [0, T] \times \Omega$.
\end{definition}

Note that Definition \ref{def:bspde} implicitly assumes that the coefficients $\alpha$ and $\beta$ are differentiable in $x$. We shall always assume that this is the case.

\begin{assumption} \label{ass:bspde}
We assume that $\Psi \in L^2(\Omega, \F_T, \pr; L^2(\R^d))$ and that there exist constants $c > 0$ and $C > 0$ such that
\begin{enumerate}[noitemsep, label = (\roman*)]
    \item \label{it:diff_reg} for all $(t, \omega)$ the maps $x \mapsto \alpha_t(x)$ and $x \mapsto \beta_t(x)$ have a weak derivative;
    \item \label{it:lipschitz_bspde} for all $(t, \omega)$ and $(r, p, q)$, $(r', p', q') \in L_{\eta}^2(\R^d) \times L_{\eta}^2(\R^d; \R^d) \times L_{\eta}^2(\R^d; \R^{d_W})$ we have
    \begin{align*}
        \bigl\lVert F_t(\cdot, r, p, q) - F_t(\cdot, r', p', q')\bigr\rVert_{\eta} \leq C\bigl(\lVert r - r'\rVert_{\eta} + \lVert p - p'\rVert_{\eta} + \lVert q - q'\rVert_{\eta}\bigr);
    \end{align*}
    \item \label{it:growth_bspde} the coefficients $\alpha$, $\nabla \cdot \alpha$, $\beta$, and $\nabla \cdot \beta$ are bounded by $C$ and for all $(t, \omega)$ and $(r, p, q) \in L_{\eta}^2(\R^d) \times L_{\eta}^2(\R^d; \R^d) \times L_{\eta}^2(\R^d; \R^{d_W})$, we have
    \begin{align*}
        \lVert F_t(\cdot, r, p, q)\rVert_{\eta} \leq C\bigl(1 + \lVert r\rVert_{\eta} + \lVert p\rVert_{\eta} + \lVert q\rVert_{\eta}\bigr);
    \end{align*}
    \item \label{it:nondeg_bspde} for all $(t, \omega, x)$ we have $\alpha_t(x) - \frac{\beta_t(x) \beta^{\top}_t(x)}{2} \geq c I_d$.
\end{enumerate}
\end{assumption}

We have the following existence and uniqueness result for BSPDE \eqref{eq:bspde}.

\begin{theorem} \label{thm:bspde}
Let Assumption \ref{ass:bspde} be satisfied. Then BSPDE \eqref{eq:bspde} has a unique solution $(u, q, m)$ for which $u \in D_{\bb{F}}^2([0, T]; L^2_{\eta}(\R^d))$. Moreover, it holds that
\begin{equation} \label{eq:bspde_estimate}
    \sup_{t \in [0, T]}\ev\lVert u_t\rVert_{\eta}^2 + \ev\lVert m_T\rVert_{\eta}^2 + \ev\int_0^T \lVert \nabla u_t \rVert_{\eta}^2 + \lVert q_t \rVert_{\eta}^2 \, \d t \leq C\bigl(1 + \ev\lVert \Psi \rVert_{\eta}^2\bigr)
\end{equation}
for a constant $C > 0$ that only depends on the coefficients, the exponent $\eta$, and $T$.
\end{theorem}

The remainder of this subsection is dedicated to proving Theorem \ref{thm:bspde}. First, we reduce to the case $\eta = 0$. We observe that if $(u, q, m) \in \cal{S}_{\eta}$ solves BSPDE \eqref{eq:bspde} then $(\tilde{u}, \tilde{q}, \tilde{m}) \in \cal{S}_0$ defined by $\tilde{u}_t(x) = e^{\eta(x)/2}u_t(x)$, $\tilde{q}_t(x) = e^{\eta(x)/2}q_t(x)$, and $\tilde{m}_t(x) = e^{\eta(x)/2}m_t(x)$ solves BSPDE \eqref{eq:bspde} with $F$ replaced by an appropriately chosen function $\tilde{F} \define [0, T] \times \Omega \times \R^d \times L^2(\R^d) \times L^2(\R^d; \R^d) \times L^2(\R^d; \R^{d_W}) \to \R$, satisfying Assumption \ref{ass:bspde} with $\eta$ replaced by the constant zero function $\R^d \to \R$, and terminal condition $\Psi$ replaced by $\tilde{\Psi} \in L^2(\Omega, \F, \pr; L^2(\R^d))$ defined as $\tilde{\Psi}(x) = e^{\eta(x)/2} \Psi(x)$. The converse is also true. Thus, we may simply analyse the BSPDE satisfied by $(\tilde{u}, \tilde{q}, \tilde{m})$. Existence, uniqueness, and the estimate \eqref{eq:bspde_estimate} then transfer to BSPDE \eqref{eq:bspde}. Hence, from now on we shall assume that $\eta = 0$.

We obtain existence and uniqueness for BSPDE \eqref{eq:bspde} through a fixed-point argument, so we first have to consider linear BSPDEs of the form
\begin{align} \label{eq:bspde_linear}
\begin{split}
    \d u_t(x) = -\Bigl(h_t(x) + \alpha_t(x) : \nabla^2 u_t(x) + \beta_t(x) : \nabla q_t(x)\Bigr) \, \d t  + q_t(x) \, \d W_t + \d m_t(x)
\end{split}
\end{align}
with terminal condition $u_T(x) = \Psi(x)$. The process $h$ lies in $L^2_{\bb{F}}([0, T]; L^2(\R^d))$ and fulfills the role of the coefficient $F$. A solution $(u, q, m) \in \cal{S}_0$ is understood in the sense of Definition \ref{def:bspde}. We construct a solution to BSPDE \eqref{eq:bspde_linear} by virtue of a Galerkin approximation. 

\begin{proposition} \label{prop:bspde_linear}
Let Assumption \ref{ass:bspde} be satisfied and let $h \in L^2_{\bb{F}}([0, T]; L^2(\R^d))$. Then BSPDE \eqref{eq:bspde_linear} has a unique solution $(u, q, m)$ for which $u \in D_{\bb{F}}^2([0, T]; L^2(\R^d))$. Moreover, for any $t \in [0, T]$, it holds that
\begin{align} \label{eq:bspde_linear_estimate}
\begin{split}
    \ev\lVert u_t\rVert_{L^2}^2 + \ev\lVert m_T - m_t\rVert_{L^2}^2 + \ev&\int_t^T \lVert \nabla u_s \rVert_{L^2}^2 + \lVert q_s \rVert_{L^2}^2 \, \d s \\
    &\leq C\biggl(\ev\lVert \Psi \rVert_{L^2}^2 + \ev\int_t^T  \lvert \langle u_s, h_s\rangle\rvert \, \d s\biggr)
\end{split}
\end{align}
for a constant $C > 0$ independent of $h$ and $t$.
\end{proposition}

\begin{proof}
\textit{Existence}: Our arguments mirror those given in the proof of \cite[Theorem 3.1]{jin_1999_bspde} and \cite[Lemma 4.5]{al_hussein_bspde_2009}. The main difference lies in the fact that BSPDE \eqref{eq:bspde_linear} is driven by a martingale with jumps. As mentioned above, we proceed via a Galerkin approximation: for any $n \geq 1$, we construct a Galerkin approximation $(u^n, q^n, m^n)$ on $n$ basis elements, then we show that the resulting sequence is bounded in $\cal{S}_0$ and that any weakly convergent subsequence yields a solution to BSPDE \eqref{eq:bspde_linear}.

Let us fix an orthonormal basis $(e_n)_{n \geq 1}$ of $L^2(\R^d)$ consisting of elements in $C^{\infty}_c(\R^d)$, such that the linear span of $(e_n)_{n \geq 1}$ is dense in $H^1(\R^d)$. For $n \geq 1$ we consider the BSDE
\begin{align*}
    \d u^{ni}_t = -\sum_{j = 1}^n \Bigl(\langle \alpha_t : \nabla^2 e_j, e_i\rangle u^{nj}_t + \langle \beta_t : \nabla e_j, e_i\rangle q^{nj}_t\Bigr) \, \d t - \langle h_t, e_i\rangle\, \d t  + q^{ni}_t \, \d W_t + \d m^{ni}_t
\end{align*}
with terminal condition $u^{ni}_T = \langle \Psi, e_i\rangle$ for $1 \leq i \leq n$. The existence of an $\bb{F}$-progressively measurable solution $u^{ni} \in D_{\bb{F}}^2([0, T])$, $q^{ni} \in L^2_{\bb{F}}([0, T]; \R^{d_W})$, $m^{ni} \in \cal{M}^2_{\bb{F}}([0, T])$, $i = 1$,\ldots, $n$ is guaranteed by \cite[Example 1.20]{carmona_mfg_ii_2018}. Next, we define the random function $u^n \in D_{\bb{F}}^2([0, T]; H^2(\R^d))$ by $u^n_t(x) = \sum_{i = 1}^n u^{ni}_t e_i(x)$ and, similarly, $q^n \in L_{\bb{F}}^2([0, T]; H^2(\R^d; \R^{d_W}))$ and $m^n \in \cal{M}_{\bb{F}}^2([0, T]; H^2(\R^d))$. Then, let $\pi_n$ denote the orthogonal projection of $L^2(\R^d)$ onto the linear span of $e_1$,~\ldots, $e_n$, which keeps $u^n_t$, $q^n_t$, and $m^n_t$ fixed. The projection allows us to rewrite the finite-dimensional linear BSDE above in the form
\begin{align} \label{eq:fd_adjoint}
\begin{split}
    \d u^n_t(x) = -\pi_n\bigl(h_t + \alpha_t : \nabla^2 u^n_t + \beta_t : \nabla q^n_t\bigr)(x) \, \d t + q^n_t(x) \, \d W_t + \d m^n_t(x)
\end{split}
\end{align}
with terminal condition $u^n_T(x) = (\pi_n\Psi)(x)$. 

Now, we apply the generalisation of It\^o's formula from Theorem \ref{thm:ito} for Hilbert space-valued semimartingales with jumps to obtain
\begin{align} \label{eq:fd_galerkin_ito}
\begin{split}
    \lVert u^n_t\rVert_{L^2}^2 &= \lVert \pi_n \Psi\rVert_{L^2}^2 - \int_t^T 2\Bigl(\bigl\langle \nabla u^n_s, \nabla \cdot (\alpha_s u^n_s) \bigr\rangle + \bigl\langle q^n_s, \nabla \cdot (\beta_s u^n_s) \bigr\rangle\Bigr) \, \d s \\
    &\ \ \ + 2\int_t^T \langle h_s, u^n_s\rangle \, \d s- \int_t^T 2\langle q^n_s, u^n_s\rangle \cdot \d W_s - \int_t^T 2 \langle u^n_{s-}, \d m^n_s\rangle \\
    &\ \ \ - \int_t^T \lVert q^n_s\rVert_{L^2}^2 \, \d s - ([m^n]_T - [m^n]_t).
\end{split}
\end{align}
The stochastic integrals appearing in \eqref{eq:fd_galerkin_ito} are true martingales by Proposition \ref{prop:bdg}, so disappear upon taking expectations on both sides. Hence, if we apply the same estimates as in the proof of \cite[Theorem 3.1]{jin_1999_bspde}, we find that for any $\delta > 0$ it holds that
\begin{align} \label{eq:fd_galerkin_est}
\begin{split}
    \ev\lVert u^n_t\rVert_{L^2}^2 &\leq \ev\lVert \pi_n \Psi\rVert_{L^2}^2 + \ev \int_t^T \Bigl(\delta \bigl(\lVert \nabla u^n_s\rVert_{L^2}^2 + \lVert q^n_s\rVert_{L^2}^2\bigr) + \tfrac{C_{\alpha}^2 + C_{\beta}^2}{\delta} \lVert u^n_s\rVert_{L^2}^2 + \lVert h_s\rVert_{L^2}^2\Bigr) \, \d s \\
    &\ \ \ - \ev\int_t^T \Bigl(2 \lVert \sigma_s \nabla u^n_s\rVert_{L^2}^2 + \bigl\lVert \beta^{\top}_s \nabla u^n_s + q^n_s\bigr\rVert_{L^2}^2\Bigr) \, \d s - \ev([m^n]_T - [m^n]_t),
\end{split}
\end{align}
where $\sigma_s(x)$ is a matrix with $\sigma^{\top}_s(x)\sigma_s(x) = \alpha_s(x) - \tfrac{\beta_s(x) \beta^{\top}_s(x)}{2}$, so that $\sigma^{\top}_s(x)\sigma_s(x) \geq c_{\alpha, \beta} I_d$ by Assumption \ref{ass:bspde} \ref{it:nondeg_bspde}. For any $\epsilon \in (0, 1)$, we have
\begin{equation*}
    -\bigl\lVert \beta^{\top}_s \nabla u^n_s + q^n_s\bigr\rVert_{L^2}^2 \leq \frac{(1 - \epsilon)C_{\beta}^2}{\epsilon} \lVert \nabla u^n_s\rVert_{L^2}^2 - (1 - \epsilon) \lVert q^n_s\rVert_{L^2}^2.
\end{equation*}
Now, we first choose $\epsilon$ sufficiently close to $1$ such that $\frac{(1 - \epsilon)C_{\beta}^2}{\epsilon} \leq \frac{c_{\alpha, \beta}}{2}$, then we fix $\delta = \tfrac{c_{\alpha, \beta}}{2} \land \frac{1 - \epsilon}{2}$. In view of Equation \eqref{eq:fd_galerkin_est} this yields
\begin{align*}
    \ev\lVert u^n_t\rVert_{L^2}^2 + \ev([m^n]_T - [&m^n]_t) + \ev\int_t^T \lVert \nabla u^n_s\rVert_{L^2}^2 + \lVert q^n_s\rVert_{L^2}^2 \, \d s \\
    &\leq C \ev\lVert \pi_n \Psi\rVert_{L^2}^2 + C \ev\int_t^T \lVert h_s \rVert + \lVert u^n_s\rVert_{L^2}^2 \, \d s
\end{align*}
for a positive constant $C > 0$ independent of $n \geq 1$ and $h$. Hence, by Gr\"onwall's inequality, we can deduce that
\begin{align} \label{eq:galerkin_estimate_2}
\begin{split}
    \ev\lVert u^n_t\rVert_{L^2}^2 + \ev \lVert m^n_T - m^n_t&\rVert_{L^2}^2 + \ev\int_t^T \lVert \nabla u^n_s\rVert_{L^2}^2 + \lVert q^n_s\rVert_{L^2}^2 \, \d s \\
    &\leq e^{CT}\ev\lVert \pi_n \Psi\rVert_{L^2}^2 + e^{CT}\ev \int_t^T \lVert h_t\rVert_{L^2}^2 \, \d t,
\end{split}
\end{align}
where we used that $\ev([m^n]_T - [m^n]_t) = \ev\bigl(\lVert m^n_T\rVert_{L^2}^2 - \lVert m^n_t\rVert_{L^2}^2\bigr) = \ev \lVert m^n_T - m^n_t\rVert_{L^2}^2$ because the martingale $m^n$ satisfies $\ev\langle m^n_T, m^n_t\rangle = \ev \langle m^n_t, m^n_t\rangle = \ev \lVert m^n_t\rVert_{L^2}^2$. Consequently, the sequence $(u^n, q^n, m^n)_n$ is bounded in $\cal{S}_0$ and, therefore, has a weakly convergent subsequence. Then we proceed as in the proofs of \cite[Theorem 3.1]{jin_1999_bspde} and \cite[Lemma 4.5]{al_hussein_bspde_2009} to show that any weak limit $(u, q, m) \in \cal{S}_0$ is a solution of BSPDE \eqref{eq:bspde_linear}. We will not provide the details and instead refer the reader to the cited references. Note that the test functions for the solutions lie in the closure of the linear of span of $(e_n)_{n \geq 1}$ in $H^1(\R^d)$. By our choice of $(e_n)_{n \geq 1}$ this closure is given by $H^1(\R^d)$, as desired. 

\textit{Uniqueness}: To prove uniqueness for BSPDE \eqref{eq:bspde_linear}, it is convenient to first establish \eqref{eq:bspde_linear_estimate} for arbitrary solutions to BSPDE \eqref{eq:bspde_linear}.
So let us fix a solution $(u, q, m)$ of BSPDE \eqref{eq:bspde_linear}. We apply the generalisation of It\^o's formula from Theorem \ref{thm:ito} to $u$, noting this in particular gives $u \in D_{\bb{F}}^2([0, T]; L^2(\R^d))$, and then use the same estimates as above to deduce the inequality
\begin{align*}
    \ev\lVert u_t\rVert_{L^2}^2 + \ev&([m]_T - [m]_t) + \ev\int_t^T \lVert \nabla u_s\rVert_{L^2}^2 + \lVert q_s\rVert_{L^2}^2 \, \d s \\
    &\leq C \ev\lVert \Psi\rVert_{L^2}^2 + C \ev\int_t^T \lvert \langle h_s, u_s\rangle\rvert + \lVert u_s\rVert_{L^2}^2 \, \d s
\end{align*}
for a positive constant $C > 0$ independent of $h$ and $t$. Then a simple application of Gr\"onwall's inequality yields the estimate \eqref{eq:bspde_linear_estimate} for $(u, q, m)$. 

Now, if $(u', q', m')$ is another solution to BSPDE \eqref{eq:bspde_linear} with the same terminal condition $\Psi$ and input $h$, then the difference $(u - u', q - q', m - m')$ satisfies BSPDE \eqref{eq:bspde_linear} with vanishing initial condition and input. Hence, the just established bound from \eqref{eq:bspde_linear_estimate} immediately implies that the difference $(u - u', q - q', m - m')$ vanishes. Thus, BSPDE \eqref{eq:bspde_linear} has a unique solution.
\end{proof}

We can use Proposition \ref{prop:bspde_linear} and, in particular, the estimate \eqref{eq:bspde_linear_estimate} to obtain a solution to BSPDE \eqref{eq:bspde} via a fixed-point argument. Let us note again that we will only prove Theorem \ref{thm:bspde} for a vanishing weight $\eta = 0$.

\begin{proof}[Proof of Theorem \ref{thm:bspde}]
We construct a map $\Phi \define \cal{S}_0 \to \cal{S}_0$ as follows: for any $(u, q, m) \in \cal{S}_0$ we define the random function $h^u \define [0, T] \times \Omega \times \R^d \to \R$ by $h^u_t(x) = F_t(x, u_t, \nabla u_t, q_t)$. Note that Assumption \ref{ass:bspde} \ref{it:growth_bspde} implies that
\begin{align} \label{eq:h_bound}
\begin{split}
    \ev \int_0^T \lVert h^u_t\rVert_{L^2}^2 \, \d t &\leq \ev \int_0^T C^2\bigl(1 + \lVert u_t\rVert_{L^2} + \lVert \nabla u_t\rVert_{L^2} + \lVert q_t\rVert_{L^2}\bigr)^2 \, \d t < \infty,
\end{split}
\end{align}
so that $h^u$ is an element of $L^2_{\bb{F}}([0, T]; L^2(\R^d))$. Then, we define $\Phi(u, q, m)$ as the unique solution $(u', q', m')$ of the linear BSPDE \eqref{eq:bspde_linear} with input $h$ replaced by $h^u$. Using the Banach fixed-point theorem, we will show that $\Phi$ has a unique fixed point, which constitutes a solution of BSPDE \eqref{eq:bspde}. Uniqueness of the fixed point then implies uniqueness for BSPDE \eqref{eq:bspde}.

Fix two elements $(u^1, q^1, m^1)$, $(u^2, q^2, m^2) \in \cal{S}_0$, so the difference $\Delta = (\Delta^u, \Delta^q, \Delta^m) = \Phi(u^1, q^1, m^1) - \Phi(u^2, q^2, m^2)$ solves the linear BSPDE \eqref{eq:bspde_linear} with input $h^{u^1} - h^{u^2}$ and initial condition zero. Using \eqref{eq:bspde_linear_estimate} we get that
\begin{align} \label{eq:bspde_cutoff_estimate}
    \ev\lVert \Delta^u_t\rVert_{L^2}^2 &+ \ev\lVert \Delta^m_T - \Delta^m_t\rVert_{L^2}^2 + \ev\int_t^T \lVert \nabla \Delta^u_s \rVert_{L^2}^2 + \lVert \Delta^q_s \rVert_{L^2}^2 \, \d s \notag \\
    &\leq C\ev\int_t^T \bigl\lvert \bigl\langle \Delta^u_s, h^{u^1}_s - h^{u^2}_s\bigr\rangle \bigr\rvert \, \d s \notag \\
    &\leq \frac{C}{2\epsilon} \ev \int_t^T \lVert \Delta^u_s\rVert_{L^2}^2 \, \d t + \frac{\epsilon}{2} \ev \int_t^T \bigl\lVert h^{u^1}_s - h^{u^2}_s\bigr\rVert_{L^2}^2 \, \d s \notag \\
    &\leq \frac{C}{2\epsilon} \ev \int_t^T \lVert \Delta^u_s\rVert_{L^2}^2 \, \d t + \frac{3C_F^2\epsilon}{2} \ev \int_t^T \lVert u^1_s - u^2_s\rVert_{H^1}^2 + \lVert q^1_s - q^2_s\rVert_{L^2}^2 \, \d s,
\end{align}
where the last inequality follows from Assumption \ref{ass:bspde} \ref{it:lipschitz_bspde}. Next, we fix $\epsilon = (3C_F^2)^{-1}$ and set $C_0 = 3 C C_F^2$ to find that
\begin{align} \label{eq:bspde_estimate_1}
\begin{split}
    \ev\lVert \Delta^u_t\rVert_{L^2}^2 &+ \ev\lVert \Delta^m_T - \Delta^m_t\rVert_{L^2}^2 + \ev\int_t^T \lVert \nabla \Delta^u_s \rVert_{L^2}^2 + \lVert \Delta^q_s \rVert_{L^2}^2 \, \d s \\
    &\leq \frac{C_0}{2} \ev \int_t^T \lVert \Delta^u_s\rVert_{L^2}^2 \, \d t + \frac{1}{2} \ev \int_t^T \lVert u^1_s - u^2_s\rVert_{H^1}^2 + \lVert q^1_s - q^2_s\rVert_{L^2}^2 \, \d s.
\end{split}
\end{align}
Now, we proceed as in \cite[Example 1.20]{carmona_mfg_ii_2018}. We multiply the previous inequality by $e^{\kappa t}$ for $\kappa \geq 1$ and then integrate over $[0, T]$ to obtain
\begin{align*}
    \int_0^T &e^{\kappa t} \ev\lVert \Delta^u_t\rVert_{L^2}^2 \, \d t + \int_0^T \frac{1}{\kappa}(e^{\kappa s} - 1) \ev\bigl[\lVert \nabla \Delta^u_s \rVert_{L^2}^2 + \lVert \Delta^q_s \rVert_{L^2}^2\bigr] \, \d s \\
    &\leq \frac{C_0}{2\kappa} \int_0^T e^{\kappa s} \ev\lVert \Delta^u_s\rVert_{L^2}^2 \, \d s + \frac{1}{2} \int_0^T \frac{1}{\kappa}(e^{\kappa s} - 1) \ev\bigl[\lVert u^1_s - u^2_s\rVert_{H^1}^2 + \lVert q^1_s - q^2_s\rVert_{L^2}^2\bigr]\, \d s.
\end{align*}
Next, we evaluate Equation \eqref{eq:bspde_estimate_1} at $t = 0$, multiply it by $\frac{1}{2\kappa}$, and add it to the inequality above. This yields
\begin{align*}
    \int_0^T e^{\kappa t} \ev\lVert&\Delta^u_t\rVert_{L^2}^2 \, \d t + \frac{1}{2\kappa} \ev \lVert m_T\rVert_{L^2}^2 + \int_0^T \theta(s) \ev\bigl[\lVert \nabla \Delta^u_s \rVert_{L^2}^2 + \lVert \Delta^q_s \rVert_{L^2}^2\bigr] \, \d s \\
    &\leq \frac{C_0}{4\kappa} \int_0^T e^{\kappa s} \ev\lVert \Delta^u_s\rVert_{L^2}^2 \, \d s + \frac{1}{2} \int_0^T \theta(s) \ev\bigl[\lVert u^1_s - u^2_s\rVert_{H^1}^2 + \lVert q^1_s - q^2_s\rVert_{L^2}^2\bigr] \, \d s,
\end{align*}
where $\theta(s) = \frac{1}{2\kappa} + \frac{1}{\kappa}(e^{\kappa s} - 1)$.
Then we fix $\kappa \geq 1$ large enough so that $\frac{C_0}{4\kappa} < \frac{1}{4}$ and rearrange, which finally implies
\begin{equation} \label{eq:bspde_estimate_2}
    \int_0^T \theta(t) \ev\bigl[\lVert\Delta^u_t\rVert_{H^1}^2 + \lVert \Delta^q_t \rVert_{L^2}^2\bigr] \, \d t + \frac{2}{3\kappa} \ev \lVert m_T\rVert_{L^2}^2 \leq \frac{2}{3} \int_0^T \theta(t) \ev\bigl[\lVert u^1_s - u^2_s\rVert_{H^1}^2 + \lVert q^1_s - q^2_s\rVert_{L^2}^2\bigr]\, \d t.
\end{equation}
Here we used that $\theta(t) \leq e^{\kappa t}$. We can equipped $\cal{S}_0$ with the norm
\begin{equation*}
    (u, q, m) \mapsto \int_0^T \theta(t) \ev\bigl[\lVert u_t\rVert_{H^1}^2 + \ev\lVert q_t \rVert_{L^2}^2\bigr] \, \d t + \frac{2}{3\kappa} \ev \lVert m_T\rVert_{L^2}^2,
\end{equation*}
which turns $\cal{S}_0$ into a Banach space. Moreover, owing to \eqref{eq:bspde_estimate_2}, the map $\Phi$ is a contraction on this Banach space and by the Banach fixed-point theorem has a unique fixed point $(u, q, m)$. This fixed point is a solution to BSPDE \eqref{eq:bspde} in the sense of Definition \eqref{def:bspde}. Moreover, since any other solution to BSPDE \eqref{eq:bspde} is also a fixed point of $\Phi$, the solution $(u, q, m)$ is unique.

Lastly, we derive \eqref{eq:bspde_estimate} for $(u, q, m)$. Since $(u, q, m)$ solves BSPDE \eqref{eq:bspde_linear} with input $h^u$, the estimate \eqref{eq:bspde_linear_estimate} yields
\begin{align*}
    \ev\lVert u_t\rVert_{L^2}^2 + \ev\lVert m_T - m_t\rVert_{L^2}^2 + \ev&\int_t^T \lVert \nabla u_s \rVert_{L^2}^2 + \lVert q_s \rVert_{L^2}^2 \, \d s \\
    &\leq C\biggl(\ev\lVert \Psi \rVert_{L^2}^2 + \ev\int_t^T  \lvert \langle u_s, h^u_s\rangle\rvert \, \d s\biggr)
\end{align*}
Now, we proceed as above, using a bound analogous to \eqref{eq:h_bound} to estimate $\ev \int_t^T \lVert h^u_s\rVert_{L^2}^2 \, \d s$, in order to establish \eqref{eq:bspde_estimate}. This concludes the proof.
\end{proof}

\begin{proof}[Proof of Propositions \ref{prop:adjoint} and \ref{prop:hjb}]
The existence and uniqueness of solutions to BSPDE \eqref{eq:adjoint} and BSPDE \eqref{eq:hjb} is a straightforward application of Theorem \ref{thm:bspde}. We simply need to verify the assumptions of Theorem \ref{thm:bspde} under the hypotheses of Propositions \ref{prop:adjoint} and \ref{prop:hjb}. Note that in the context of Propositions \ref{prop:adjoint} and \ref{prop:hjb}, the exponent $\eta$ used in this section is given by $-\eta_0 \sqrt{1 + \lvert x \rvert^2}$. This in particular ensures the Lipschitz continuity of the nonlocal dependence on $u_t$ and $\nabla u_t$ appearing in the coefficients of BSPDEs \eqref{eq:adjoint} and \eqref{eq:hjb} with respect to the norm $\lVert \cdot \rVert_{\eta}$. We shall not provide the details here.
\end{proof}

%% file: 4_control/fpe.tex
The goal of this section is to analyse the control problem for the random Fokker--Planck equation \eqref{eq:rfpe}. We begin by establishing well-posedness for the state equation and, subsequently, establish the existence of optimal controls. Then, we prove the necessary and sufficient SMP.

\subsection{Well-Posedness of the Fokker--Planck Equation} \label{sec:sfpe}

In this subsection, we study the stochastic and random Fokker--Planck equations, (S)PDEs \eqref{eq:sfpe} and \eqref{eq:rfpe} introduced in Section \ref{sec:main_results}. In particular, we prove Propositions \ref{prop:sfpe_to_rfpe} and \ref{prop:rfpe}. Throughout this section Assumption \ref{ass:fpe} is in place and we recall that $\eta(x) = \eta_0 \sqrt{1 + \lvert x \rvert^2}$ for some $\eta_0 > 0$ and $\bar{\eta} = -\eta$.

\begin{proof}[Proof of Proposition \ref{prop:sfpe_to_rfpe}]
We only show that any solution $\nu$ to the stochastic Fokker--Planck equation \eqref{eq:sfpe} induces a solution $\mu$ to the random Fokker--Planck equation \eqref{eq:rfpe}, given by $\mu_t = (\id - \sigma_0W_t)^{\#}\nu_t$. By mollifying $\nu$ if necessary, we may assume that $\nu_t$ has a density in $L^2(\R^d)$, which we also denote by $\nu_t$, such that $\nu \in L_{\bb{F}}^2([0, T]; H^1(\R^d))$. Next, let $\varphi \in C^2_c(\R^d)$ and set $\tilde{\varphi}_t(x) = \varphi(x - \sigma_0 W_t)$, so that
\begin{equation*}
    \d \tilde{\varphi}_t(x) = -(\nabla \tilde{\varphi}_t(x))^{\top} \sigma_0 \, \d W_t + \frac{1}{2} \nabla^2 \tilde{\varphi}_t(x) \colon \sigma_0\sigma_0^{\top} \, \d t.
\end{equation*}
Then, we may apply the generalisation of It\^o's formula from Theorem \ref{thm:ito} to $\langle \nu_t, \tilde{\varphi}\rangle$ to find
\begin{align*}
    \d \langle \mu_t, \varphi\rangle &= \d \langle \nu_t, \tilde{\varphi}_t\rangle \\
    &= \bigl\langle \nu_t, \L \tilde{\varphi}_t(t, \cdot, \nu_t, \gamma_t)\bigr\rangle \, \d t + \frac{1}{2} \bigl\langle \nu_t, \nabla^2 \tilde{\varphi}_t \colon \sigma_0\sigma_0^{\top}\bigr\rangle \, \d t - \bigl\langle \nu_t, \nabla^2 \tilde{\varphi}_t \colon \sigma_0 \sigma_0^{\top}\bigr\rangle \, \d t \\
    &= \bigl\langle \mu_t, \tilde{\L}\varphi(t, \cdot, \mu_t, \tilde{\gamma}_t)\bigr\rangle \, \d t.
\end{align*}
Since $\varphi \in C^2_c(\R^d)$ was arbitrary, this shows that $\mu$ solves PDE \eqref{eq:rfpe}. The equality of the cost functions, $\tilde{J}(\tilde{\gamma}) = J(\gamma)$, is clear.
\end{proof}

We immediately continue with the proof of Proposition \ref{prop:rfpe}. 

\begin{proof}[Proof of Proposition \ref{prop:rfpe}]
We proceed in several steps. First, we study general linear PDEs with random coefficients. Then, a solution $\mu \in C_{\bb{F}}^2([0, T]; L_{\eta}^2(\R^d)) \cap L_{\bb{F}}^2([0, T]; H_{\eta}^1(\R^d))$ to PDE \eqref{eq:rfpe} is obtained through a fixed-point argument. Subsequently, we verify that this solution is in fact measure-valued and, lastly, we established uniqueness among such measure-valued solutions.  

\textit{Step 1}: We begin by studying linear PDEs with random coefficients. For arbitrary functions $\alpha \in L_{\bb{F}}^{\infty}([0, T]; L^{\infty}(\R^d))$, $\beta \in L_{\bb{F}}^{\infty}([0, T]; L^{\infty}(\R^d; \R^d))$, $h^0 \in L_{\bb{F}}^2([0, T]; L^2_{\eta}(\R^d))$, and $h^1 \in L_{\bb{F}}^2([0, T]; L^2_{\eta}(\R^d; \R^d))$ let us consider the linear PDE
\begin{equation} \label{eq:sfpe_random}
    \d \langle \mu_t, \varphi\rangle = \Bigl(\bigl\langle \mu_t, \alpha_t \varphi + \beta_t \cdot \nabla \varphi + \tilde{a}(t, \cdot) \colon \nabla^2 \varphi\bigr\rangle + \langle h^0_t, \varphi\rangle + \langle h^1_t, \nabla \varphi\rangle\Bigr)\, \d t
\end{equation}
for $\varphi \in C_c^2(\R^d)$, with initial condition $\mu_0 \in L^2_{\eta}(\R^d)$. Using the parabolic analogue of the well-known Lax--Milgram theorem, see e.g.\@ \cite[Theorem 4.1]{lions_bvp_1972}, we obtain a unique solution $\mu \in C_{\bb{F}}^2([0, T]; L_{\eta}^2(\R^d)) \cap L_{\bb{F}}^2([0, T]; H_{\eta}^1(\R^d))$ to PDE \eqref{eq:sfpe_random}. Moreover, applying It\^o's formula, Theorem \ref{thm:ito}, to the PDE satisfied by $e^{\eta/2} \mu$ and applying standard estimates, we obtain
\begin{align} \label{eq:sfpe_estimate}
\begin{split}
    \sup_{0 \leq s \leq t} &\lVert \mu_s\rVert_{\eta}^2 + \int_0^t \lVert \nabla \mu_s\rVert_{\eta}^2 \, \d s \\
    &\leq C\lVert \mu_0\rVert_{\eta}^2 + C \int_0^t \Bigl(\bigl\lvert \bigl\langle \mu_s (\alpha_s + \lvert \beta_s\rvert) + h^0_s + \lvert h^1_s\rvert, \mu_s\bigr\rangle_{\eta}\bigr\rvert + \bigl\lvert\bigl\langle \mu_s \beta_s + h^1_s, \nabla \mu_s\bigr\rangle_{\eta}\bigr\rvert\Bigr) \, \d s
\end{split}
\end{align}
for $0 \leq t \leq T$ and a constant $C > 0$ depending on the uniform parabolicity constant from Assumption \ref{ass:fpe} \ref{it:nondeg_fpe}, the parameter $\eta_0$, and the time horizon $T$.

\textit{Step 2}: Next, we employ a fixed-point argument to find a solution for the nonlinear PDE \eqref{eq:rfpe}. Let us construct a map $\Phi$ from $L_{\bb{F}}^2([0, T]; L_{\eta}^2(\R^d))$ to itself as follows: for $\mu \in L_{\bb{F}}^2([0, T]; L_{\eta}^2(\R^d))$ we define the processes $\alpha^{\mu} \in L_{\bb{F}}^{\infty}([0, T]; L^{\infty}(\R^d))$ and $\beta^{\mu} \in L_{\bb{F}}^{\infty}([0, T]; L^{\infty}(\R^d; \R^d))$ by $\alpha^{\mu}_t(x) = \tilde{\lambda}(t, x, \mu_t)$ and $\beta^{\mu}_t(x) = \tilde{b}\bigl(t, x, \mu_t, \gamma_t(x)\bigr)$ for $(t, \omega,x) \in [0, T] \times \Omega \times \R^d$. Then, we let $\Phi(\mu) \in C_{\bb{F}}^2([0, T]; L_{\eta}^2(\R^d)) \cap L_{\bb{F}}^2([0, T]; H_{\eta}^1(\R^d)) \subset L_{\bb{F}}^2([0, T]; L_{\eta}^2(\R^d))$ denote the unique solution to PDE \eqref{eq:sfpe_random} with $\alpha$, $\beta$ replaced by $\alpha^{\mu}$, $\beta^{\mu}$, $h^0 = 0$, $h^1 = 0$, and initial condition $\mu_0 = v_0$, which is in $L_{\eta}^2(\R^d)$. From \eqref{eq:sfpe_estimate}, the boundedness of $\alpha$ and $\beta$, and an application of Gr\"onwall's inequality, it follows that $\sup_{0 \leq t \leq T} \lVert \Phi(\mu)_t\rVert_{\eta}$ is bounded by a constant independent of $\mu \in L_{\bb{F}}^2([0, T]; L_{\eta}^2(\R^d))$. 
We will show that $\Phi$ is a contraction on $L_{\bb{F}}^2([0, T]; L_{\eta}^2(\R^d))$, when equipped with a norm (that we shall define at a later stage) which is different from but equivalent to the standard norm. 

For two flows $\mu^1$, $\mu^2 \in L_{\bb{F}}^2([0, T]; L_{\eta}^2(\R^d))$ we can consider the difference $\Delta = (\Delta_t)_{0 \leq t \leq T}$ given by $\Delta_t = \Phi(\mu^1)_t - \Phi(\mu^2)_t$. The process $\Delta$ solves PDE \eqref{eq:sfpe_random} with $\alpha$, $\beta$ replaced by $\alpha^{\mu_1},$ $\beta^{\mu_1}$ and $h^0$ and $h^1$ defined by
\begin{equation*}
    h^0_t(x) = \Phi(\mu^2)_t(x) \bigl(\alpha^{\mu_1}_t(x) - \alpha^{\mu_2}_t(x)\bigr) \quad \text{and} \quad h^1_t(x) = \Phi(\mu^2)_t(x) \bigl(\beta^{\mu_1}_t(x) - \beta^{\mu_2}_t(x)\bigr),
\end{equation*}
respectively. Thus, the estimate \eqref{eq:sfpe_estimate} implies
\begin{align} \label{eq:bounds_delta}
\begin{split}
    \sup_{0 \leq s \leq t}& \lVert \Delta_t\rVert_{\eta}^2 + \int_0^t \lVert \nabla \Delta_s\rVert_{\eta}^2 \, \d s \\
    &\leq C \int_0^t \bigl\lvert \bigl\langle \Delta_s, (\alpha^{\mu_1}_s + \lvert \beta^{\mu_1}_s\rvert) \Delta_s + \beta^{\mu_1}_s \cdot \nabla \Delta_s\bigr\rangle_{\eta}\bigr\rvert \, \d s \\
    &\ \ \ + C\int_0^T \Bigl\lvert \Bigl\langle \Phi(\mu^2)_s, \bigl(\alpha^{\mu_1}_t - \alpha^{\mu_2}_t + \lvert \beta^{\mu_1}_s - \beta^{\mu_2}_s\rvert\bigr)\Delta_s + (\beta^{\mu_1}_s - \beta^{\mu_2}_s) \cdot \nabla \Delta_s\Bigr\rangle_{\eta}\Bigr\rvert \, \d s.
\end{split}
\end{align}
The integrand in the second line above is easily estimated by $\frac{2\epsilon (C_{\lambda} + C_b) + C_b^2}{2\epsilon} \lVert \Delta_s\rVert_{\eta}^2 + \frac{\epsilon}{2} \lVert \nabla \Delta_s \rVert_{\eta}^2$ for an $\epsilon > 0$ that we shall choose later. Next, by Assumption \ref{ass:fpe} \ref{it:continuity_fpe} it holds that
\begin{align*}
    \Bigl\lvert \Bigl\langle \Phi(\mu^2)_s&, \bigl(\alpha^{\mu_1}_t - \alpha^{\mu_2}_t + \lvert \beta^{\mu_1}_s - \beta^{\mu_2}_s\rvert\bigr)\Delta_s + (\beta^{\mu_1}_s - \beta^{\mu_2}_s) \cdot \nabla \Delta_s\Bigr\rangle_{\eta}\Bigr\rvert \\
    &\leq \lVert \Phi(\mu^2)_s\rVert_{\eta} d_0(\mu^1_t, \mu^2_t)\bigl((C_{\lambda} + C_b)\lVert \Delta_s\rVert_{\eta}  + C_b\lVert \nabla \Delta_s\rVert_{\eta}\bigr)  \\
    &\leq \frac{C_0 (1 + \epsilon)\lVert \Phi(\mu^2)_s\rVert_{\eta}^2}{2\epsilon}\lVert \mu^1_s - \mu^2_s\rVert_{\eta}^2 + \frac{1}{2} \lVert \Delta_s\rVert_{\eta}^2 + \frac{\epsilon}{2}\lVert \nabla \Delta_s \rVert_{\eta}^2
\end{align*}
for some constant $C_0 > 0$. Here we used that $d_0(v, v') \leq \lVert e^{-\eta/2}\rVert_{L^2} \lVert v - v'\rVert_{\eta}$ for all $v$, $v' \in L^2_{\eta}(\R^d)$. Inserting the previous two estimates into \eqref{eq:bounds_delta}, choosing $\epsilon = \frac{1}{2C}$, rearranging, and then applying Gr\"onwall's inequality implies that
\begin{equation*}
    \lVert \Delta_t\rVert_{\eta}^2 + \frac{1}{2} \int_0^t \lVert \nabla \Delta_s\rVert_{\eta}^2 \, \d s \leq C_1 \int_0^t\lVert \mu^1_s - \mu^2_s\rVert_{\eta}^2 \, \d s
\end{equation*}
for $C_1 > 0$, where we used that $ \lVert \Phi(\mu^2)_s\rVert_{\eta}^2$ is bounded uniformly over $(t, \omega) \in [0, T] \times \Omega$ and $\mu^2 \in L^2_{\bb{F}}([0, T]; L^2_{\eta}(\R^d))$. We drop the integral on the left-hand side above, multiply both sides with $e^{-2C_1t}$, and then integrate over $t \in [0, T]$. This yields
\begin{equation*}
    \ev \int_0^T e^{-2C_1t} \lVert \Delta_t\rVert_{\eta}^2 \, \d t \leq \frac{1}{2} \ev \int_0^T e^{-2C_1t}\lVert \mu^1_t - \mu^2_t\rVert_{\eta}^2 \, \d t.
\end{equation*}
Clearly, the function $\mu \mapsto \ev\int_0^T e^{-2C_1t} \lVert \mu_t\rVert_{\eta}^2 \, \d t$ defines a norm on $L_{\bb{F}}^2([0, T]; L_{\eta}^2(\R^d))$, which is equivalent to the standard norm and turns $\Phi$ into a contraction. Thus, by the Banach fixed-point theorem, there exists a unique fixed point $\mu \in L_{\bb{F}}^2([0, T]; L_{\eta}^2(\R^d))$. Since $\mu = \Phi(\mu)$, it holds that $\mu \in C_{\bb{F}}^2([0, T]; L_{\eta}^2(\R^d)) \cap L_{\bb{F}}^2([0, T]; H_{\eta}^1(\R^d))$ solves PDE \eqref{eq:rfpe}. If $\mu' \in C_{\bb{F}}^2([0, T]; L_{\eta}^2(\R^d)) \cap L_{\bb{F}}^2([0, T]; H_{\eta}^1(\R^d))$ is another solution to PDE \eqref{eq:rfpe}, then $\mu'$ is also a fixed point of $\Phi$, so by uniqueness of the fixed point of $\Phi$ we get $\mu' = \mu$. In other words, $\mu$ is the unique solution in $C_{\bb{F}}^2([0, T]; L_{\eta}^2(\R^d)) \cap L_{\bb{F}}^2([0, T]; H_{\eta}^1(\R^d))$. The estimate \eqref{eq:rfpe_estimate} follows from \eqref{eq:sfpe_estimate}.

\textit{Step 3}: Next, let us show that $\mu$ actually takes values in $C([0, T]; \M(\R^d))$. First, we establish $\mu_t(x) \geq 0$ for a.e.\@ $x \in \R^d$ and $\leb \otimes \pr$-a.e.\@ $(t, \omega) \in [0, T] \times \Omega$. For that we choose the function $F$ in the generalisation of It\^o's formula from Theorem \ref{thm:ito} to be $F(v) = \int_{\R^d} \varphi(v(x)) \, \d x$ with $\varphi(x) = (x_-)^2$. Using the identities $x \varphi'(x) = -2 \varphi(x)$ and $x^2 \varphi''(x) = 2 \varphi(x)$ it is easy to see that
\begin{equation*}
    F(\mu_t) \leq F(\mu_0) + \int_0^t \biggl(2 C_{\lambda} + \frac{C_b^2 + C_a^2}{c_a}\biggr) F(\mu_s) \, \d s,
\end{equation*}
where $c_a > 0$ is the uniform parabolicity constant from Assumption \ref{ass:fpe} \ref{it:nondeg_fpe}. Since $F(\mu_0) = F(v_0) = 0$, Gr\"onwall's inequality implies that $F(\mu_t) = 0$, which means that $\mu_t(x) \geq 0$ for a.e.\@ $x \in \R^d$. By the discussion below the statement of Proposition \ref{prop:rfpe} this shows that $\mu_t \in \M(\R^d)$. Continuity of the trajectories then follows from the fact that $\langle \mu_s, \varphi\rangle \to \langle \mu_t, \varphi\rangle$ as $s \to t$ for $\varphi \in C^2_c(\R^d)$ and that $\lVert e^{\eta/2} \mu_s\rVert_{L^2}$ is bounded uniformly over $s \in [0, T]$.

\textit{Step 4}: So far we only proved uniqueness of solutions in the space $C_{\bb{F}}^2([0, T]; L_{\eta}^2(\R^d)) \cap L_{\bb{F}}^2([0, T]; H_{\eta}^1(\R^d))$, but we need to establish uniqueness within the space of $\bb{F}$-adapted $C([0, T]; \M(\R^d))$-valued processes. Assume we have another $\bb{F}$-adapted solution $\mu'$ of PDE \eqref{eq:rfpe} that takes values in $C([0, T]; \M(\R^d))$. Then, using a mollification argument together with standard PDE estimates, we can show that $\mu'$ must actually take values in $C_{\bb{F}}^2([0, T]; L_{\eta}^2(\R^d)) \cap L_{\bb{F}}^2([0, T]; H_{\eta}^1(\R^d))$ as well. Consequently, uniqueness within $C_{\bb{F}}^2([0, T]; L_{\eta}^2(\R^d)) \cap L_{\bb{F}}^2([0, T]; H_{\eta}^1(\R^d))$ guarantees that $\mu' = \mu$. This concludes the proof.
\end{proof}

%% file: 4_control/optimal.tex
\subsection{Existence of Optimal Controls}

In this subsection, we show that the value of the optimal control problem for PDE \eqref{eq:rfpe} does not depend on the probabilistic setup (Proposition \ref{prop:independent}) and that there exists a probabilistic setup, which admits an optimal control. To establish the latter, we build on the compactness arguments from \cite{anita_sde_fpe_2021}. For the remainder of the subsection we will emphasise the probabilistic setup $\bb{W}$ by adding it as a subscript to the symbols denoting the generator, the space of controls, and the cost functional.

We begin with some preliminary remarks on the parabolic function spaces appearing in our analysis. First, we know that for any $R > 0$, the space $H^1_{\eta}(B_R)$ is compactly embedded in $L_{\eta}^2(B_R)$. Here $B_R$ denotes the centred ball of radius $R$ in $\R^d$. The space $L_{\eta}^2(B_R)$, in turn, is continuously embedded in $H^1_{\eta}(\R^d)^{\ast}$, where we extend real-valued functions on $B_R$ to $\R^d$ by zero. Consequently, the Aubin--Lions lemma (see e.g.\@ \cite[Theorem 3.1.1]{zheng_nonlinear_2004}) implies that the space $W([0, T]; H^1_{\eta}(B_R))$ of functions $h \in L^2([0, T]; H^1_{\eta}(B_R))$ whose distributional time derivative $\dot{h}$ lies in $L^2([0, T]; H^1_{\eta}(\R^d)^{\ast})$, is compactly embedded in $L^2([0, T]; L_{\eta}^2(B_R))$. 

Now, let $\gamma \in \bb{G}_{\bb{W}}$ and denote the corresponding solution to PDE \eqref{eq:rfpe} by $\mu$. Then, it follows immediately from PDE \eqref{eq:rfpe} and the estimate \eqref{eq:rfpe_estimate} that
\begin{equation} \label{eq:distr_der}
    \int_0^T \lVert \dot{\mu}_t\rVert_{H^1_{\eta}(\R^d)^{\ast}}^2 \, \d t \leq C \int_0^T \lVert \mu_t\rVert_{H^1_{\eta}(\R^d)}^2 \, \d t \leq C \lVert v_0\rVert_{\eta}^2
\end{equation}
for some $C > 0$. Together with estimate \eqref{eq:rfpe_estimate}, this implies that every realisation of $\mu$ for any control $\gamma \in \bb{G}_{\bb{W}}$ on any probabilistic setup $\bb{W}$ lies in the same sufficiently large but bounded subset of $W([0, T]; H^1_{\eta}(B_R))$. Hence, in view of the compact embedding from above, if $(\gamma^n)_n$ is a sequence of controls, each defined on some probabilistic setup $\bb{W}_n$, and we denote the corresponding solutions to PDE \eqref{eq:rfpe} by $\mu^n$, then the sequence $(\mu^n)_n$ is tight on $L^2([0, T]; L_{\eta}^2(B_R))$ for all $R > 0$. 

The following proposition establishes semicontinuity properties of the cost functional under suitable assumption on the drift coefficient and the running cost. In its statement, we add the qualifier ``$\text{weak}$'' to the function space $L^2([0, T]; L_{\textup{loc}}^2(\R^d; \R^{d_G}))$, to indicate that $L^2([0, T]; L_{\textup{loc}}^2(\R^d; \R^{d_G}))$ is equipped with the weak topology.

\begin{proposition} \label{prop:stability}
Let Assumption \ref{ass:fpe} be satisfied. Let $(\bb{W}_n)_n$ be a sequence of probabilistic setups with associated Brownian motions $W^n$, let $\gamma^n \in \bb{G}_{\bb{W}^n}$, and denote the corresponding solution to PDE \eqref{eq:rfpe} by $\mu^n$. Then $(\mu^n, \gamma^n, W^n)_n$ is tight on $L^2([0, T]; L_{\eta, \textup{loc}}^2(\R^d)) \times (L^2([0, T]; L_{\textup{loc}}^2(\R^d; \R^{d_G})), \textup{weak}) \times C([0, T]; \R^{d_W})$ and any subsequential limit $(\mu, \gamma, W)$ can be realised on some probabilistic setup $\bb{W}$ such that $W$ is the Brownian motion underlying $\bb{W}$ and $\gamma \in \bb{G}_{\bb{W}}$. If in addition
\begin{enumerate}[noitemsep, label = \textup{(\roman*)}]
    \item \label{it:linear_convex} for all $(t, x, v) \in [0, T] \times \R^d \times \M(\R^d)$ the map $G \ni g \mapsto b(t, x, v, g)$ is affine and the map $G \ni g \mapsto f(t, x, v, g)$ is convex, then $\mu$ is the solution to PDE \eqref{eq:rfpe} corresponding to $\gamma$ and
    \begin{equation} \label{eq:cost_lsc}
        \liminf_{k \to \infty} \tilde{J}_{\bb{W}_{n_k}}(\gamma^{n_k}) \geq \tilde{J}_{\bb{W}}(\gamma)
    \end{equation}
    along the convergent subsequence $(n_k)_k$;
    \item \label{it:cont_usc} the subsequential weak convergence of $(\gamma^n)_n$ is with respect to the strong topology on $L^2([0, T]; L_{\textup{loc}}^2(\R^d; \R^{d_G}))$ and for all $(t, x, v) \in [0, T] \times \R^d \times \M(\R^d)$ the map $G \ni g \mapsto b(t, x, v, g)$ is continuous and the map $G \ni g \mapsto f(t, x, v, g)$ is upper semicontinuous, then $\mu$ is the solution to PDE \eqref{eq:rfpe} corresponding to $\gamma$ and
    \begin{equation} \label{eq:cost_usc}
        \limsup_{k \to \infty} \tilde{J}_{\bb{W}_{n_k}}(\gamma^{n_k}) \leq \tilde{J}_{\bb{W}}(\gamma)
    \end{equation}
    along the convergent subsequence $(n_k)_k$.
\end{enumerate}
\end{proposition}

\begin{proof}
First, we prove that under the assumptions from \ref{it:linear_convex} and \ref{it:cont_usc}, the limiting process $\mu$ is a solution to PDE \eqref{eq:rfpe} with control $\gamma$ on some suitable probabilistic setup. This is achieved by passing to the limit in a slightly weakened formulation of PDE \eqref{eq:rfpe}. For this to work, we must show that $\mu$ has continuous trajectories as an $L^2_{\eta}(\R^d)$-valued process. Lastly, we establish the statements \eqref{eq:cost_lsc} and \eqref{eq:cost_usc}.

Let $(\bb{W}_n)_n$ and $(\mu^n, \gamma^n, W^n)_n$ be as in the statement of the proposition. By the remarks above the statement of the proposition, we know that the sequence $(\mu^n)_n$ is bounded in $L^2([0, T]; H_{\eta}^1(\R^d))$ and tight on $L^2([0, T]; L_{\eta, \text{loc}}^2(\R^d))$, while $(g^n)_n$ is tight on the space $(L^2([0, T]; L_{\text{loc}}^2(\R^d; \R^{d_G})), \text{weak})$. Thus, selecting a subsequence if necessary, we may assume that the sequence $(\mu^n, \gamma^n, W^n)_n$ converges weakly to a random variable $(\mu, \gamma, W)$, where $\mu$ takes values in $L^2([0, T]; H^1_{\eta}(\R^d))$, $g$ takes values in $L^2([0, T]; L^2_{\text{loc}}(\R^d; \R^{d_G}))$, and $W$ is a $d_W$-dimensional Brownian motion. The random variables $\mu$, $\gamma$, and $W$ can be accommodated on some probabilistic setup $\bb{W} = (\Omega, \F, \bb{F}, \pr, W)$, such that $\gamma$ is $\bb{F}$-progressively measurable. Note that the filtration $\bb{F}$ will in general be larger than the filtration generated by $W$. Moreover, since
\begin{equation*}
    \int_0^T \bigl(\lVert \mu^n_t\rVert_{H_{\eta}^1(\R^d)}^2 + \lVert \dot{\mu}^n_t\rVert_{H^1_{\eta}(\R^d)^{\ast}}^2\bigr) \, \d t \leq C \lVert v_0\rVert_{\eta}^2
\end{equation*}
for all $n \geq 1$ by \eqref{eq:rfpe_estimate} and \eqref{eq:distr_der}, the Portmanteau theorem implies that the same holds for $\mu$, i.e.\@
\begin{equation*}
    \int_0^T \bigl(\lVert \mu_t\rVert_{H_{\eta}^1(\R^d)}^2 + \lVert \dot{\mu}_t\rVert_{H^1_{\eta}(\R^d)^{\ast}}^2\bigr) \, \d t \leq C \lVert v_0\rVert_{\eta}^2.
\end{equation*}
Thus, by \cite[Theorem 3.4]{barbu_ana_contr_1993}, $\mu$ in fact takes values in $C([0, T]; L^2_{\eta}(\R^d))$. Finally, $G$ is a convex set and $\gamma^n_t(x) \in G$ for $(t, \omega, x) \in [0, T] \times \Omega \times \R^d$, so the random variable $\gamma$ actually takes values in $L^2([0, T]; L^2(\R^d; G))$, whence $\gamma \in \bb{G}_{\bb{W}}$.

Next, we will show that under the additional assumptions from \ref{it:linear_convex} and \ref{it:cont_usc}, $\mu$ is the solution to PDE \eqref{eq:rfpe} corresponding to $\gamma$. To achieve this, we pass to the limit in a slightly weaker but nonetheless equivalent formulation of the PDE satisfied by $\mu^n$:
\begin{equation} \label{eq:rfpe_weak}
    \int_0^T \bigl\langle \mu^n_t, \partial_t \varphi_t + \tilde{\L}_{\bb{W}_n}\varphi_t(t, \cdot, \mu^n_t, \gamma^n_t)\bigr\rangle \, \d t = 0
\end{equation}
for all $\varphi \in C^{1, 2}_c([0, T] \times \R^d)$ with initial condition $\mu^n_0 = v_0$. That $\mu_0 = v_0$ follows from the fact that $(\mu^n)_n$ is tight on $H^1([0, T]; H^1_{\eta}(\R^d)^{\ast})$. Next, combining the continuous mapping theorem with Lemma \ref{lem:int_conv} \ref{it:conv_1} under the assumptions from \ref{it:linear_convex} or Lemma \ref{lem:int_conv} \ref{it:conv_3} under the assumptions from \ref{it:cont_usc} allows us to pass to the weak limit in \eqref{eq:rfpe_weak}, so that
\begin{equation*}
    \int_0^T \bigl\langle \mu_t, \partial_t \varphi_t + \tilde{\L}_{\bb{W}}\varphi_t(t, \cdot, \mu_t, \gamma_t)\bigr\rangle \, \d t = 0
\end{equation*}
for all $\varphi \in C^{1, 2}_c([0, T] \times \R^d)$. From that and the continuity of the trajectories of $\mu$, we can easily deduce PDE \eqref{eq:rfpe}.

Lastly, let us derive the semicontinuity statements \eqref{eq:cost_lsc} and \eqref{eq:cost_usc}. Here, we can again make use of Lemma \ref{lem:int_conv}. Under the assumptions from \ref{it:linear_convex} we use Lemma \ref{lem:int_conv} \ref{it:conv_2} while under the assumptions from \ref{it:cont_usc} we apply Lemma \ref{lem:int_conv} \ref{it:conv_4}. Combining these statements with the Portmanteau theorem yields \eqref{eq:cost_lsc} and \eqref{eq:cost_usc}.
\end{proof}

Using Proposition \ref{prop:stability}, we can show that the value $\inf_{\gamma \in \bb{G}_{\bb{W}}} \tilde{J}_{\bb{W}}(\gamma)$ of the control problem is independent of the underlying probabilistic setup $\bb{W}$. We achieve this by showing that any control in $\bb{G}_{\bb{W}}$ can be approximated by a sequence of controls that only depend on the information generated by the noise $W$. To that end, we let $\bb{G}_{\bb{W}}^0$ denote the space of random functions $[0, T] \times \Omega \times \R^d \to G$ which are progressively measurable with respect to the filtration generated by $W$.

\begin{proposition} \label{prop:independent}
Let Assumption \ref{ass:fpe} be satisfied and suppose that for all $(t, x, v) \in [0, T] \times \R^d \times \M(\R^d)$ the map $G \ni g \mapsto b(t, x, v, g)$ is continuous and the map $G \ni g \mapsto f(t, x, v, g)$ is upper semicontinuous. Then $\inf_{\gamma \in \bb{G}_{\bb{W}}^0} \tilde{J}_{\bb{W}}(\gamma) = \inf_{\gamma \in \bb{G}_{\bb{W}}} \tilde{J}_{\bb{W}}(\gamma)$. In particular, $\inf_{\gamma \in \bb{G}_{\bb{W}}} \tilde{J}_{\bb{W}}(\gamma)$ does not depend on the probabilistic setup $\bb{W}$.
\end{proposition}

\begin{proof}
First, we reduce to the case of piecewise constant controls. Let $\gamma \in \bb{G}_{\bb{W}}$. By \cite[Lemma 4.4]{liptser_statistics_rp_1977} there exists a sequence of partitions $0 = t^n_0 < t^n_1 < \dots < t^n_n = T$ for $n \geq 1$ as well as controls $(\gamma^n)_n$, such that $\gamma^n_0$ is a nonrandom element in $\cal{G}$, the space of measurable functions $\R^d \to G$, $\gamma^n$ is constant on the interval $[t^n_i, t^n_{i + 1})$ for $i = 0$,~\ldots, $n - 1$, and for any $R > 0$ we have $\ev \int_0^T \lVert \bf{1}_{B_R}(\gamma^n_t - \gamma_t)\rVert_{L^2}^2 \, \d t \to 0$ as $n \to \infty$. Then Proposition \ref{prop:stability} \ref{it:cont_usc} implies that $\limsup_{n \to \infty} \tilde{J}_{\bb{W}}(\gamma^n) \leq \tilde{J}_{\bb{W}}(\gamma)$ as $n \to \infty$. Consequently, minimising over the set of controls which are piecewise constant in time leads to the same value as minimising over all of $\bb{G}_{\bb{W}}$. Hence, from now on we may assume that $\gamma$ is piecewise constant itself, i.e.\@ there exists a partition $0 = t_0 < t_1 < \dots < t_n = T$ such that $\gamma_0$ is a nonrandom element in $\cal{G}$ and $\gamma$ is constant on the interval $[t_i, t_{i + 1})$ for $i = 0$,~\ldots, $n - 1$. 

The next step is to remove the randomness in $\gamma$ beyond that provided by the noise $W$, replacing it with the randomness of an increasingly small initial segment of the Brownian motion $W$. To achieve this, we must first separate the stochasticity in $\gamma$ originating from $W$ and from the remaining filtration $\bb{F}$. By \cite[Lemma 1.3]{kurtz_weak_strong_2014}, we can inductively find maps $\Phi_i \define C([0, T]; \R^{d_W}) \times C([0, T]; \R^{d_W}) \times [0, 1]^i \to \cal{G}$, for $i = 1$,~\ldots, $n - 1$, such that if $B$ is a $d_W$-dimensional Brownian motion and $\xi_1$,~\ldots, $\xi_{n - 1}$ are uniformly distributed on $[0, 1]$ (all defined on some sufficiently rich probability space $(\Omega', \F', \pr')$), such that $(B, \xi_1, \dots, \xi_{n - 1})$ forms an independent family, then defining
\begin{equation*}
    g_i = \Phi_i\bigl(B_{\cdot \land t_i}, B_{\cdot \lor t_i} - B_{t_i}, \xi_1, \dots, \xi_i\bigr),
\end{equation*}
we have $(B, g_1, \dots, g_{n - 1}) \sim (W, \gamma_{t_1}, \dots, \gamma_{t_{n - 1}})$. Note that since $\gamma$ is adapted to $\bb{F}$ and $W$ is an $\bb{F}$-Brownian motion, it follows that $B_{\cdot \lor t_i} - B_{t_i}$ is independent of $g_i$. We can use this fact to successively replace the increments $B_{\cdot \lor t_i} - B_{t_i}$ in $g_i$ by some external randomness. Indeed, let $\zeta_1$,~\ldots, $\zeta_{n - 1}$ be random variables on $(\Omega', \F', \pr')$, such that $\zeta_i \sim B_{\cdot \lor t_i} - B_{t_i}$ and $(B, \xi_1, \dots, \xi_{n - 1}, \zeta_1, \dots, \zeta_{n - 1})$ forms an independent family. We set
\begin{equation*}
    g'_i = \Phi_i\bigl(B_{\cdot \land t_i}, \zeta_i, \xi_1, \dots, \xi_i\bigr)
\end{equation*}
and claim that $(B, g_1, \dots, g_{n - 1}) \sim (B, g_1, \dots, g_{i - 1}, g'_i, \dots, g'_{n - 1})$ for any $i \in \{1, \dots, n\}$. This is obvious for $i = n$, so for the sake of induction let us assume the claim holds for some $i \leq n$. Then, since $g_{i - 1}$ is independent of $B_{\cdot \lor t_{i - 1}} - B_{t_{i - 1}}$, we can replace the appearance of $B_{\cdot \lor t_{i - 1}} - B_{t_{i - 1}}$ in $g_{i - 1}$ by $\zeta_{i - 1}$ without changing the law of $(B, g_1, \dots, g_{i - 1}, g'_i, \dots, g'_{n - 1})$. However, after this replacement, $g_{i - 1}$ becomes $g'_{i - 1}$, which concludes the induction step and, therefore, proves the claim. Evaluating the claim for $i = 1$ implies that
\begin{equation*}
    (B, g'_1, \dots, g'_{n - 1}) \sim (B, g_1, \dots, g_{n - 1}) \sim (W, \gamma_{t_1}, \dots, \gamma_{t_{n - 1}}).
\end{equation*}

We now want to express all the randomness in the vector $(B, g'_1, \dots, g'_{n - 1})$ through $W$. Let us set $s_k = t_1/k$ for $k \geq 1$ and then define the $d_W$-dimensional Brownian motion $W^k$ by 
\begin{equation*}
    W^k_t =
    \begin{cases}
         \frac{\sqrt{k}}{\sqrt{k - 1}} (W_{s_k + \frac{k - 1}{k} t} - W_{s_k}) &\text{if } t \in [0, t_1], \\
         W_t - W_{t_1} + \frac{\sqrt{k}}{\sqrt{k - 1}} (W_{t_1} - W_{s_k}) &\text{if } t \in [t_1, T].
    \end{cases}
\end{equation*}
Next, we choose a map $\Psi \define \R \to [0, 1]^{n - 1} \times C([0, T]; \R^{d_W})^{n - 1}$ such that
\begin{equation*}
    (\xi^k_1, \dots, \xi^k_{n - 1}, \zeta^k_1, \dots, \zeta^k_{n - 1}) = \Psi\bigl(\tfrac{1}{\sqrt{s_k}}W_{s_k}\bigr) \sim (\xi_1, \dots, \xi_{n - 1}, \zeta_1, \dots, \zeta_{n - 1}).
\end{equation*}
Finally, we define controls $\gamma^k \in \bb{G}$ by $\gamma^k_t = g_0$ if $t \in [0, t_1)$ and $\gamma^k_t = \Phi_i(W^k_{\cdot \land t_i}, \zeta^k_i, \xi^k_1, \dots, \xi^k_i)$ for $t \in [t_i, t_{i + 1})$ with $i \in \{1, \dots, n - 1\}$, so that $(\gamma^k, W^k) \sim (\gamma, W)$. In particular, it holds that the sequence $(\gamma^k, W)_k$ converges weakly to $(\gamma, W)$ on $L^2([0, T]; L_{\text{loc}}^2(\R^d; \R^{d_G})) \times C([0, T]; \R^{d_W})$, so using Proposition \ref{prop:stability} \ref{it:cont_usc} once more we conclude $\limsup_{k \to \infty} \tilde{J}_{\bb{W}}(\gamma^k) \leq \tilde{J}_{\bb{W}}(\gamma)$. Now, for any $\epsilon > 0$ we may choose $\gamma$ such that $\tilde{J}_{\bb{W}}(\gamma) \leq \inf_{\gamma' \in \bb{G}_{\bb{W}}}\tilde{J}_{\bb{W}}(\gamma') + \epsilon$, whence
\begin{equation*}
    \inf_{\gamma' \in \bb{G}_{\bb{W}}^0}\tilde{J}_{\bb{W}}(\gamma') \leq \limsup_{k \to \infty} \tilde{J}_{\bb{W}}(\gamma^k) \leq \inf_{\gamma' \in \bb{G}_{\bb{W}}}\tilde{J}_{\bb{W}}(\gamma') + \epsilon.
\end{equation*}
Letting $\epsilon \to 0$ gives the desired identity.
\end{proof}

Let us now turn to the existence of optimal controls and the proof of Theorem \ref{thm:optimal}. 

\begin{proof}[Proof of Theorem \ref{thm:optimal}]
The first statement of Theorem \ref{thm:optimal} follows immediately from Proposition \ref{prop:independent}. Next, let $\bb{W}$ be any probabilistic setup and fix a sequence of controls $(\gamma^n)_n$ such that $\lim_{n \to \infty} \tilde{J}_{\bb{W}}(\gamma^n) = \inf_{\gamma \in \bb{G}}\tilde{J}_{\bb{W}}(\gamma) = V$ and denote the corresponding solutions to PDE \eqref{eq:rfpe} by $\mu^n$. By Proposition \ref{prop:stability}, the sequence $(\mu^n, \gamma^n, W)_n$ is tight on the space $L^2([0, T]; L_{\eta, \textup{loc}}^2(\R^d)) \times (L^2([0, T]; L_{\textup{loc}}^2(\R^d; \R^{d_G})), \textup{weak}) \times C([0, T]; \R^{d_W})$. Any subsequential limit $(\mu, \gamma, W')$ can be realised on some probabilistic setup $\bb{W}'$ such that $W'$ is the Brownian motion underlying $\bb{W}'$ and $\gamma \in \bb{G}_{\bb{W}'}$. Moreover, by Item \ref{it:linear_convex} of Proposition \ref{prop:stability}, $\mu$ is a solution to PDE \eqref{eq:rfpe} corresponding to the control $\gamma$ and
\begin{equation*}
    V = \liminf_{k \to \infty} \tilde{J}_{\bb{W}}(\gamma^{n_k}) \geq \tilde{J}_{\bb{W}'}(\gamma) \geq V
\end{equation*}
along the convergent subsequence $(n_k)_k$. Thus $\gamma$ is an optimal control on $\bb{W}'$.
\end{proof}

%% file: 4_control/smp.tex
\subsection{The Necessary Stochastic Maximum Principle} \label{sec:smp}

The objective of this subsection is to establish the necessary SMP, Theorem \ref{thm:nsmp}. Throughout the subsection we shall assume that Assumptions \ref{ass:fpe} and \ref{ass:smp} are in place. 

The main step in the proof the necessary SMP is to derive a suitable expression of the Gateaux derivative $\frac{\d}{\d \epsilon} \tilde{J}(\gamma + \epsilon h)\vert_{\epsilon = 0}$ of the cost functional $\tilde{J}$ at a control $\gamma \in \bb{G}$ in some direction $h$ in terms of the adjoint process. The necessary condition then follows by inserting the convexity of the Hamiltonian $H(t, x, v, p, g) := \tilde{b}(t, x, v, g) \cdot p + \tilde{f}(t, x, v, g)$ in the control argument, guaranteed by Assumption \ref{ass:smp} \ref{it:convexity}, into this expression. To compute the Gateaux derivative, we have to study the variation of $\mu$ under small perturbations $\epsilon h$ of the control $\gamma$. Let $h \define [0, T] \times \Omega \times \R^d \to \R^{d_G}$ be a bounded $\bb{F}$-progressively measurable random function such that $\gamma^{\epsilon} = \gamma + \epsilon h$ maps into $G$ for all $\epsilon \geq 0$. Denote the unique solution to PDE \eqref{eq:rfpe} with control $\gamma^{\epsilon} \in \bb{G}$ by $\mu^{\epsilon}$ and define $\cal{V}^{\epsilon} = \frac{1}{\epsilon} (\mu^{\epsilon} - \mu)$. We want to show that $\cal{V}^{\epsilon}$ converges in $C_{\bb{F}}^2([0, T]; L_{\eta}^2(\R^d)) \cap L_{\bb{F}}^2([0, T]; H_{\eta}^1(\R^d))$ to the solution of the random linear PDE
\begin{align} \label{eq:spde_variation}
\begin{split}
    \d \langle \cal{V}_t, \varphi\rangle &= \bigl\langle \cal{V}_t, \tilde{\L}\varphi(t, \cdot, \mu_t, \gamma_t)\big\rangle \, \d t + \bigl\langle \mu_t,  h_t^{\top} \nabla_g \tilde{b}(t, \cdot, \mu_t) \nabla \varphi\bigr\rangle \, \d t \\
    &\ \ \ + \biggl(\int_{\R^d} \Bigl\langle \mu_t, D\tilde{\lambda}(t, \cdot, \mu_t)(x) \varphi + D\tilde{b}(t, \cdot, \mu_t, \gamma_t)(x) \cdot \nabla \varphi \Bigr\rangle \cal{V}_t(x) \, \d x\biggr) \, \d t
\end{split}
\end{align}
for $\varphi \in C_c^2(\R^d)$ with initial condition $\cal{V}_0 = 0$. Note that by Assumption \ref{ass:smp} \ref{it:convexity}, the function $g \mapsto \tilde{b}(t, x, v, g)$ is affine, so $\nabla_g \tilde{b}(t, x, v, g)$ does not depend on $g$ and, therefore, we write $\nabla_g \tilde{b}(t, x, v)$ instead of $\nabla_g \tilde{b}(t, x, v, g)$. We understand $\nabla_g \tilde{b}$ as a function with values in $\R^{d_G \times d}$.

\begin{proposition} \label{prop:variation}
Let Assumptions \ref{ass:fpe} and \ref{ass:smp} be satisfied. Then $(\cal{V}^{\epsilon})_{\epsilon > 0}$ converges in $C_{\bb{F}}^2([0, T]; L_{\eta}^2(\R^d)) \cap L_{\bb{F}}^2([0, T]; H_{\eta}^1(\R^d))$ to the unique solution $\cal{V}$ of PDE \eqref{eq:spde_variation}.
\end{proposition}

\begin{proof}
We will first show that the family $(\cal{V}^{\epsilon > 0})_{\epsilon}$ is Cauchy in the space $C_{\bb{F}}^2([0, T]; L_{\eta}^2(\R^d)) \cap L_{\bb{F}}^2([0, T]; H_{\eta}^1(\R^d))$. It is then straightforward to see that its limit $\cal{V}$ satisfies PDE \eqref{eq:spde_variation}. To prove that $(\cal{V}^{\epsilon})_{\epsilon > 0}$ is a Cauchy sequence, we derive a PDE for $\cal{V}^{\epsilon}$. Then applying the estimate \eqref{eq:sfpe_estimate} to this PDE, we establish the boundedness of the family $(\cal{V}^{\epsilon})_{\epsilon > 0}$ from which we subsequently deduce the Cauchy property.

We set $\mathfrak{l}^{\epsilon}_t(x) = \tilde{\lambda}(t, x, \mu^{\epsilon}_t)$ and $\mathfrak{b}^{\epsilon}_t(x) = \tilde{b}\bigl(t, x, \mu^{\epsilon}_t, \gamma^{\epsilon}_t(x)\bigr)$ for $(t, \omega, x) \in [0, T] \times \Omega \times \R^d$ and $\epsilon \geq 0$. Our first objective is to derive a PDE for $\cal{V}^{\epsilon}$. It holds for $\varphi \in C^2_c(\R^d)$ that
\begin{equation} \label{eq:differential_var}
    \d \langle \cal{V}^{\epsilon}, \varphi\rangle = \bigl\langle \cal{V}^{\epsilon}_t, \tilde{\L}\varphi(t, \cdot, \mu^{\epsilon}_t, \gamma^{\epsilon}_t)\big\rangle \, \d t + \frac{1}{\epsilon}\Bigl\langle \mu_t, (\mathfrak{l}^{\epsilon}_t - \mathfrak{l}^0_t)\varphi + (\mathfrak{b}^{\epsilon}_t - \mathfrak{b}^0_t) \cdot \nabla \varphi \Bigr\rangle \, \d t.
\end{equation}
Let us compute the last term on the right-hand side. We can decompose the difference $\mathfrak{b}^{\epsilon}_t(x) - \mathfrak{b}^0_t(x)$ as
\begin{equation*}
    \bigl(\mathfrak{b}^{\epsilon}_t(x) - \tilde{b}\bigl(t, x, \mu^{\epsilon}_t, \gamma_t(x)\bigr)\bigr) + \bigl(\tilde{b}\bigl(t, x, \mu^{\epsilon}_t, \gamma_t(x)\bigr) - \mathfrak{b}^0_t(x)\bigr).
\end{equation*}
The function $\tilde{b}$ is linear in the control argument by Assumption \ref{ass:smp} \ref{it:convexity}, so it holds that $\mathfrak{b}^{\epsilon}_t(x) - \tilde{b}(t, x, \mu^{\epsilon}_t, \gamma_t(x)) = \epsilon (h_t(x))^{\top}\nabla_g \tilde{b}(t, x, \mu^{\epsilon}_t)$. Next, by Assumption \ref{ass:smp} \ref{it:lfd}, for $(t, x, g) \in [0, T] \times \R^d \times G$ the maps $\M(\R^d) \ni v \mapsto \tilde{\lambda}(t, x, v)$ and $\M(\R^d) \ni v \mapsto \tilde{b}(t, x, v, g)$ have the linear functional derivatives $D\tilde{\lambda}$ and $D\tilde{b}$ given by
\begin{align*}
    D\tilde{\lambda}(t, x, v)(y) &= D\lambda\bigl(t, x + \sigma_0 W_t, (\id + \sigma_0 W_t)^{\#}v\bigr)(y + \sigma_0 W_t), \\
    D\tilde{b}(t, x, v, g)(y) &= Db\bigl(t, x + \sigma_0 W_t, (\id + \sigma_0 W_t)^{\#}v, g\bigr)(y + \sigma_0 W_t)
\end{align*}
for $(t, x, v, g, y) \in [0, T] \times \R^d \times \M(\R^d) \times G \times \R^d \to \R$. Thus, we obtain that
\begin{align} \label{eq:derivative_in_m_lambda}
    \tilde{\lambda}(t, x, \mu^{\epsilon}_t) - \tilde{\lambda}(t, x, \mu_t) &= \int_0^1 \Bigl\langle \mu^{\epsilon}_t - \mu_t, D\tilde{\lambda}\bigl(t, x, \theta \mu^{\epsilon}_t + (1 - \theta)\mu_t\bigr)\Bigr\rangle \, \d \theta \notag \\
    &= \int_0^1 \Bigl\langle \epsilon \cal{V}^{\epsilon}_t, D\tilde{\lambda}\bigl(t, x, \theta \mu^{\epsilon}_t + (1 - \theta)\mu_t\bigr)\Bigr\rangle \, \d \theta
\end{align}
and, similarly,
\begin{equation} \label{eq:derivative_in_m_b}
    \tilde{b}\bigl(t, x, \mu^{\epsilon}_t, \gamma_t(x)\bigr) - \tilde{b}\bigl(t, x, \mu_t, \gamma_t(x)\bigr) = \int_0^1 \Bigl\langle \epsilon \cal{V}^{\epsilon}_t, D\tilde{b}\bigl(t, x, \theta \mu^{\epsilon}_t + (1 - \theta)\mu_t, \gamma_t(x)\bigr)\Bigr\rangle \, \d \theta.
\end{equation}
We set $\Lambda^{\epsilon, \theta}_t(x, y) = D\tilde{\lambda}\bigl(t, x, \theta \mu^{\epsilon}_t + (1 - \theta)\mu_t\bigr)(y)$ and $B^{\epsilon, \theta}_t(x, y) = D\tilde{b}\bigl(t, x, \theta \mu^{\epsilon}_t + (1 - \theta)\mu_t, \gamma_t(x)\bigr)(y)$ for $(t, \omega, x, y) \in [0, T] \times \Omega \times \R^d \times \R^d$. Then, we insert the expressions $\mathfrak{b}^{\epsilon}_t(x) - \tilde{b}(t, x, \mu^{\epsilon}_t, \gamma_t(x)) = \epsilon (h_t(x))^{\top}\nabla_g \tilde{b}(t, x, \mu^{\epsilon}_t)$, \eqref{eq:derivative_in_m_lambda}, and \eqref{eq:derivative_in_m_b} into Equation \eqref{eq:differential_var}, which yields
\begin{align} \label{eq:var_eps_spde}
\begin{split}
    \d \langle \cal{V}^{\epsilon}_t, \varphi\rangle = &\bigl\langle \cal{V}^{\epsilon}_t, \tilde{\L}\varphi(t, \cdot, \mu^{\epsilon}_t, \gamma^{\epsilon}_t)\big\rangle \, \d t + \bigl\langle \mu_t, h_t^{\top}\nabla_g \tilde{b}(t, x, \mu^{\epsilon}_t) \nabla \varphi\bigr\rangle \, \d t \\
    &+ \biggl(\int_0^1\biggl(\int_{\R^d} \Bigl\langle \mu_t, \Lambda^{\epsilon, \theta}_t(\cdot, x)\varphi + B^{\epsilon, \theta}_t(\cdot, x) \cdot \nabla \varphi\Bigr\rangle \cal{V}^{\epsilon}_t(x)\, \d x\biggr) \d \theta\biggr) \, \d t
\end{split}
\end{align}
for $\varphi \in C^2_c(\R^d)$. 

Next, we derive a uniform bound on $\sup_{0 \leq t \leq T} \lVert \cal{V}^{\epsilon}_t\rVert_{\eta}^2$. Clearly, the PDE \eqref{eq:var_eps_spde} has the same form as PDE \eqref{eq:sfpe_random} with $\alpha$, $\beta$, $h^0$, and $h^1$ replaced by $\mathfrak{l}^{\epsilon}$, $\mathfrak{b}^{\epsilon}$, $\mathfrak{h}^{\epsilon, 0}$, and $\mathfrak{h}^{\epsilon, 1}$, respectively, where the latter two are defined as
\begin{align*}
    \mathfrak{h}^{\epsilon, 0}_t(x) &= \mu_t(x) \int_0^1 \bigl\langle \Lambda^{\epsilon, \theta}_t(x, \cdot), \cal{V}^{\epsilon}_t\bigr\rangle \, \d \theta, \\
    \mathfrak{h}^{\epsilon, 1}_t(x) &= \mu_t(x) (h_t(x))^{\top}\nabla_g \tilde{b}(t, x, \mu^{\epsilon}_t) + \mu_t(x) \int_0^1 \bigl\langle B^{\epsilon, \theta}_t(x, \cdot), \cal{V}^{\epsilon}_t\bigr\rangle \, \d \theta.
\end{align*}
We will now apply the estimate \eqref{eq:sfpe_estimate} to $\cal{V}^{\epsilon}$ and then further bound the terms appearing on the right-hand side of this estimate. Since $\bigl\lvert \bigl\langle \Lambda^{\epsilon, \theta}_t(x, \cdot), \cal{V}^{\epsilon}_t\bigr\rangle\bigr\rvert \leq C_{\lambda} \lVert e^{-\eta/2}\rVert_{L^2} \lVert \cal{V}^{\epsilon}_t\rVert_{\eta}$ and $\bigl\lvert \bigl\langle B^{\epsilon, \theta}_t(x, \cdot), \cal{V}^{\epsilon}_t\bigr\rangle\bigr\rvert \leq C_{b} \lVert e^{-\eta/2}\rVert_{L^2} \lVert \cal{V}^{\epsilon}_t\rVert_{\eta}$, we get $\lvert \cal{V}^{\epsilon}_t(x) \mathfrak{l}^{\epsilon}_t(x)\rvert + \lvert \mathfrak{h}^{\epsilon, 0}_t(x)\rvert \leq C_{\lambda} \lvert \cal{V}^{\epsilon}_t(x)\rvert + C_{\lambda}\lVert e^{-\eta/2}\rVert_{L^2} \lVert \cal{V}^{\epsilon}_t\rVert_{\eta} \mu_t(x)$ and
\begin{equation*}
    \lvert \cal{V}^{\epsilon}_t(x) \mathfrak{b}^{\epsilon}_t(x)\rvert + \lvert \mathfrak{h}^{\epsilon, 1}_t(x)\rvert \leq C_b \lvert \cal{V}^{\epsilon}_t(x)\rvert + C_{b, G} \bigl(1 + \lVert e^{-\eta/2}\rVert_{L^2} \lVert \cal{V}^{\epsilon}_t\rVert_{\eta}\bigr) \mu_t(x).
\end{equation*}
Inserting this into the bound provided by \eqref{eq:sfpe_estimate} yields
\begin{equation} \label{eq:var_eps_est}
    \sup_{0 \leq s \leq t} \lVert \cal{V}^{\epsilon}_s\rVert_{\eta}^2 + \int_0^t \lVert \nabla \cal{V}^{\epsilon}_s\rVert_{\eta}^2 \, \d s \leq C \int_0^t \Bigl(1 + \lVert \cal{V}^{\epsilon}_s\rVert_{\eta}^2 + (1 + \lVert \cal{V}^{\epsilon}_s\rVert_{\eta}) \lVert \nabla \cal{V}^{\epsilon}_s\rVert_{\eta}\Bigr) \, \d s
\end{equation}
for a constant $C > 0$ depending on $\mu$, $\eta_0$, the bounds on the coefficients $\lambda$ and $b$, and $\lVert h\rVert_{L^{\infty}} \leq \sup_{y \in G} \lvert y\rvert$. Using the estimate $(1 + \lVert \cal{V}^{\epsilon}_s\rVert_{\eta}) \lVert \nabla \cal{V}^{\epsilon}_s\rVert_{\eta} \leq \frac{C}{2} (1 + \lVert \cal{V}^{\epsilon}_s\rVert_{\eta})^2 + \frac{1}{2C} \lVert \nabla \cal{V}^{\epsilon}_s\rVert_{\eta}^2$, rearranging \eqref{eq:var_eps_est}, and applying Gr\"onwall's inequality implies that 
\begin{equation*}
    \sup_{0 \leq t \leq T} \lVert \cal{V}^{\epsilon}_t\rVert_{\eta}^2 + \frac{1}{2}\int_0^T \lVert \nabla \cal{V}^{\epsilon}_t\rVert_{\eta}^2 \, \d t \leq C
\end{equation*}
for a possibly enlarged constant $C$ independent of $\epsilon$.

Using the uniform bound on the family $(\cal{V}^{\epsilon})_{\epsilon > 0}$, we can show that it is Cauchy in $C_{\bb{F}}^2([0, T]; L^2_{\eta}(\R^d)) \cap L_{\bb{F}}^2([0, T]; H^1_{\eta}(\R^d))$. Indeed, let us define $\Delta^{\epsilon, \delta} = \cal{V}^{\epsilon} - \cal{V}^{\delta}$. Note that again, the PDE satisfied by $\Delta^{\epsilon, \delta}$ is of the same form as PDE \eqref{eq:sfpe_random}. This time with $\alpha$, $\beta$, $h^0$, and $h^1$ replaced by $\mathfrak{l}^{\epsilon}$, $\mathfrak{b}^{\epsilon}$, $\mathfrak{h}^{\epsilon, \delta, 0}$, and $\mathfrak{h}^{\epsilon, \delta, 1}$, respectively, where
\begin{equation*}
    \mathfrak{h}^{\epsilon, \delta, 0}_t(x) = \cal{V}^{\delta}_t(x) \bigl(\mathfrak{l}^{\epsilon}_t(x) - \mathfrak{l}^{\delta}_t(x)\bigr) + \mathfrak{h}^{\epsilon, 0}_t(x) - \mathfrak{h}^{\delta, 0}_t(x)
\end{equation*}
and
\begin{equation*}
    \mathfrak{h}^{\epsilon, \delta, 1}_t(x) = \cal{V}^{\delta}_t(x) \bigl(\mathfrak{b}^{\epsilon}_t(x) - \mathfrak{b}^{\delta}_t(x)\bigr) +  \mathfrak{h}^{\epsilon, 1}_t(x) - \mathfrak{h}^{\delta, 1}_t(x).
\end{equation*}
Once more, we make us of the estimate \eqref{eq:sfpe_estimate} to show that the expression $\sup_{0 \leq t \leq T} \lVert \Delta^{\epsilon, \delta}_t\rVert_{\eta}^2 + \int_0^T \lVert \nabla \Delta^{\epsilon, \delta}_t\rVert_{\eta}^2 \, \d t$ vanishes as $\epsilon$, $\delta \to 0$. The estimate reads
\begin{align} \label{eq:sfpe_estimate_1}
\begin{split}
    \sup_{0 \leq s \leq t} \lVert \Delta^{\epsilon, \delta}_s\rVert_{\eta}^2 + \int_0^t \lVert \nabla \Delta^{\epsilon, \delta}_s\rVert_{\eta}^2 \, \d s &\leq C \int_0^t \bigl\lvert \bigl\langle \Delta^{\epsilon, \delta}_s (\mathfrak{l}^{\epsilon}_s + \lvert \mathfrak{b}^{\epsilon}_s\rvert) + \mathfrak{h}^{\epsilon, \delta, 0}_s + \lvert \mathfrak{h}^{\epsilon, \delta, 1}_s\rvert, \Delta^{\epsilon, \delta}_s\bigr\rangle_{\eta}\bigr\rvert \, \d s \\
    &\ \ \ + C \int_0^t \bigl\lvert\bigl\langle \Delta^{\epsilon, \delta}_s \mathfrak{b}^{\epsilon}_s + \mathfrak{h}^{\epsilon, \delta, 1}_s, \nabla \Delta^{\epsilon, \delta}_s\bigr\rangle_{\eta}\bigr\rvert \, \d s.
\end{split}
\end{align}
Let us first derive bounds on the functions $\mathfrak{h}^{\epsilon, \delta, 0}$ and $\mathfrak{h}^{\epsilon, \delta, 1}$. Note that
\begin{align*}
    \mathfrak{h}^{\epsilon, 0}_s(x) - \mathfrak{h}^{\delta, 0}_s(x) = \mu_s(x) \int_0^1 \bigl\langle \Lambda^{\epsilon, \theta}_s(x, \cdot), \Delta^{\epsilon, \delta}_s\bigr\rangle \, \d \theta + \mu_s(x) \int_0^1 \bigl\langle \Lambda^{\epsilon, \theta}_s(x, \cdot) - \Lambda^{\delta, \theta}_s(x, \cdot), \cal{V}^{\delta}_s\bigr\rangle \, \d \theta,
\end{align*}
so that
\begin{align*}
    \lvert \mathfrak{h}^{\epsilon, \delta, 0}_s(x)\rvert &\leq C_{\lambda}  d_0(\mu^{\epsilon}_s, \mu^{\delta}_s) \lvert \cal{V}^{\delta}_s(x)\rvert + C_{\lambda} \lVert \Delta^{\epsilon, \delta}_s\rVert_{\eta} \mu_s(x) \\
    &\ \ \ + \lVert \cal{V}^{\delta}_s\rVert_{\eta} \mu_s(x) \int_0^1 \bigl\lVert e^{-\eta/2} \bigl(\Lambda^{\epsilon, \theta}_s(x, \cdot) - \Lambda^{\delta, \theta}_s(x, \cdot)\bigr)\bigr\rVert_{L^2} \, \d \theta \\
    &\leq C_{\lambda} \lVert e^{-\eta/2}\rVert_{L^2} \lVert \mu^{\epsilon}_s - \mu^{\delta}_s\rVert_{\eta} \lvert \cal{V}^{\delta}_s(x)\rvert + \bigl(C_{\lambda} \lVert \Delta^{\epsilon, \delta}_s\rVert_{\eta} + \lVert \cal{V}^{\delta}_s\rVert_{\eta} I^{\epsilon, \delta, 0}_s(x)\bigr) \mu_s(x),
\end{align*}
where $I^{\epsilon, \delta, 0}_s(x) = \int_0^1 \bigl\lVert e^{-\eta/2} \bigl(\Lambda^{\epsilon, \theta}_s(x, \cdot) - \Lambda^{\delta, \theta}_s(x, \cdot)\bigr)\bigr\rVert_{L^2} \, \d \theta$. Similarly, we obtain
\begin{align*}
    \lvert \mathfrak{h}^{\epsilon, \delta, 1}_s(x)\rvert &\leq C_{b, G} \bigl(\lVert e^{-\eta/2}\rVert_{L^2} \lVert \mu^{\epsilon}_s - \mu^{\delta}_s\rVert_{\eta} + \lvert \epsilon - \delta\rvert\bigr)\lvert \cal{V}^{\delta}_s(x)\rvert \\
    &\ \ \ + \bigl(C_{b, G} \lVert \Delta^{\epsilon, \delta}_s\rVert_{\eta} + \lVert \cal{V}^{\delta}_s\rVert_{\eta} I^{\epsilon, \delta, 1}_s(x)\bigr) \mu_s(x)
\end{align*}
with $I^{\epsilon, \delta, 1}_s(x) = \int_0^1 \bigl\lVert e^{-\eta/2} \bigl(B^{\epsilon, \theta}_s(x, \cdot) - B^{\delta, \theta}_s(x, \cdot)\bigr)\bigr\rVert_{L^2} \, \d \theta$. Plugging the previous two displays into \eqref{eq:sfpe_estimate_1}, we find
\begin{align*}
    \sup_{0 \leq s \leq t} &\lVert \Delta^{\epsilon, \delta}_s\rVert_{\eta}^2 + \int_0^t \lVert \nabla \Delta^{\epsilon, \delta}_s\rVert_{\eta}^2 \, \d s \\
    &\leq C\int_0^t \bigl(\lVert \Delta^{\epsilon, \delta}_s\rVert_{\eta} + \lVert \nabla \Delta^{\epsilon, \delta}_s\rVert_{\eta}\bigr) \Bigl(\lVert \Delta^{\epsilon, \delta}_s\rVert_{\eta} + \lVert \mu^{\epsilon}_s - \mu^{\delta}_s\rVert_{\eta} + \lvert \epsilon - \delta\rvert\Bigr) \, \d s \\
    &\ \ \ + C\int_0^t \bigl(\lVert \Delta^{\epsilon, \delta}_s\rVert_{\eta} + \lVert \nabla \Delta^{\epsilon, \delta}_s\rVert_{\eta}\bigr) \bigl\lVert (I^{\epsilon, \delta, 0}_s + I^{\epsilon, \delta, 1}_s)\mu_s\bigr\rVert_{\eta} \, \d s \\
    &\leq \int_0^t \frac{C + 2 C^2}{2} \Bigl(3\lVert \mu^{\epsilon}_s - \mu^{\delta}_s\rVert_{\eta}^2 + 3\lvert \epsilon - \delta\rvert^2 + \bigl\lVert (I^{\epsilon, \delta, 0}_s + I^{\epsilon, \delta, 1}_s)\mu_t\bigr\rVert_{\eta}^2\Bigr) \, \d s \\
    &\ \ \ + \int_0^t (C + 3C^2)\lVert \Delta^{\epsilon, \delta}_s\rVert_{\eta}^2 + \frac{1}{2}\lVert \nabla \Delta^{\epsilon, \delta}_s\rVert_{\eta}^2 \, \d s
\end{align*}
for a constant $C > 0$ independent of $\epsilon$, $\delta > 0$. Again, we rearrange and apply Gr\"onwall's inequality, whereby
\begin{equation*}
    \sup_{0 \leq t \leq T} \lVert \Delta^{\epsilon, \delta}_t\rVert_{\eta}^2 + \int_0^T \lVert \nabla \Delta^{\epsilon, \delta}_t\rVert_{\eta}^2 \, \d t \leq C\int_0^T \Bigl(\lVert \mu^{\epsilon}_t - \mu^{\delta}_t\rVert_{\eta}^2 + \lvert \epsilon - \delta\rvert^2 + \bigl\lVert (I^{\epsilon, \delta, 0}_t + I^{\epsilon, \delta, 1}_t)\mu_t\bigr\rVert_{\eta}^2\Bigr) \, \d t,
\end{equation*}
where we enlarge $C$ if necessary. We claim that the three expressions on the right-hand side vanish as $\epsilon$, $\delta \to 0$. Firstly, it holds that $\lVert \mu^{\epsilon}_t - \mu^{\delta}_t\rVert_{\eta}^2 \leq \lVert \epsilon \cal{V}^{\epsilon}_t - \delta \cal{V}^{\delta}_t\lVert_{\eta}^2$, which tends to zero due to the uniform in $\epsilon$ bound on $\sup_{0 \leq s \leq T} \lVert \cal{V}^{\epsilon}_s\rVert_{\eta}^2$. The convergence of the second term is clear. For the third term, we note that owing to the continuity of $D\tilde{\lambda}$ and $D\tilde{b}$ in the measure argument, we have $I^{\epsilon, \delta, 0}_t(x) + I^{\epsilon, \delta, 1}_t(x) \to 0$ for a.e.\@ $(t, x) \in [0, T] \times \R^d$. Since $I^{\epsilon, \delta, 0}_t(x)$ and $I^{\epsilon, \delta, 1}_t(x)$ are bounded uniformly over $(t, x) \in [0, T] \times \R^d$, the dominated convergence theorem implies that $\bigl\lVert (I^{\epsilon, \delta, 0}_t + I^{\epsilon, \delta, 1}_t)\mu_t\bigr\rVert_{\eta} \to 0$. Hence, $\sup_{0 \leq t \leq T} \lVert \Delta^{\epsilon, \delta}_t\rVert_{\eta}^2 + \int_0^T \lVert \nabla \Delta^{\epsilon, \delta}_t\rVert_{\eta}^2 \, \d t \to 0$ as $\epsilon$, $\delta \to 0$, so that $(\cal{V}^{\epsilon})_{\epsilon > 0}$ is Cauchy in $C_{\bb{F}}^2([0, T]; L^2_{\eta}(\R^d)) \cap L_{\bb{F}}^2([0, T]; H^1_{\eta}(\R^d))$.

Let $\cal{V}$ denote the limit of $(\cal{V}^{\epsilon})_{\epsilon > 0}$ in $C_{\bb{F}}^2([0, T]; L^2_{\eta}(\R^d)) \cap L_{\bb{F}}^2([0, T]; H^1_{\eta}(\R^d))$. Using the convergence of $\cal{V}^{\epsilon}$ to $\cal{V}$, we can pass to the limit in PDE \eqref{eq:var_eps_spde} to see that $\cal{V}$ solves PDE \eqref{eq:spde_variation}. Uniqueness for PDE \eqref{eq:spde_variation} can be easily deduced from \eqref{eq:sfpe_estimate}. This concludes the proof.
\end{proof}

Next, we use the adjoint process $(u, m)$, characterised by BSPDE \eqref{eq:adjoint}, and the variation $\cal{V}$ to derive an expression for the derivative of the cost function $\tilde{J}$ defined in \eqref{eq:cost_rfpe}. Inserting the convexity of $g \mapsto H(t, x, v, p, g) := \tilde{b}(t, x, v, g) \cdot p + \tilde{f}(t, x, v, g)$ implied by Assumption \ref{ass:smp} \ref{it:convexity} into this expression then yields the necessary SMP.

\begin{proposition} \label{eq:cost_derivative}
Let Assumptions \ref{ass:fpe} and \ref{ass:smp} be satisfied. Then for any control $\gamma \in \bb{G}$ and any $\bb{F}$-progressively measurable random function $h \define [0, T] \times \Omega \times \R^d \to \R^{d_G}$, such that $\gamma + \epsilon h$ maps into $G$, it holds that
\begin{equation} \label{eq:gateaux_derivative}
    \frac{\d}{\d \epsilon} \tilde{J}(\gamma + \epsilon h)\Big\vert_{\epsilon = 0} = \ev \int_0^T \bigl\langle \mu_t, \nabla_g H(t, \cdot, \mu_t, \nabla u_t, \gamma_t) \cdot h_t\bigr\rangle \, \d t.
\end{equation}
\end{proposition}

\begin{proof}
Set $\gamma^{\epsilon} = \gamma + \epsilon h$ and denote the solutions to PDE \eqref{eq:rfpe} corresponding to the inputs $\gamma^{\epsilon}$ and $\gamma$ by $\mu^{\epsilon}$ and $\mu$, respectively. Further, let $\cal{V}^{\epsilon} = \frac{1}{\epsilon}(\mu^{\epsilon} - \mu)$ and let $\cal{V}$ denote the limit of $(\cal{V}^{\epsilon})_{\epsilon > 0}$ in $C_{\bb{F}}^2([0, T]; L^2_{\eta}(\R^d)) \cap L_{\bb{F}}^2([0, T]; H^1_{\eta}(\R^d))$, which exists by Proposition \ref{prop:variation}. Then, we decompose
\begin{align} \label{eq:decomp_cost}
\begin{split}
    \tilde{J}(g^{\epsilon}&) - \tilde{J}(g) \\
    &= \ev \int_0^T \bigl\langle \mu^{\epsilon}_t - \mu_t, \tilde{f}(t, \cdot, \mu^{\epsilon}_t, \gamma^{\epsilon}_t)\bigr\rangle \, \d t + \ev\int_0^T \bigl\langle \mu_t, \tilde{f}(t, \cdot, \mu^{\epsilon}_t, \gamma^{\epsilon}_t) - \tilde{f}(t, \cdot, \mu^{\epsilon}_t, \gamma_t)\bigr\rangle \, \d t \\
    &\ \ \ + \ev\int_0^T \bigl\langle \mu_t, \tilde{f}(t, \cdot, \mu^{\epsilon}_t, \gamma_t) - \tilde{f}(t, \cdot, \mu_t, \gamma_t)\bigr\rangle \, \d t + \ev[\tilde{\psi}(\mu^{\epsilon}_T) - \tilde{\psi}(\mu_T)].
\end{split}
\end{align}
We divide the equation by $\epsilon$ and then take the limit as $\epsilon$ tends to zero. Owing to Assumptions \ref{ass:fpe} \ref{it:continuity_fpe} and \ref{ass:smp} \ref{it:cost_cntrl_diff} as well as Proposition \ref{prop:variation}, we can pass to the limit in the first expression on the right-hand side above, which is given by $\bigl\langle \cal{V}_t, \tilde{f}(t, \cdot, \mu_t, \gamma_t)\bigr\rangle$. For the second term, we use Assumption \ref{ass:smp} \ref{it:cost_cntrl_diff}, whereby
\begin{align*}
    \frac{1}{\epsilon}\Bigl(\tilde{f}\bigl(t, x, \mu^{\epsilon}_t, \gamma^{\epsilon}_t(x)\bigr) - \tilde{f}\bigl(t, x, \mu^{\epsilon}_t, \gamma_t(x)\bigr)\Bigr) = \int_0^1 \nabla_g f\bigl(t, x, \mu^{\epsilon}_t, \gamma_t(x) + \epsilon \theta h_t(x)\bigr) \cdot h_t(x) \, \d \theta,
\end{align*}
where $\nabla_g \tilde{f}$ denotes the shift of $\nabla_g f$ along $W$. Since $(v, g) \mapsto \nabla_g f(t, x, v, g)$ is continuous and bounded by Assumption \ref{ass:smp} \ref{it:cost_cntrl_diff}, we conclude by the dominated convergence theorem that the second expression tends to $\ev\int_0^T \bigl\langle \mu_t, \nabla_g \tilde{f}(t, \cdot, \mu_t, \gamma_t) \cdot h_t\bigr\rangle \, \d t$. Next, by Assumption \ref{ass:smp} \ref{it:lfd}, for any $(t, x, g) \in [0, T] \times \R^d \times G$ the map $v \mapsto f(t, x, v, g)$ has a linear functional derivative, which is bounded and continuous in the measure argument by Assumption \ref{ass:smp} \ref{it:lfd}. Hence, we can write
\begin{align*}
    \frac{1}{\epsilon}\Bigl(\tilde{f}\bigl(t, x, \mu^{\epsilon}_t, \gamma_t(x)\bigr) - \tilde{f}\bigl(t, x, \mu_t, \gamma_t(x)\bigr)\Bigr) &= \int_0^1 \Bigl\langle \cal{V}^{\epsilon}_t, D\tilde{f}\bigl(t, x, \theta \mu^{\epsilon}_t + (1 - \theta)\mu_t, \gamma_t(x)\bigr)\Bigr\rangle \, \d \theta.
\end{align*}
Again, the dominated convergence theorem allows us to pass to the limit, so the third term on the right-hand side of \eqref{eq:decomp_cost} tends to $\ev\int_0^T \bigl(\int_{\R^d} \bigl\langle \mu_t, D\tilde{f}(t, \cdot, \mu_t, \gamma_t)(x)\bigr\rangle \cal{V}_t(x) \, \d x\bigr) \, \d t$. The same reasoning allows us to take the limit as $\epsilon \to 0$ in the last term on the right-hand side of Equation \eqref{eq:decomp_cost}. Putting everything together, we obtain
\begin{align} \label{eq:gateaux_der_1}
\begin{split}
    \frac{\d}{\d \epsilon} &\tilde{J}(\gamma + \epsilon h)\Big\vert_{\epsilon = 0} \\
    &= \ev\int_0^T \bigl\langle \cal{V}_t, \tilde{f}(t, \cdot, \mu_t, \gamma_t)\bigr\rangle \, \d t + \ev\int_0^T \bigl\langle \mu_t, \nabla_g \tilde{f}(t, \cdot, \mu_t, \gamma_t) \cdot h_t\bigr\rangle \, \d t \\
    &\ \ \ + \ev\int_0^T \biggl(\int_{\R^d} \bigl\langle \mu_t, D\tilde{f}(t, \cdot, \mu_t, \gamma_t)(x)\bigr\rangle \cal{V}_t(x) \, \d x\biggr) \, \d t + \ev \bigl\langle \cal{V}_t, D\tilde{\psi}(\mu_T)\bigr\rangle.
\end{split}
\end{align}

Next, we apply the generalisation of It\^o's formula from Corollary \ref{cor:ito} to $\langle \cal{V}_t, u_t\rangle$, where $(u, m)$ is the adjoint process, solving BSPDE \eqref{eq:adjoint}. This implies that
\begin{align*}
    &\bigl\langle \cal{V}_T, D\tilde{\psi}(\mu_T)\bigr\rangle \notag \\
    &= \langle \cal{V}_0, u_0\rangle + \int_0^T \biggl(\Bigl\langle \cal{V}_t, \tilde{\lambda}(t, \cdot, \mu_t) u_t + \tilde{b}(t, \cdot, \mu_t, \gamma_t) \cdot \nabla u_t\Bigr\rangle - \bigl\langle  \nabla \cdot (\tilde{a}(t, \cdot) \cal{V}_t), \nabla u_t\big\rangle\biggr) \, \d t \notag \\
    &\ \ \ + \int_0^T \biggl(\int_{\R^d} \Bigl\langle \mu_t, D\tilde{\lambda}(t, \cdot, \mu_t)(x) u_t + D\tilde{b}(t, \cdot, \mu_t, \gamma_t)(x) \cdot \nabla u_t \Bigr\rangle \cal{V}_t(x) \, \d x\biggr) \, \d t \notag \\
    &\ \ \ + \int_0^T \bigl\langle \mu_t, h_t^{\top} \nabla_g \tilde{b}(t, \cdot, \mu_t) \nabla u_t \bigr\rangle \, \d t \notag \\
    &\ \ \ - \int_0^T \biggl(\Bigl\langle \tilde{\lambda}(t, \cdot, \mu_t) u_t + \tilde{b}(t, \cdot, \mu_t, \gamma_t) \cdot \nabla u_t + \tilde{f}(t, \cdot, \mu_t, \gamma_t), \cal{V}_t\Bigr\rangle - \bigl\langle \nabla u_t, \nabla \cdot (\tilde{a}(t, \cdot) \cal{V}_t)\bigr\rangle\biggr) \, \d t \notag \\
    & \ \ \ - \int_0^T \biggl(\int_{\R^d} \Bigl\langle \mu_t, D\tilde{\lambda}(t, \cdot, \mu_t)(x) u_t + D\tilde{b}(t, \cdot, \mu_t, \gamma_t)(x) \cdot \nabla u_t \cal{V}_t(x) \, \d x\biggr) \, \d t \notag \\
    &\ \ \ - \int_0^T \biggl(\int_{\R^d} \bigl\langle \mu_t, D\tilde{f}(t, \cdot, \mu_t, \gamma_t)(x) \bigr\rangle \cal{V}_t(x) \, \d x\biggr) \, \d t \notag \\
    & = \int_0^T \biggl(\bigl\langle \mu_t, h_t^{\top} \nabla_g \tilde{b}(t, \cdot, \mu_t) \nabla u_t \bigr\rangle - \bigl\langle \cal{V}_t, \tilde{f}(t, \cdot, \mu_t, \gamma_t)\bigr\rangle\, \d x\biggr) \, \d t \\
    &\ \ \ - \int_0^T \biggl(\int_{\R^d} \bigl\langle \mu_t, D\tilde{f}(t, \cdot, \mu_t, \gamma_t)(x)\bigr\rangle \cal{V}_t(x) \, \d x\biggr) \, \d t,
\end{align*}
where we used that $\cal{V}_0 = 0$. Now, we take expectations on both sides of the above equality and then plug the resulting expression for $\ev \bigl\langle \cal{V}_T, D\tilde{\psi}(\mu_T)\bigr\rangle$ into Equation \eqref{eq:gateaux_der_1}. This yields the desired identity \eqref{eq:gateaux_derivative}.
\end{proof}

From \eqref{eq:gateaux_derivative} and the convexity of $H$ in the control argument, guaranteed by Assumption \ref{ass:smp} \ref{it:convexity}, we can deduce the necessary SMP.

\begin{proof}[Proof of Theorem \ref{thm:nsmp}]
Let $\gamma$, $\gamma' \in \bb{G}$. Since $G$ is convex, for any $\epsilon > 0$ the random function $\gamma^{\epsilon} = (1 - \epsilon) \gamma + \epsilon \gamma' = \gamma + \epsilon (\gamma' - \gamma)$ maps into $G$. Hence, we can apply Proposition \ref{eq:cost_derivative} with $h$ replaced by $\gamma' - \gamma$, which implies that
\begin{align*}
    \frac{\d}{\d \epsilon} \tilde{J}(\gamma^{\epsilon})\Big\vert_{\epsilon = 0} &= \ev \int_0^T \bigl\langle \mu_t, \nabla_g H(t, \cdot, \mu_t, \nabla u_t, \gamma_t) \cdot (\gamma'_t - \gamma_t)\bigr\rangle \, \d t \\
    &\leq \ev \int_0^T \bigl\langle \mu_t, H(t, \cdot, \mu_t, \nabla u_t, \gamma'_t) - H(t, \cdot, \mu_t, \nabla u_t, \gamma_t)\bigr\rangle \, \d t,
\end{align*}
where we used that the function $g \mapsto H(t, x, v, p, g)$ is convex by Assumption \ref{ass:smp} \ref{it:convexity}. On the other hand, we know that $\gamma$ is a minimiser of the cost functional $\tilde{J}$, so that
\begin{equation*}
    \frac{\d}{\d \epsilon} \tilde{J}(\gamma^{\epsilon})\Big\vert_{\epsilon = 0} = \lim_{\epsilon \to 0}\frac{1}{\epsilon} \bigl(\tilde{J}(\gamma + \epsilon(\gamma' - \gamma)) - \tilde{J}(\gamma)\bigr) \geq 0.
\end{equation*}
Combining the previous two inequalities yields
\begin{equation} \label{eq:conv_ineq_cost}
    \ev \int_0^T \bigl\langle \mu_t, H(t, \cdot, \mu_t, \nabla u_t, \gamma'_t) - H(t, \cdot, \mu_t, \nabla u_t, \gamma_t)\bigr\rangle \, \d t \geq 0.
\end{equation}
Now, for any $\bb{F}$-progressively measurable set $A \subset [0, T] \times \Omega \times \R^d$ and any $g \in G$ we can define a control $\gamma' \in \bb{G}$ by $\gamma'_t(x) = g$ if $(t, \omega, x) \in A$ and $\gamma'_t(x) = \gamma_t(x)$ otherwise. Applying \eqref{eq:conv_ineq_cost} with this choice of $\gamma'$ implies that
\begin{equation*}
    \ev\int_0^T \biggl(\int_{\R^d} \bf{1}_A(t, \cdot, x) \Bigl(H\bigl(t, x, \mu_t, \nabla u_t(x), g\bigr) - H\bigl(t, x, \mu_t, \nabla u_t(x), \gamma_t(x)\bigr)\Bigr) \, \d \mu_t(x)\biggr) \, \d t
\end{equation*}
is nonnegative. Since $A$ and $g$ were arbitrary, we conclude that for all $g \in G$ it holds that
\begin{equation*}
    H\bigl(t, x, \mu_t, \nabla u_t(x), g\bigr) \geq H\bigl(t, x, \mu_t, \nabla u_t(x), \gamma_t(x)\bigr)
\end{equation*}
for $\mu_t$-a.e.\@ $x \in \R^d$ and $\leb \otimes \pr$-a.e.\@ $(t, \omega) \in [0, T] \times \Omega$.
\end{proof}

\subsection{The Sufficient Stochastic Maximum Principle}

Lastly, we shall establish the sufficient SMP under Assumptions \ref{ass:fpe}, \ref{ass:smp} \ref{it:lfd}, and \ref{ass:separability}. The idea of its proof is to compare the cost of the control $\gamma \in \bb{G}$ from the statement of the sufficient SMP, Theorem \ref{thm:ssmp}, with that of any other control, and then to insert the convexity assumption on the functions $\psi$ and $\cal{H}_0$. 

By Lemma \ref{lem:ldf_convex}, for any convex function $F \define \M(\R^d) \to \R$, which has a linear functional derivative $DF$ such that for all $v$, $v' \in \M(\R^d)$ it holds that $DF(\theta v + (1 - \theta)v')(x) \to DF(v')(x)$ as $\theta \in (0, 1]$ tends to zero, we have $F(v) - F(v') \geq \langle v - v', DF(v')\rangle$. 
Under Assumptions \ref{ass:smp} \ref{it:lfd} and \ref{ass:separability}, the function $v \mapsto \cal{H}_0(t, v, r, p)$ satisfies the above continuity condition for any $(t, \omega, r, p) \in [0, T] \times \Omega \times \B_b(\R^d) \times \B_b(\R^d; \R^d)$, so if $v \mapsto \cal{H}_0(t, v, r, p)$ is convex in the measure argument, then
\begin{equation*}
    \cal{H}_0(t, v, r, p) - \cal{H}_0(t, v', r, p) \geq \bigl\langle v - v', D\cal{H}_0(t, v', r, p)\bigr\rangle.
\end{equation*}
This inequality readily extends to $v$, $v' \in L^2_{\eta}(\R^d; [0, \infty))$, $r \in L^2_{\bar{\eta}}(\R^d)$, and $p \in L^2_{\bar{\eta}}(\R^d; \R^d)$. Consequently, writing $\cal{H}(t, v, r, p, g) = \cal{H}_0(t, v, r, p) + \bigl\langle v, H_1(t, \cdot, p, g)\bigr\rangle$, if $g'$ is any other element of $\cal{G}$, we obtain that
\begin{align} \label{eq:convexity_expression}
    &\cal{H}(t, v, r, p, g) - \cal{H}(t, v', r, p, g') - \bigl\langle v - v', D\cal{H}(t, v, r, p, g)\bigr\rangle \notag \\
    &= \cal{H}_0(t, v, r, p) - \cal{H}_0(t, v', r, p) - \bigl\langle v - v', D\cal{H}_0(t, v, r, p)\bigr\rangle \notag \\
    &\ \ \ + \bigl\langle v, H_1(t, \cdot, p, g)\bigr\rangle - \bigl\langle v', H_1(t, \cdot, p, g')\bigr\rangle - \bigl\langle v - v', H_1(t, \cdot, p, g)\bigr\rangle \notag \\
    &\leq \bigl\langle v', H_1(t, \cdot, p, g) - H_1(t, \cdot, p, g')\bigr\rangle.
\end{align}
Now, if $g \in \cal{G}$ is a pointwise minimiser of the map $h \mapsto H_1(t, x, p, h)$, then clearly the expression on the right-hand side is nonpositive. This is the key insight in the proof of Theorem \ref{thm:ssmp}.

\begin{proof}[Proof of Theorem \ref{thm:ssmp}]
Let $\gamma \in \bb{G}$, $\mu$, and $(u, m)$ be as in the statement of the theorem. We may assume without loss of generality that Equation \eqref{eq:ssmp} holds for $\leb$-a.e.\@ $x \in \R^d$ and $\leb \otimes \pr$-a.e.\@ $(t, \omega) \in [0, T] \times \Omega$. Indeed, otherwise we follow the steps outlined below Theorem \ref{thm:nsmp} to obtain a new control $\gamma^{\ast}$, which coincides with $\gamma$ for $\mu_t$-a.e.\@ $x \in \R^d$ and $\leb \otimes \pr$-a.e.\@ $(t, \omega) \in [0, T] \times \Omega$, and a corresponding adjoint process $(u^{\ast}, m^{\ast})$, such that Equation \eqref{eq:ssmp} holds with $\gamma$ replaced by $\gamma^{\ast}$ and $(u, m)$ replaced by $(u^{\ast}, m^{\ast})$ for $\leb$-a.e.\@ $x \in \R^d$ and $\leb \otimes \pr$-a.e.\@ $(t, \omega) \in [0, T] \times \Omega$.

Next, fix $\gamma' \in \bb{G}$ and denote the corresponding solution of PDE \eqref{eq:rfpe} by $\mu'$. We compute
\begin{align} \label{eq:cost_diff}
    \tilde{J}(\gamma) - \tilde{J}(\gamma') &= \ev\biggl[\int_0^T \bigl\langle \mu_t, \tilde{f}(t, \cdot, \mu_t, \gamma_t)\bigr\rangle - \bigl\langle \mu'_t, \tilde{f}(t, \cdot, \mu'_t, \gamma'_t)\bigr\rangle \, \d t \biggr] + \ev\bigl[\tilde{\psi}(\mu_T) - \tilde{\psi}(\mu'_T)\bigr] \notag \\
    &\leq \ev\biggl[\int_0^T \bigl\langle \mu_t, \tilde{f}(t, \cdot, \mu_t, \gamma_t)\bigr\rangle - \bigl\langle \mu'_t, \tilde{f}(t, \cdot, \mu'_t, \gamma'_t)\bigr\rangle \, \d t \biggr] + \ev\langle \mu_T - \mu'_T, u_T\rangle,
\end{align}
where we used that the map $v \mapsto \psi(v)$ is convex and has a linear functional derivative that is continuous in the measure argument, so that Lemma \ref{lem:ldf_convex} applies. Note further that $D\tilde{\psi}(\mu_T) = u_T$. Then, as in the proof of Proposition \ref{eq:cost_derivative}, we apply the generalisation of It\^o's formula from Corollary \ref{cor:ito} to $\langle \mu_t - \mu'_t, u_T\rangle$, whereby
\begin{align*}
    \langle \mu_T - \mu'_T, u_T\rangle &= \langle \mu_0 - \mu'_0, u_0\rangle + \int_0^T \bigl\langle \mu_t, \tilde{\lambda}(t, \cdot, \mu_t) u_t + \tilde{b}_0(t, \cdot, \mu_t, \gamma_t) \cdot \nabla u_t\bigr\rangle \, \d t \\
    &\ \ \ - \int_0^T \bigl\langle \mu'_t, \tilde{\lambda}(t, \cdot, \mu'_t) u_t + \tilde{b}_0(t, \cdot, \mu'_t, \gamma_t) \cdot \nabla u_t\bigr\rangle \, \d t \\
    &\ \ \ - \int_0^T \bigl\langle \mu_t - \mu'_t, D\cal{H}(t, \mu_t, u_t, \nabla u_t, \gamma_t)\bigr\rangle \, \d t \\
    &= -\int_0^T \Bigl(\bigl\langle \mu_t, \tilde{f}(t, \cdot, \mu_t, \gamma_t)\bigr\rangle - \bigl\langle \mu'_t, \tilde{f}(t, \cdot, \mu'_t, \gamma'_t)\bigr\rangle\Bigr) \, \d t \\
    &\ \ \ + \int_0^T \Bigl(\cal{H}(t, \mu_t, u_t, \nabla u_t, \gamma_t) - \cal{H}(t, \mu'_t, u_t, \nabla u_t, \gamma'_t)\Bigr) \, \d t \\
    &\ \ \ - \int_0^T\bigl\langle \mu_t - \mu'_t, D\cal{H}(t, \mu_t, u_t, \nabla u_t, \gamma_t)\bigr\rangle \, \d t \\
    &\leq -\int_0^T \Bigl(\bigl\langle \mu_t, \tilde{f}(t, \cdot, \mu_t, \gamma_t)\bigr\rangle - \bigl\langle \mu'_t, \tilde{f}(t, \cdot, \mu'_t, \gamma'_t)\bigr\rangle\Bigr) \, \d t \\
    &\ \ \ + \int_0^T \bigl\langle \mu'_t, H_1(t, \cdot, \nabla u_t, \gamma_t) - H_1(t, \cdot, \nabla u_t, \gamma'_t)\bigr\rangle \, \d t,
\end{align*}
where we made use of inequality \eqref{eq:convexity_expression}, which followed from the convexity assumption on $\cal{H}_0$. By Equation \eqref{eq:ssmp}, the last term on the right-hand side is nonpositive, so that 
\begin{equation} \label{eq:bound_on_mult}
    \langle \mu_T - \mu'_T, u_T\rangle \leq -\int_0^T \Bigl(\bigl\langle \mu_t, \tilde{f}(t, \cdot, \mu_t, \gamma_t)\bigr\rangle - \bigl\langle \mu'_t, \tilde{f}(t, \cdot, \mu'_t, \gamma'_t)\bigr\rangle\Bigr) \, \d t.
\end{equation}
Taking expectation on both sides above and inserting the resulting expression into \eqref{eq:cost_diff} yields $\tilde{J}(\gamma) - \tilde{J}(\gamma') \leq 0$. Since $\gamma'$ was arbitrary, we obtain $\tilde{J}(\gamma) = \inf_{\gamma' \in \bb{G}}\tilde{J}(\gamma')$, so $\gamma$ is optimal.

Next, assume that $G \ni g \mapsto H_1\bigl(t, x, \nabla u_t(x), g\bigr)$ has a unique minimum for $\leb$-a.e.\@ $x \in \R^d$ and $\leb \otimes \pr$-a.e.\@ $(t, \omega) \in [0, T] \times \Omega$. Our goal is to show that $\gamma$ is the unique optimal control up to modification on null sets of $\mu_t$. If $\gamma' \in \bb{G}$ is not a modification of $\gamma$ on null sets of $\mu_t$, either \eqref{eq:bound_on_mult} holds with strict inequality with positive probability, so that $\tilde{J}(\gamma) - \tilde{J}(\gamma') < 0$ or $\gamma_t(x) = \gamma'_t(x)$ for $\mu'_t$-a.e.\@ $x \in \R^d$ and $\leb \otimes \pr$-a.e.\@ $(t, \omega) \in [0, T] \times \Omega$. In the former case, we are done, so let us lead the latter case to a contradiction. If $\gamma$ and $\gamma'$ are modifications of each other on null sets of $\mu'_t$, then the solution to the random Fokker--Planck equation \eqref{eq:rfpe} with input $\gamma$ is  $\mu'_t$. Since PDE \eqref{eq:rfpe} has a unique solution by Proposition \ref{prop:rfpe} and $\mu$ is another solution with input $\gamma$, it follows that $\mu' = \mu$. Thus, $\gamma'$ is a modification of $\gamma$ on null sets of $\mu_t$, in contradiction to our choice of $\gamma'$. This concludes the proof.
\end{proof}

%% file: 5_application/application.tex
In this final section we discuss the application of the theory developed in this article to the mean-field model for government interventions in financial markets from \cite{hambly_mvcp_arxiv_2023}. The model consists of an infinite system of banks, each represented by the capital it holds in excess of the regulatory capital requirements. The institutions have common exposures, modelled by the noise $W$, and mutual obligations. Whenever the excess capital of a bank is negative, so that it is in breach of the capital requirements, it has a positive default intensity, expressed by the coefficient $\lambda$. Through the mutual obligations, the default of an institution leads to losses at other banks, making them more likely to default as well. This contagion may precipitate systemic events, where a large portion of the system is wiped out in a short period of time. The controller, which should be thought of as a government or another central authority, can inject capital into the system to prevent defaults. 

The excess capital $X = (X_t)_{0 \leq t \leq T}$ of a representative bank in the mean-field system follows the McKean--Vlasov SDE
\begin{equation} \label{eq:sde_model}
    X_t = (\xi - c) + \int_0^t \alpha_s \, \d s + \sigma B_t + \sigma_0 W_t - \kappa L_t,
\end{equation}
where $\xi$ is the initial capital and $c > 0$ denotes the regulatory capital requirements, $\alpha_s \in [0, g_{\text{max}}]$ is the amount of capital injected by the government at time $t \in [0, T]$, $B$ and $W$ are independent Brownian motions that represent the fluctuations of banks' assets values correlated through the common noise $W$, $\kappa > 0$ measures the size of the mutual obligations, and $L_t$ denotes the fraction of banks that have defaulted up to time $t \in [0, T]$ conditional on the common noise $W$. In particular, we assume that the losses from default are proportional to the fraction of defaulted entities $L_t$. Let us try to give a more precise account of the latter quantity. As mentioned above, the representative institution's instantaneous default intensity is $\lambda(X_t)$, where we assume that $\lambda$ is bounded and $\lambda(x) = 0$ for $x \geq 0$ while $\lambda(x) > 0$ for $x < 0$. Default occurs, when the cumulative intensity $\Lambda_t = \int_0^t \lambda(X_s) \, \d s$ exceeds a standard exponentially distributed random variable $\theta$, which is independent of $X_0$, $B$, $W$, and $\alpha$. Hence, $L_t = \pr(\theta \geq \Lambda_t \vert W) = 1 - \ev[e^{-\Lambda_t}] = \int_0^t \langle \nu_s, \lambda\rangle \, \d s$, where $\nu_t = \pr(X_t \in \cdot, \, \theta > \Lambda_t \vert W)$ is the conditional subprobability distribution of the solvent banks. The government's objective is to minimise the cost functional
\begin{equation*}
    \alpha \mapsto \ev\biggl[\int_0^T \bf{1}_{\theta > \Lambda_t} w \alpha_t + L_T\biggr] = \ev\biggl[\int_0^T e^{-\Lambda_t} w \alpha_t + L_T\biggr]
\end{equation*}
over $\bb{F}^{\xi, B, W}$-progressively measurable $[0, g_{\text{max}}]$-valued processes $\alpha = (\alpha_t)_{0 \leq t \leq T}$. The weight $w > 0$ trades off the negative externalities caused by the bailout, such as moral hazard, measured by the size of the capital injections with the cost of a potential systemic crisis, captured by the fraction of defaulted entities $L_T$ at time $T$.

If $\alpha$ takes the form $\alpha_t = \gamma_t(X_t)$ for an $\bb{F}^W$-progressively measurable random function $\gamma \define [0, T] \times \Omega \times \R \to [0, g_{\text{max}}]$, then one can show that $\nu = (\nu_t)_{0 \leq t \leq T}$ satisfies the nonlinear stochastic Fokker--Planck equation
\begin{equation} \label{eq:sfpe_model}
    \d \langle \nu_t, \varphi\rangle = \Bigl\langle \nu_t, -\lambda \varphi + \gamma_t \partial_x \varphi - \kappa \langle \nu_t, \lambda\rangle \partial_x \varphi + \tfrac{1}{2}(\sigma^2 + \sigma_0^2\bigr) \partial_x^2 \varphi \Bigr\rangle \, \d t + \langle \nu_t, \sigma_0 \partial_x \varphi\rangle \, \d W_t
\end{equation}
for $\varphi \in C_c^2(\R)$ with initial condition $\nu_0 = \L(X_0)$. This SPDE is a special case of SPDE \eqref{eq:sfpe} (note, however, the sign change of $\lambda$) and we can rewrite the cost functional in terms of $\nu$ as
\begin{equation*}
    J(\gamma) = \ev\biggl[\int_0^T w\langle \nu_t, \gamma_t\rangle \, \d t + 1 - \nu_T(\R)\biggr].
\end{equation*}
This setup fits the framework introduced in Section \ref{sec:main_results}. Moreover, by Theorem \ref{thm:equivalence}, the control of McKean--Vlasov SDE \eqref{eq:sde_model} over $\bb{F}^{X_0, B, W}$-progressively measurable $[0, g_{\text{max}}]$-valued controls yields the same optimal value as the control of the nonlinear stochastic Fokker--Planck equation \eqref{eq:sfpe_model} over $\bb{F}^W$-progressively measurable random functions $[0, T] \times \Omega \times \R \to [0, g_{\text{max}}]$. (Note that here we use that the weak formulation for the McKean--Vlasov control problem as introduced in Appendix \ref{sec:equivalence} is equivalent to the strong formulation with $\bb{F}^{X_0, B, W}$-progressively measurable controls considered here by \cite[Theorem 2.8]{hambly_mvcp_arxiv_2023}, and that the value for the control problem of the nonlinear stochastic Fokker--Planck equation is independent of the weak setup, cf.\@ Theorem \ref{thm:optimal}.) The nonlinear random Fokker--Planck equation associated to SPDE \eqref{eq:sfpe_model} is
\begin{equation} \label{eq:rfpe_model}
    \d \langle \mu_t, \varphi\rangle = \Bigl\langle \mu_t, -\tilde{\lambda} \varphi + \gamma_t \partial_x \varphi - \kappa \langle \mu_t, \tilde{\lambda}\rangle \partial_x \varphi + \tfrac{\sigma^2}{2} \partial_x^2 \varphi \Bigr\rangle \, \d t
\end{equation}
for $\varphi \in C_c^2(\R)$ with initial condition $\mu_0 = \L(X_0)$, where $\gamma$ is still an $\bb{F}^W$-progressively measurable random function $[0, T] \times \Omega \times \R \to [0, g_{\text{max}}]$.

After replacing $\bb{F}^W$ with a larger filtration $\bb{F}$ if necessary, Theorem \ref{thm:optimal} implies that there exists an optimal $\bb{F}$-progressively measurable control $\gamma \define [0, T] \times \Omega \times \R \to [0, g_{\text{max}}]$ for $g_{\text{max}} > 0$. Now, for $p \in \R$, the reduced Hamiltonian $g \mapsto H_1(p, g) = g p + w g = (p + w) g$ is linear and, therefore, convex. Hence, the necessary SMP, Theorem \ref{thm:nsmp}, implies that the optimal control takes the form $\gamma_t(x) = g_{\text{max}} \bf{1}_{\partial_x u_t(x) \leq -w}$, where we use that the minimiser of the map $g \mapsto H_1(p, g)$ over $g \in [0, g_{\text{max}}]$ is given by $p \mapsto g_{\text{max}}\bf{1}_{p \leq -w}$. The adjoint process $(u, m)$ solves the semilinear BSPDE
\begin{align} \label{eq:hjb_model}
\begin{split}
    \d u_t(x) &= \Bigl(\tilde{\lambda}(x) u_t(x) + g_{\text{max}} (\partial_x u_t(x) + w)_- + \kappa \langle \mu_t, \tilde{\lambda}\rangle \partial_x u_t(x) + \kappa\langle \mu_t, \partial_x u_t\rangle \tilde{\lambda}(x) \Bigr) \, \d t \\
    &\ \ \ -\tfrac{\sigma^2}{2} \partial_x^2 u_t(x) \, \d t + \d m_t(x)
\end{split}
\end{align}
with terminal condition $u_T(x) = D(1 - \nu_T(\R)) = -1$. The nonlinearity $p \mapsto -g_{\text{max}} (p + w)_-$ arises from the minimisation of $(p + w) g$ over $g \in [0, g_{\text{max}}]$. As one would expect, the optimal control is of bang-bang type, i.e.\@ either the controller injects the maximal amount of capital allowed (this happens when $\partial_x u_t(x) \leq -w$) or none at all. Naturally, a smaller weight lowers the threshold at which the controller intervenes. 

Let us assume for the moment, that $\kappa$ and $\sigma_0$ vanish, so that BSPDE \eqref{eq:hjb_model} becomes a local and deterministic PDE:
\begin{equation} \label{eq:hjb_model_deterministic}
    -\partial_t u_t(x) + \lambda(x) u_t(x) + g_{\text{max}}(\partial_x u_t(x) + w)_- - \frac{\sigma^2}{2} \partial_x^2 u_t(x) = 0.
\end{equation}
The functions which are constantly equal to $-1$ and $0$ are sub- and supersolution to this PDE, respectively, so by the comparison principle we have $-1 \leq u_t(x) \leq 0$. Next, set $u^h_t(x) = u_t(x + h)$ for $h > 0$. Then, since $\lambda$ is nonincreasing and $u_t$ is nonpositive, it is easy to see that $u^h$ is a subsolution, whence $u_t(x + h) = u^h_t(x) \leq u_t(x)$. Since $h > 0$ is arbitrary, we see that $x \mapsto u_t(x)$ is nonincreasing, in other words $\partial_x u_t(x) \leq 0$. One may then conjecture that for each time $t \in [0, T)$ the set $\{x \in \R \define \partial_x u_t(x) \leq - w\}$, where the controller is active, is a possibly empty and otherwise compact interval $[a_t, b_t]$ in $\R$. That is, banks whose excess capital is above $b_t$ are safe and do not obtain any capital injections while banks with an excess capital below $a_t$ are not worth saving. This conjecture is confirmed by a finite-difference approximation of the solution to PDE \eqref{eq:hjb_model_deterministic}, depicted on the left-hand plot in Figure \ref{fig:adjoint_process}, but at this stage we are unable to prove it. 


\begin{figure}[tb] 
    \makebox[\linewidth][c]{
    \begin{subfigure}[b]{0.5\columnwidth}
        \centering
        \includegraphics[width=\columnwidth]{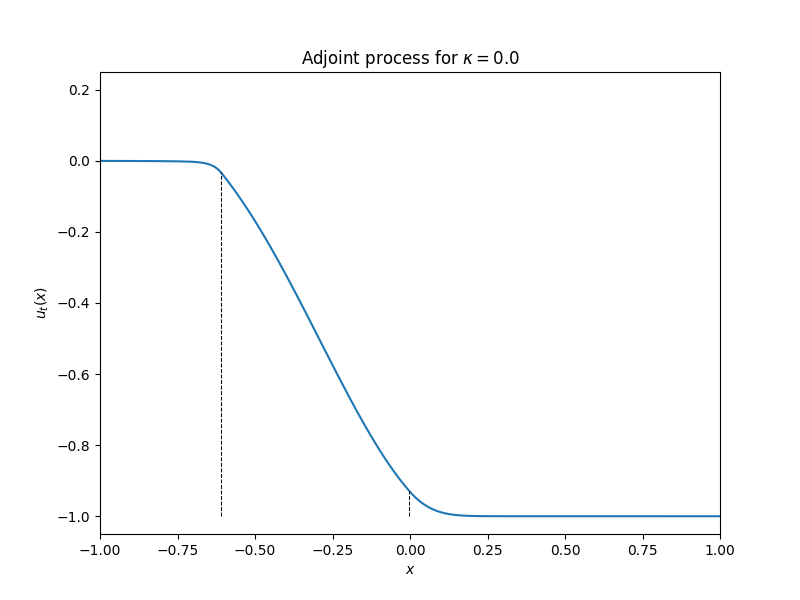}
    \end{subfigure}

    \hspace{-0.5cm}

    \begin{subfigure}[b]{0.5\columnwidth}
        \centering
        \includegraphics[width=\columnwidth]{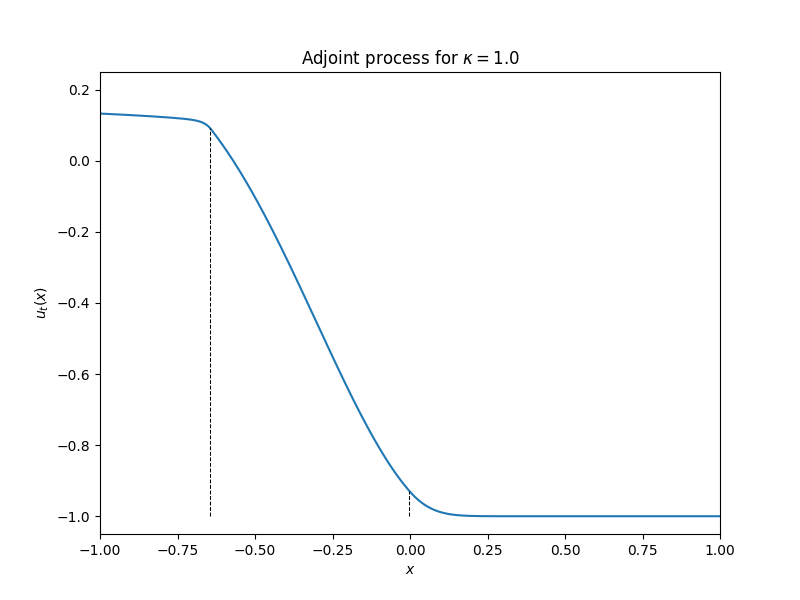}
    \end{subfigure}
    }
    \caption{Plot of numerical approximations of the adjoint process $u_t(x)$ for $\kappa = 0$ (left-hand side) and $\kappa = 1$ (right-hand side). The dotted lines indicate the interval for which $\partial_x u_t(x) \leq - w$.}
    \label{fig:adjoint_process}
 \end{figure}

Things become more complicated if $\kappa > 0$, in which case the state equation \eqref{eq:sfpe_model} becomes nonlinear. Then, owing to the nonlocal term $\kappa \langle \mu_t, \partial_x u_t\rangle \tilde{\lambda}(x)$ in BSPDE \eqref{eq:hjb_model}, it is no longer necessarily the case that $u_t$ is nonpositive.
However, as for the local PDE we conjecture that the set $\{x \in \R \define \partial_x u_t(x) \leq - w\}$ is either empty or a compact interval in $\R$.
This is once again confirmed by a numerical approximation of the solution to PDE \eqref{eq:hjb_model} in the setting without noise, i.e.\@ where $\sigma_0 = 0$ so that $m$ vanishes. The result is plotted on the right-hand side of Figure \ref{fig:adjoint_process}. To compute this approximation, we discretised the noiseless versions of both the Fokker--Planck equation \eqref{eq:rfpe_model} and PDE \eqref{eq:hjb_model} using a finite-difference scheme and then employed a Picard iteration to simulate the forward-backward system. The code used to produce both plots in Figure \ref{fig:adjoint_process} can be found on GitHub\footnote{\url{https://github.com/philkant/fokker-planck-control}}.

In conclusion, let us remark that the sufficient SMP, Theorem \ref{thm:ssmp}, cannot be applied to the nonlinear model, because if $\kappa > 0$, then the map $\M(\R) \ni v \mapsto \cal{H}_0(t, v, u_t, \partial_x u_t) = \bigl\langle v, -\tilde{\lambda} u_t - \kappa \langle v, \tilde{\lambda}\rangle \partial_x u_t\bigr\rangle$ will not be convex. This is due to the quadratic term $v \mapsto -\kappa \langle v, \tilde{\lambda}\rangle \langle v, \partial_x u_t\rangle$, which is convex if and only if $\tilde{\lambda} = -\partial_x u_t$ and the latter will not be satisfied. On the other hand, if $\kappa$ vanishes, this issue disappears and the sufficient SMP holds. In particular, since the map $g \mapsto g \partial_x u_t(x) + w g$ has a unique minimiser, namely $g_{\text{max}}\bf{1}_{\partial_x u_t(x) \leq - w}$, the optimal control $\gamma_t(x) = g_{\text{max}}\bf{1}_{\partial_x u_t(x) \leq - w}$ is unique (up to modifications on null sets of $\mu_t$).

%% file: 6_appendix/equivalence.tex
In this section we establish the equivalence between the optimal control of the nonlinear stochastic Fokker--Planck equation \eqref{eq:sfpe} and the associated McKean--Vlasov SDE.

We will first introduce a weak formulation of the control problem for the McKean--Vlasov SDE. We adopt the concept from \cite[Definition 8.1]{lacker_mimicking_2020}.

\begin{definition}
A \textit{weak control} is a tuple $(\Omega, \F, \bb{F}, \bb{F}^0, \pr, B, W, X, \nu, \alpha)$ such that
\begin{enumerate}[noitemsep, label = (\roman*)]
    \item $(\Omega, \F, \bb{F}, \pr)$ is a filtered probability space and $\bb{F}^0$ is a subfiltration of $\bb{F}$;
    \item $B$ is a $d$-dimensional $\bb{F}$-Brownian motion, $W$ is a $d_W$-dimensional $\bb{F}^0$-adapted $\bb{F}$-Brownian motion, $X$ is a continuous $\R^d$-valued $\bb{F}$-adapted process, $\nu$ is a continuous $\M(\R^d)$-valued $\bb{F}^0$-adapted process, and $\alpha$ is a $G$-valued $\bb{F}$-progressively measurable process;
    \item $X_0$, $B$, and $\F^0_T$ are independent;
    \item for all $t \in [0, T]$ we have $\F^{X, \alpha}_t \perp \F^B_T \lor \F^0_T \vert \F^B_t \lor \F^0_t$;
    \item for all $t \in [0, T]$ we have
    \begin{equation} \label{eq:sde}
        \d X_t = b(t, X_t, \nu_t, \alpha_t) \, \d t + \sigma(t, X_t) \, \d B_t + \sigma_0 \, \d W_t, \quad X_0 \sim \nu_0,
    \end{equation}
    and $\nu_t = \ev[e^{\Lambda_t} \delta_{X_t} \vert \F^0_T]$, where $\Lambda_t = \int_0^t \lambda(s, X_s, \nu_s) \, \d s$.
\end{enumerate}
\end{definition}


The coefficient $\sigma \define [0, T] \times \R^d \to \R^{d \times d}$ in \eqref{eq:sde} is a square root of the diffusion coefficient $a$ appearing in the generator $\L$ of the stochastic Fokker--Planck equation \eqref{eq:sfpe}.

We have the following result that establishes the equivalence between the two control formulations. Note that its proof still holds if both $\sigma$ and $\sigma_0$ depend on space and the measure argument as long as they are uniformly Lipschitz continuous in the space variable.

\begin{theorem} \label{thm:equivalence}
Let Assumption \ref{ass:fpe} be satisfied. If, in addition, the map $G \ni g \mapsto b(t, x, v, g)$ is affine and $G \ni g \mapsto f(t, x, v, g)$ is convex for all $(t, x, v) \in [0, T] \times \R^d \times \M(\R^d)$, then for any weak control $(\Omega, \F, \bb{F}, \bb{F}^0, \pr, B, W, X, \nu, \alpha)$, there exists a probabilistic setup $(\tilde{\Omega}, \tilde{\F}, \tilde{\bb{F}}, \tilde{\pr}, \tilde{W})$ and a control $\gamma \in \bb{G}_{\tilde{\bb{F}}}$ with a corresponding solution $\tilde{\nu}$ to SPDE \eqref{eq:sfpe} such that $\L^{\tilde{\pr}}(\tilde{\nu}_t) = \L^{\pr}(\nu_t)$ for all $t \in [0, T]$ and
\begin{equation*}
    \tilde{\ev}\biggl[\int_0^T \langle \tilde{\nu}_t, f(t, \cdot, \tilde{\nu}_t, \gamma_t)\rangle \, \d t + \psi(\tilde{\nu}_T)\biggr] \leq \ev\biggl[\int_0^T e^{\Lambda_t} f(t, X_t, \nu_t, \alpha_t) \, \d t + \psi(\nu_T)\biggr].
\end{equation*}
Conversely, if $\nu$ is a solution to SPDE \eqref{eq:sfpe} for a control $\gamma \in \bb{G}_{\bb{F}}$ on some probabilistic setup $(\Omega, \F, \bb{F}, \pr, W)$, then the probabilistic setup $(\Omega, \F, \bb{F}, \pr, W)$ together with $\nu$ and $\gamma$ can be extended to a weak control $(\tilde{\Omega}, \tilde{\F}, \tilde{\bb{F}}, \tilde{\bb{F}}^0, \tilde{\pr}, \tilde{B}, \tilde{W}, \tilde{X}, \tilde{\nu}, \tilde{\alpha})$ such that $\tilde{\bb{F}}^0$, $\tilde{W}$, and $\tilde{\nu}$ are the extensions of $\bb{F}$, $W$, and $\nu$, respectively, and $\tilde{\alpha}_t = \tilde{\gamma}_t(\tilde{X}_t)$ for $\leb \otimes \tilde{\pr}$-a.e.\@ $(t, \omega) \in [0, T] \times \tilde{\Omega}$, where $\tilde{\gamma}$ denotes the extension of $\gamma$.
\end{theorem}

\begin{proof}
Fix a weak control $(\Omega, \F, \bb{F}, \bb{F}^0, \pr, B, W, X, \nu, \alpha)$
and define measurable functions $r$, $g \define [0, T] \times \R^d \times \M(\R^d) \to \R$ by $r(t, x, v) = \log\ev[e^{\Lambda_t} \vert X_t = x,\, \nu_t = v]$ and
\begin{equation*}
    g(t, x, v) = \ev\bigl[e^{\Lambda_t - r(t, X_t, \nu_t)} \alpha_t \bigr\vert X_t = x,\, \nu_t = v\bigr],
\end{equation*}
so that $\langle \nu_t, \varphi\rangle = \ev[e^{r(t, X_t, \nu_t)} \varphi(X_t) \vert \nu_t]$ for any $\varphi \define \R^d \to \R$ measurable and bounded. Next, fix $\varphi \in C^2_c(\R^d)$ and apply It\^o's formula to $e^{\Lambda_t} \varphi(X_t)$, whereby
\begin{equation*}
    \d (e^{\Lambda_t} \varphi(X_t)) = \L\varphi(t, X_t, \nu_t, \alpha_t) \, \d t + e^{\Lambda_t} (\nabla \varphi(X_t))^{\top} \sigma(t, X_t) \, \d B_t + e^{\Lambda_t} (\nabla \varphi(X_t))^{\top} \sigma_0 \, \d W_t.
\end{equation*}
We take conditional expectation with respect to $\F^0_T$ on both sides to obtain
\begin{equation*}
    \d \langle \nu_t, \varphi\rangle = \ev[\L\varphi(t, X_t, \nu_t, \alpha_t) \vert \F^0_T] \, \d t + \langle \nu_t, (\nabla \varphi)^{\top} \sigma_0 \rangle \, \d W_t,
\end{equation*}
where we use the stochastic Fubini theorem to exchange the conditional expectation with respect to $\F^0_T$ and the stochastic integration against $W$. The stochastic integral with respect to $B$ vanishes, since $B$ is independent of $\F^0_T$. Now, we may proceed as in the proof of Corollary 1.6 in \cite{lacker_mimicking_2020}, to show that the law $P_t = \L(\nu_t)$ satisfies the Fokker--Planck equation
\begin{align*}
    \int_{\M(\R^d)} \Phi&(\langle v, \varphi\rangle) \, \d (P_t - P_0)(v) \\
    &= \sum_{i = 1}^k \int_0^t \biggl(\int_{\M(\R^d)}  \partial_{x_i} \Phi(\langle v, \varphi\rangle) \bigl\langle v, \L\varphi_i\bigl(s, \cdot, v, g(s, \cdot, v)\bigr)\bigr\rangle \, \d P_s(v)\biggr) \, \d s \\
    &\ \ \ + \frac{1}{2}\sum_{i, j = 1}^k \int_0^t \biggl(\int_{\M(\R^d)}  \partial^2_{x_i x_j} \Phi(\langle v, \varphi\rangle) \langle v, (\nabla \varphi_i)^{\top} \sigma_0\rangle \cdot \langle v, (\nabla \varphi_j)^{\top} \sigma_0\rangle \, \d P_s(v)\biggr) \, \d s
\end{align*}
for $\varphi = (\varphi_1, \dots, \varphi_k) \in C^2_c(\R^d; \R^k)$ and $\Phi \in C^2_c(\R^k)$ on $\M(\R^d)$. Here we used that
\begin{align*}
    \ev\bigl[\partial_{x_i}\Phi&(\langle \nu_t, \varphi\rangle) e^{\Lambda_t} \nabla \varphi_i(X_t) \cdot b(t, X_t, \nu_t, \alpha_t)\bigr] \\
    &= \ev\Bigl[\partial_{x_i}\Phi(\langle \nu_t, \varphi\rangle) e^{r(t, X_t, \nu_t)} \nabla \varphi_i(X_t) \cdot b\Bigl(t, X_t, \nu_t, \ev\bigl[e^{\Lambda_t - r(t, X_t, \nu_t)} \alpha_t \big\vert X_t,\, \nu_t\bigr]\Bigr) \Bigr] \\
    &= \ev\Bigl[\partial_{x_i}\Phi(\langle \nu_t, \varphi\rangle) e^{r(t, X_t, \nu_t)} \nabla \varphi_i(X_t) \cdot b\bigl(t, X_t, \nu_t, g(t, X_t, \nu_t\bigr) \Bigr] \\
    &= \ev\Bigl[\partial_{x_i}\Phi(\langle \nu_t, \varphi\rangle) \bigl\langle \nu_t, \nabla \varphi_i \cdot b\bigl(t, X_t, \nu_t, g(t, \cdot, \nu_t\bigr)\bigr\rangle\Bigr],
\end{align*}
where we used in the first equality that the expectation operator $\ev[e^{\Lambda_t - r(t, X_t, \nu_t)} \cdot \lvert X_t,\, \nu_t]$ can be moved inside the control argument of $b$ since $b$ is affine in this argument. Hence, we can apply the superposition principle for stochastic Fokker--Planck equations from \cite[Theorem 1.5]{lacker_mimicking_2020}, which can easily be extended to equations with zeroth-order terms, to show there exists a probabilistic setup $(\tilde{\Omega}, \tilde{\F}, \tilde{\bb{F}}, \tilde{\pr}, \tilde{W})$ with $\tilde{\F}$ countable generated, supporting a continuous $\tilde{\bb{F}}$-adapted $\M(\R^d)$-valued process $\tilde{\nu}$ that solves SPDE \eqref{eq:sfpe} with input $\gamma \in \bb{G}_{\tilde{\bb{F}}}$ given by $\gamma_t(x) = g(t, x, \tilde{\nu}_t)$ for $(t, \omega, x) \in [0, T] \times \tilde{\Omega} \times \R^d$. Moreover, by convexity of $f_1$ in the control argument, it holds that
\begin{align*}
    \ev\biggl[\int_0^T e^{\Lambda_t}& f(t, X_t, \nu_t, \alpha_t) \, \d t + \psi(\nu_T)\biggr] \\
    &= \ev\biggl[\int_0^T e^{r(t, X_t, \nu_t)} \ev\bigl[e^{\Lambda_t - r(t, X_t, \nu_t)} f(t, X_t, \nu_t, \alpha_t) \big\vert X_t, \, \nu_t\bigr] \, \d t + \psi(\nu_T)\biggr] \\
    &\geq \ev\biggl[\int_0^T e^{r(t, X_t, \nu_t)} f\bigl(t, X_t, \nu_t, \ev\bigl[e^{\Lambda_t - r(t, X_t, \nu_t)}\alpha_t\big\vert X_t, \, \nu_t\bigr]\bigr) \, \d t + \psi(\nu_T)\biggr] \\
    &= \ev\biggl[\int_0^T \bigl\langle \nu_t, f\bigl(t, \cdot, \nu_t, g(t, \cdot, \nu_t)\bigr)\bigr\rangle \, \d t + \psi(\nu_T)\biggr] \\
    &= \tilde{\ev}\biggl[\int_0^T \bigl\langle \tilde{\nu}_t, f\bigl(t, \cdot, \tilde{\nu}_t, g(t, \cdot, \tilde{\nu}_t)\bigr)\bigr\rangle \, \d t + \psi(\tilde{\nu}_T)\biggr],
\end{align*}
where we applied Jensen's inequality using that $\F \ni A \mapsto \ev[e^{\Lambda_t - r(t, X_t, \nu_t)} \bf{1}_A \vert X_t, \nu_t]$ is a (random) probability measure. This concludes the proof of the first part of the proposition.

Next, let us fix a probabilistic setup $(\Omega, \F, \bb{F}, \pr, W)$ and let $\nu$ be a solution to SPDE \eqref{eq:sfpe} for a control $\gamma \in \bb{G}_{\bb{F}}$. Then we let $\mu = (\mu_t)_{0 \leq t \leq T}$ be the unique solution to the linear SPDE
\begin{align*}
    \d \langle \mu_t, \varphi\rangle = \bigl\langle \mu_t, \L\varphi(t, \cdot, \nu_t, \gamma_t)\bigr\rangle \, \d t - \langle \mu_t, \lambda(t, \cdot, \nu_t)\rangle \, \d t + \bigl\langle \mu_t, (\nabla \varphi)^{\top} \sigma_0\bigr\rangle \, \d W_t
\end{align*}
for $\varphi \in C_c^2(\R^d)$ with initial condition $\mu_0 = \nu_0$. Since this stochastic Fokker--Planck equation does not have a zeroth-order term and $\nu_0 \in \P(\R^d)$, it follows that $\mu_t$ takes values in $\P(\R^d)$. Consequently, by the superposition principle for SDEs with random coefficients from \cite[Theorem 1.3]{lacker_mimicking_2020}, there exists an extension $(\tilde{\Omega}, \tilde{\F}, \tilde{\bb{F}}, \tilde{\pr})$ of $(\Omega, \F, \bb{F}, \pr)$ supporting a $d$-dimensional $\tilde{\bb{F}}$-Brownian motion $B$ independent of $\F_T$ and a continuous $\R^d$-valued $\tilde{\bb{F}}$-adapted process $X$ such that 
\begin{enumerate}[noitemsep, label = (\roman*)]
    \item for all $t \in [0, T]$ we have $\F^X_t \perp \F^B_T \lor \F_T \vert \F^B_t \lor \F_t$;
    \item $W$ is an $\tilde{\bb{F}}$-Brownian motion;
    \item for all $t \in [0, T]$ we have
    \begin{equation*}
        \d X_t = b\bigl(t, X_t, \nu_t, \gamma_t(X_t)\bigr) \, \d t + \sigma(t, X_t) \, \d B_t + \sigma_0 \, \d W_t, \quad X_0 \sim \nu_0,
    \end{equation*}
    and $\mu_t = \L^{\tilde{\pr}}(X_t \vert \F_T)$.
\end{enumerate}
Let us set $\Lambda_t = \int_0^t \lambda(s, X_s, \nu_s) \, \d s$ and then define the $\M(\R^d)$-valued process $\tilde{\nu} = (\tilde{\nu}_t)_{0 \leq t \leq T}$ by $\tilde{\nu}_t = \tilde{\ev}[e^{\Lambda_t} \delta_{X_t} \vert \F_T]$ for $t \in [0, T]$. Proceeding as at the beginning of the proof, we can show that $\tilde{\nu}$ solves the linear SPDE
\begin{equation*}
    \d \langle \tilde{\nu}_t, \varphi\rangle = \bigl\langle \tilde{\nu}_t, \L\varphi(t, \cdot, \nu_t, \gamma_t)\bigr\rangle \, \d t + \bigl\langle \tilde{\nu}_t, (\nabla \varphi)^{\top} \sigma_0\bigr\rangle \, \d W_t
\end{equation*}
for $\varphi \in C_c^2(\R^d)$ with initial condition $\tilde{\nu}_0 = \mu_0 = \nu_0$. But the unique solution to this SPDE is $\nu$, so we conclude that $\tilde{\nu}_t = \nu_t$, meaning that $(X, \nu)$ is a solution to McKean--Vlasov SDE \eqref{eq:sde}. Consequently, $(\tilde{\Omega}, \tilde{\F}, \tilde{\bb{F}}, \bb{F}, B, W, X, \nu, \alpha)$ with $\alpha = (\alpha_t)_{0 \leq t \leq T}$ defined by $\alpha_t = \gamma_t(X_t)$ is the desired weak control.
\end{proof}

%% file: 6_appendix/ldf.tex
\begin{lemma} \label{lem:ldf_diff_quot}
Let $F \define \M(\R^d) \to \R$ be a measurable function. If $F$ has a linear functional derivative $DF \define \R^d \times \M(\R^d) \to \R$ such that $v \mapsto DF(v)(x)$ is continuous for every $x \in \R^d$, then 
\begin{equation*}
    DF(v)(x) = \lim_{\epsilon \to 0} \frac{1}{\epsilon} \bigl(F(v + \epsilon \delta_x) - F(v)\bigr)
\end{equation*}
for all $(x, v) \in \R^d \times \M(\R^d)$. Conversely, if there exists a measurable and bounded function $\F \define \R^d \times \M(\R^d) \to \R$ such that 
\begin{equation} \label{eq:der_f}
    \lim_{\epsilon \to 0} \frac{1}{\epsilon} \bigl(F(v + \epsilon (v' - v)) - F(v)\bigr) = \langle v' - v, \F(v)\rangle
\end{equation}
for all $v$, $v' \in \M(\R^d)$, then $\F$ is a linear functional derivative of $F$.
\end{lemma}

\begin{proof}
For the first statement, it follows from the definition of the linear functional derivative that
\begin{equation*}
    \lim_{\epsilon \to 0} \frac{1}{\epsilon} \bigl(F(v + \epsilon \delta_x) - F(v)\bigr) = \lim_{\epsilon \to 0} \int_0^1 \langle \delta_x, DF(v + \epsilon \theta \delta_x)\rangle \, \d \theta = DF(v)(x),
\end{equation*}
where we used the continuity of $DF$ in the measure argument in the last step. Note that in fact much less than continuity with respect to $d_0$ is required here.

To establish the second statement, let us fix $v$, $v' \in \M(\R^d)$ and define the function $f \define [0, 1] \to \R$ by $f(\theta) = F(v + \theta (v' - v))$. From \eqref{eq:der_f}, it follows that $f$ is differentiable with $f'(\theta) = \bigl\langle v' - v, \F(v + \theta(v' - v))\bigr\rangle$. Thus, by the fundamental theorem of calculus it holds that
\begin{equation*}
    F(v') - F(v) = f(1) - f(0) = \int_0^1 f'(\theta) \, \d \theta = \int_0^1 \bigl\langle v' - v, \F(\theta v' + (1 - \theta)v) \bigr\rangle \, \d \theta.
\end{equation*}
This shows that $\F$ is indeed a linear functional derivative of $F$.
\end{proof}

\begin{lemma} \label{lem:ldf_product}
Let $F \define \R^d \times \M(\R^d) \to \R$ be a measurable and bounded function such that $\M(\R^d) \ni v \mapsto F(x, v)$ has a linear functional derivative $DF(x, \cdot)$ for all $x \in \R^d$, such that $DF \define \R^d \times \M(\R^d) \times \R^d \to \R$ is measurable, bounded, and continuous in the measure argument. Then the map $\M(\R^d) \ni v \mapsto \langle v, F(\cdot, v)\rangle$ possesses the linear functional derivative $F(x, v) + \langle v, DF(\cdot, v)(x)\rangle$.
\end{lemma}

\begin{proof}
Set $\bb{F}(v) = \langle v, F(\cdot, v)\rangle$ for $v \in \M(\R^d)$. By the second part of Lemma \ref{lem:ldf_diff_quot} it is enough to show that for any $v$, $v' \in \M(\R^d)$ the limit $\lim_{\epsilon \to 0}\frac{1}{\epsilon} \bigl(\bb{F}(v + \epsilon (v' - v)) - \bb{F}(v)\bigr)$ exists and is equal to
\begin{equation} \label{eq:ldf_product}
    \int_{\R^d} \bigl(F(x, v) + \langle v, DF(\cdot, v)(x)\rangle\bigr) \, \d (v' - v)(x).
\end{equation}
We compute
\begin{align} \label{eq:decomp_ldf_product}
\begin{split}
    \frac{1}{\epsilon} \bigl(\bb{F}(v + \epsilon (v' - v)) - \bb{F}(v)\bigr) &= \bigl\langle v' - v, F(\cdot, v + \epsilon (v' - v))\bigr\rangle \\
    &\ \ \ + \frac{1}{\epsilon}\bigl\langle v, F(\cdot, v + \epsilon(v' - v)) - F(\cdot, v)\bigr\rangle.
\end{split}
\end{align}
Now, since for any $x \in \R^d$, the map $m \mapsto F(x, m)$ has a linear functional derivative, its continuous. Consequently, the first term on the right-hand side above converges to $\langle v' - v, F(\cdot, v)\rangle$ by the dominated convergence theorem. Next, for fixed $x \in \R^d$, it follows from the continuity of $m \mapsto DF(x, m)(y)$ for all $y \in \R^d$ and the first part of Lemma \ref{lem:ldf_diff_quot} that
\begin{equation*}
    \frac{1}{\epsilon}\bigl(F(x, v + \epsilon(v' - v)) - F(x, v)\bigr) \to \langle v' - v, DF(x, v)\rangle.
\end{equation*}
Applying the dominated convergence theorem once more we conclude that the second expression on the right-hand side of \eqref{eq:decomp_ldf_product} tends to
\begin{equation*}
    \int_{\R^d} \langle v, DF(\cdot, v)(x)\rangle \, \d (v' - v)(x).
\end{equation*}
Combining the two expression shows that $\frac{1}{\epsilon} \bigl(\bb{F}(v + \epsilon (v' - v)) - \bb{F}(v)\bigr)$ indeed tends to the expression in \eqref{eq:ldf_product}. 
\end{proof}

\begin{lemma} \label{lem:ldf_convex}
Let $F \define \M(\R^d) \to \R$ be convex and assume $F$ has a linear functional derivative $DF$ such that for all $v$, $v' \in \M(\R^d)$ and $x \in \R^d$ it holds that $DF(\theta v + (1 - \theta) v')(x) \to DF(v')(x)$ as $\theta \in (0, 1]$ tends to zero. Then we have $F(v) - F(v') \geq \langle v - v', DF(v')\rangle$.
\end{lemma}

\begin{proof}
By convexity of $F$, we have for any $\epsilon \in (0, 1]$ that
\begin{align*}
    F(v) - F(v') &\geq \frac{1}{\epsilon} \bigl(F(v' + \epsilon(v - v')) - F(v')\bigr) \\
    &= \int_0^1 \bigl\langle v - v', DF\bigl(v' + \epsilon \theta (v - v')\bigr)\bigr\rangle \, \d \theta.
\end{align*}
The expression on the right-hand side converges to $\langle v - v', DF(v')\rangle$ as $\epsilon \to 0$, as required. 
\end{proof}

%% file: 6_appendix/stoch_int_conv.tex
For the rest of this section we assume that $\bb{H}$ is a separable Hilbert space. We begin with an infinite-dimensional analogue of the Kunita--Watanabe inequality.

\begin{proposition} \label{prop:kunita_watanabe}
Let $m$, $\tilde{m}$ be elements of $\cal{M}_{\bb{F}}^2([0, T]; \bb{H})$ and let $F$, $\tilde{F}$ be measurable functions $[0, T] \times \Omega \to B(\bb{H})$. Then for all $t \in [0, T]$ it holds that
\begin{equation*}
    \int_0^t \trace\bigl(\lvert \tilde{F}_s F_s\rvert \, \d A_s\bigr) \leq \biggl(\int_0^t \lVert F^2_s \rVert \, \d [m]_s\biggr)^{1/2} \biggl(\int_0^t \lVert \tilde{F}^2_s\rVert \, \d [\tilde{m}]_s\biggr)^{1/2},
\end{equation*}
where $A$ is the total variation of $[[m, \tilde{m}]]$ and $\lvert \cdot \rvert$ denotes the absolute value of a bounded linear operator. Here $\lVert \cdot \rVert$ is the operator norm for elements of $B(\bb{H})$.
\end{proposition}

\begin{proof}
Using the approximation of the tensor quadratic variation by sums of square increments, see \cite[Theorem 26.11 ]{metivier_semimartingales_1982}, we can show that for $0 \leq s \leq t \leq T$ and $L$, $\tilde{L} \in B(\bb{H})$ it holds that
\begin{align*}
    \trace\bigl(\lvert \tilde{L} L\rvert (A_t - A_s)\bigr) \leq \lVert L\rVert \lVert \tilde{L}\rVert \trace(A_t - A_s) \leq \lVert L\rVert \lVert \tilde{L}\rVert \sqrt{[m]_t - [m]_s} \sqrt{[\tilde{m}]_t - [\tilde{m}]_s}.
\end{align*}
With this inequality established one can follow the proof of the classical one-dimensional Kunita--Watanabe inequality. See e.g.\@ \cite[Proposition 4.18]{gall_bm_2016}.
\end{proof}

\begin{lemma} \label{lem:stoch_int_conv}
Suppose that $(u^n)_{n \geq 1}$ is a sequence in $D_{\bb{F}}^2([0, T]; \bb{H})$ that converges to some process $u \in D_{\bb{F}}^2([0, T]; \bb{H})$, and that $(m^n)_{n \geq 1}$, $(\tilde{m}^n)_{n \geq 1}$ are sequences in $\cal{M}_{\bb{F}}^2([0, T]; \bb{H})$ that converge to processes $m$, $\tilde{m} \in \cal{M}_{\bb{F}}^2([0, T]; \bb{H})$, respectively. Next, let $F_1 \define \bb{H} \to \bb{H}$ and $F_2 \define \bb{H} \to B(\bb{H})$ be continuous and bounded on bounded sets of $\bb{H}$. Then
\begin{align} 
    \int_0^t \langle F_1(u^n_{s-}), \d m^n_s\rangle_{\bb{H}} &\to \int_0^t \langle F_1(u_{s-}), \d m_s\rangle_{\bb{H}}, \label{eq:conv_stoch_int_1} \\
    \int_0^t \trace\bigl(F_2(u^n_{s-}) \, \d [[m^n, \tilde{m}^n]]^c_s\bigr) &\to \int_0^t \trace\bigl(F_2(u_{s-}) \, \d [[m, \tilde{m}]]^c_s\bigr) \label{eq:conv_stoch_int_2}
\end{align}
uniformly in $t \in [0, T]$ in probability.
\end{lemma}

\begin{proof}
We prove the convergences in \eqref{eq:conv_stoch_int_1} and \eqref{eq:conv_stoch_int_2} separately. For the first one we write
\begin{align} \label{eq:decomp_first_integral}
\begin{split}
    \int_0^t \langle F_1(u^n_{s-})&, \d m^n_s\rangle_{\bb{H}} - \int_0^t \langle F_1(u_{s-}), \d m_s\rangle_{\bb{H}} \\
    &= \int_0^t \bigl\langle F_1(u^n_{s-}) - F_1(u_{s-}), \d m^n_s\bigr\rangle_{\bb{H}} + \int_0^t \bigl\langle F_1(u_{s-}), \d (m^n_s - m_s)\bigr\rangle_{\bb{H}}.
\end{split}
\end{align}
We will deal with the terms on the right-hand side in order. Our goal is to prove that both converge to zero in probability. For $k \geq 1$ let $\pi_k \define \bb{H} \to \bb{H}$ denote the projection onto the centred ball $B_k$ of radius $k$ in $\bb{H}$. Now, since $\ev \sup_{0 \leq s \leq t} \lVert u^n_s - u_s\rVert_{\bb{H}}^2 \to 0$, for any $\delta > 0$ we can find $K \geq 1$ such that for all $k \geq K$ it holds that 
\begin{align*}
    \pr\Bigl(\pi_k(u^n_{s-}) = u^n_{s-},\, \pi_k(u_{s-}) = u_{s-} \ \text{for } s \in (0, t]\Bigr) &= \pr\biggl(\sup_{0 \leq s < t}\lVert u^n_s\rVert_{\bb{H}}^2 \lor \sup_{0 \leq s < t}\lVert u_s\rVert_{\bb{H}}^2 \leq k\biggr) \\
    &\geq 1 - \delta.
\end{align*}
Hence, for any $\epsilon > 0$ we obtain
\begin{align} \label{eq:probability_bound}
\begin{split}
    \pr\biggl(\sup_{0 \leq t \leq T}\biggl\lvert &\int_0^t \bigl\langle F_1(u^n_{s-}) - F_1(u_{s-}), \d m^n_s\bigr\rangle_{\bb{H}}\biggr\rvert > \epsilon\biggr) \\
    &\leq \pr\biggl(\sup_{0 \leq t \leq T}\biggl\lvert \int_0^t \bigl\langle F_1(\pi_k(u^n_{s-})) - F_1(\pi_k(u_{s-})), \d m^n_s\bigr\rangle_{\bb{H}}\biggr\rvert > \epsilon\biggr) + \delta.
\end{split}
\end{align}
Next, using Proposition \ref{prop:bdg} we find that 
\begin{align} \label{eq:expectation_bound_conv}
    \pr\biggl(\sup_{0 \leq t \leq T}\biggl\lvert \int_0^t \bigl\langle F_1&(\pi_k(u^n_{s-})) - F_1(\pi_k(u_{s-})), \d m^n_s\bigr\rangle_{\bb{H}}\biggr\rvert > \epsilon\biggr) \notag \\
    &\leq \frac{1}{\epsilon}\ev\sup_{0 \leq t \leq T}\biggl\lvert \int_0^t \bigl\langle F_1(\pi_k(u^n_{t-})) - F_1(\pi_k(u_{t-})), \d m^n_s\bigr\rangle_{\bb{H}}\biggr\rvert \notag \\
    &\leq \frac{C_1}{\epsilon} \ev\biggl(\int_0^T \bigl\lVert F_1(\pi_k(u^n_{t-})) - F_1(\pi_k(u_{t-}))\bigr\rVert_{\bb{H}}^2 \, \d [m^n]_t\biggr)^{1/2} \notag \\
    &\leq \frac{C_1}{\epsilon} \ev\biggl[\sup_{0 \leq t \leq T} \bigl\lVert F_1(\pi_k(u^n_t)) - F_1(\pi_k(u_t))\bigr\rVert_{\bb{H}} \bigl([m^n]_T - [m^n]_0\bigr)^{1/2}\biggr] \notag \\
    &\leq \frac{C_1^2}{\delta \epsilon^2} \ev \sup_{0 \leq t \leq T} \bigl\lVert F_1(\pi_k(u^n_t)) - F_1(\pi_k(u_t))\bigr\rVert_{\bb{H}}^2 + \delta \ev\bigl([m^n]_T - [m^n]_0\bigr).
\end{align}
Now, since $m^n$ converges to $m$ in $\cal{M}_{\bb{F}}^2([0, T]; \bb{H})$, the quantity $\ev\bigl([m^n]_T - [m^n]_0\bigr)$ is bounded uniformly in $n \geq 1$ by some constant $C > 0$. Moreover, the map $F_1 \circ \pi_k$ is continuous and bounded, so since $\ev \sup_{0 \leq t \leq T} \lVert u^n_t - u_t\rVert_{\bb{H}}^2 \to 0$ by the dominated convergence theorem, we may choose $N \geq 1$ large enough so that for all $n \geq N$ we have $\ev \sup_{0 \leq t \leq T} \bigl\lVert F_1(\pi_k(u^n_t)) - F_1(\pi_k(u_t))\bigr\rVert_{\bb{H}}^2 \leq \frac{\delta^2 \epsilon^2}{C_1^2}$. Plugging this into \eqref{eq:expectation_bound_conv} and combining it with Equation \eqref{eq:probability_bound} implies that
\begin{equation*}
    \pr\biggl(\sup_{0 \leq t \leq T} \biggl\lvert \int_0^t \bigl\langle F_1(u^n_{s-}) - F_1(u_{s-}), \d m^n_s\bigr\rangle_{\bb{H}}\biggr\rvert > \epsilon\biggr) \leq (2 + C)\delta.
\end{equation*}
Since $\delta$ and $\epsilon$ were arbitrary this implies that $\int_0^t \bigl\langle F_1(u^n_{s-}) - F_1(u_{s-}), \d m^n_s\bigr\rangle_{\bb{H}}$ converges uniformly in $t \in [0, T]$ to zero in probability.

Next, let us consider the second term on the right-hand side of \eqref{eq:decomp_first_integral}. Arguing as above for any $\delta > 0$ we can choose $K \geq 1$ and $k \geq K$ such that for any $\epsilon > 0$,
\begin{align*}
    \pr\biggl(\biggl\lvert\int_0^t \bigl\langle F_1&(u_{s-}), \d (m^n_s - m_s)\bigr\rangle_{\bb{H}}\biggr\rvert > 
    \epsilon\biggr) \\
    &\leq \frac{1}{\epsilon}\ev\biggl\lvert\int_0^t \bigl\langle F_1(\pi_k(u_{s-})), \d (m^n_s - m_s)\bigr\rangle_{\bb{H}}\biggr\rvert + \delta \\
    &\leq \frac{\delta}{c_k^2} \ev \sup_{0 \leq t \leq T} \lVert F_1(\pi_k(u_t))\rVert_{\bb{H}}^2 + \frac{C_1^2 c_k^2}{\delta \epsilon^2} \ev\bigl[[m^n]_t - [m^{\epsilon}]_0 - ([m]_t - [m]_0)\bigr] + \delta \\
    &\leq 2\delta + \frac{C_1^2 c_k^2}{\delta \epsilon^2} \ev\bigl[[m^n]_t - [m^n]_0 - ([m]_t - [m]_0)\bigr],
\end{align*}
where $c_k = \sup_{h \in B_k}\lVert F_1(h)\rVert_{\bb{H}}$. Then we choose $N \geq 1$ large enough such that for all $n \geq 1$ we have $\ev\bigl[[m^n]_t - [m^n]_0 - ([m]_t - [m]_0)\bigr] \leq \frac{\delta^2 \epsilon^2}{C_1^2 c_n^2}$ to obtain $\pr\bigl(\sup_{0 \leq t \leq T}\bigl\lvert\int_0^t \bigl\langle F_1(u_{s-}), \d (m^n_s - m_s)\bigr\rangle_{\bb{H}}\bigr\rvert > \epsilon\bigr) \leq 3 \delta$, which means that $\int_0^t \bigl\langle F_1(u_{s-}), \d (m^n_s - m_s)\bigr\rangle_{\bb{H}}$ converges uniformly in $t \in [0, T]$ to zero in probability as well. This concludes our treatment of the convergence in \eqref{eq:conv_stoch_int_1}.

We continue with \eqref{eq:conv_stoch_int_2}. By appealing to the polarisation identity $[[m, \tilde{m}]] = \frac{1}{4}([[m + \tilde{m}]] - [[m - \tilde{m}]])$, we can reduce to the case $\tilde{m}^n = m^n$ and $\tilde{m} = m$. Then similarly to Equation \eqref{eq:decomp_first_integral}, we decompose
\begin{align*}
    \int_0^t \trace\bigl(&F_2(u^n_{s-}) \, \d [[m^n]]^c_s\bigr) - \int_0^t \trace\bigl(F_2(u_{s-}) \, \d [[m]]^c_s\bigr) \\
    &= \int_0^t \trace\bigl((F_2(u^n_{s-}) - F_2(u_{s-})) \, \d [[m^n]]^c_s\bigr) + \int_0^t \trace\bigl(F_2(u_{s-}) \, \d ([[m^n]]^c_s - [[m]]^c_s)\bigr).
\end{align*}
For the first term we apply the inequality
\begin{equation} \label{eq:bound_trace}
    \sup_{0 \leq t \leq T}\biggl\lvert \int_0^t \trace\bigl((F_2(u^n_{s-}) - F_2(u_{s-})) \, \d [[m^n]]^c_s\bigr) \biggr\rvert \leq \int_0^T \bigl\lVert F_2(u^n_{t-}) - F_2(u_{t-})\bigr\rVert \, \d [m^n]^c_t,
\end{equation}
where $\bigl\lVert F_2(u^n_{t-}) - F_2(u_{t-})\bigr\rVert$ denotes the operator norm of $F_2(u^n_{t-}) - F_2(u_{t-}) \define \bb{H} \to \bb{H}$ and then proceed as above to show that the right-hand side converges to zero in probability as $n \to \infty$. To deal with the second term, we write $[[m^n]]^c_s - [[m]]^c_s = [[m^n - m, m^n]]^c_s + [[m, m^n - m]]^c_s$ and then apply the Kunita--Watanabe inequality, Proposition \ref{prop:kunita_watanabe}, to obtain
\begin{align*}
    \sup_{ 0 \leq t \leq T} \biggl\lvert \int_0^t \trace\bigl(F_2&(u_{s-}) \, \d ([[m^n]]^c_s - [[m]]^c_s)\bigr)\biggr\rvert \\
    &\leq \sqrt{2}\biggl(\int_0^T \lVert F_2(u_{t-})\rVert \, \d ([m^n]_t + [m]_t)\biggr)^{1/2}\biggl(\int_0^T \lVert F_2(u_{t-})\rVert \, \d [m^n - m]_t\biggr)^{1/2} \\
    &\leq \delta \int_0^T \lVert F_2(u_{t-})\rVert \, \d ([m^n]_t + [m]_t) + \frac{1}{2\delta} \int_0^T \lVert F_2(u_{t-})\rVert \, \d [m^n - m]_t.
\end{align*}
The first expression on the right-hand can be made arbitrarily small by choosing $\delta > 0$ appropriately. The second term tends to zero in probability as $n \to \infty$. Hence, $\sup_{ 0 \leq t \leq T} \bigl\lvert \int_0^t \trace\bigl(F_2(u_{s-}) \, \d ([[m^n]]^c_s - [[m]]^c_s)\bigr)\bigr\rvert \to 0$ in probability, which concludes the proof.
\end{proof}

%% file: 6_appendix/int_conv.tex
In what follows, $\eta \define \R^d \to \R$ is defined by $\eta(x) = \eta_0 \sqrt{1 + \lvert x\rvert^2}$ for $x \in \R^d$ and some $\eta_0 > 0$. Recall the distance $d_0$ on the space $\M(\R^d)$ of finite measures on $\R^d$ defined in Equation \eqref{eq:blip}. 

\begin{lemma} \label{lem:loc_l2_to_d0}
Let $(\mu^n)_n$ be a bounded sequence in $L^2([0, T]; L^2_{\eta}(\R^d))$ that converges to some $\mu \in L^2([0, T]; L^2_{\eta}(\R^d))$ in $L^2([0, T]; L^2_{\eta, \textup{loc}}(\R^d))$. Then it holds that
\begin{equation*}
    \int_0^T \sup_{\lVert \varphi\rVert_{\infty} \leq 1} \langle \mu^n_t - \mu_t, \varphi\rangle \, \d t \to 0
\end{equation*}
as $n \to \infty$, where the supremum is taken over all bounded and measurable functions $\varphi \define \R^d \to \R$ with $\lVert \varphi\rVert_{\infty} \leq 1$.
\end{lemma}

\begin{proof}
For $R > 0$, we can write
\begin{align*}
    \int_0^T \sup_{\lVert \varphi\rVert_{\infty} \leq 1} \langle &\mu^n_t - \mu_t, \varphi\rangle \, \d t \\
    &\leq 2 \int_0^T \langle \lvert \mu^n_t - \mu_t\rvert, \bf{1}_{B_R}\rangle \, \d t + 2 \int_0^T \bigl\langle \lvert \mu^n_t - \mu_t\rvert, \bf{1}_{B_R^C}\bigr\rangle \, \d t \\
    &\leq 2 \lVert e^{-\eta/2}\rVert_{L^2} \int_0^T \lVert \bf{1}_{B_R}(\mu^n_t - \mu_t)\rVert_{\eta} \, \d t + 2 \bigl\lVert \bf{1}_{B_R^C} \bigr\rVert_{\bar{\eta}} \int_0^T \lVert \mu^n_t\rVert_{\eta} + \lVert \mu_t\rVert_{\eta} \, \d t.
\end{align*}
The second integral in the last line is bounded uniformly in $n \geq 1$ while $\bigl\lVert \bf{1}_{B_R^C} \bigr\rVert_{\bar{\eta}}$ vanishes as $R \to \infty$. Thus, for any $\epsilon > 0$, choosing $R$ large enough, we can make the supremum in $n \geq 1$ of the second summand in the last line smaller than $\epsilon/2$. Then, we pick $N \geq 1$ large enough, such that for any $n \geq N$, the first term in the last line is bounded by $\epsilon/2$. Together, we get $\int_0^T \sup_{\lVert \varphi\rVert_{\infty} \leq 1} \langle \mu^n_t - \mu_t, \varphi\rangle \, \d t \leq \epsilon$ for all $n \geq N$. Since $\epsilon > 0$ was arbitrary, the desired convergence follows.
\end{proof}

\begin{lemma} \label{lem:int_conv}
Suppose that
\begin{itemize}
    \item $(\mu^n)_n$ and $\mu$ are as in the statement of Lemma \ref{lem:loc_l2_to_d0};
    \item $(g^n)_n$ is a sequence in $L^2([0, T]; L^2_{\textup{loc}}(\R^d; \R^{d_G}))$, such that $g^n_t(x) \in G$ for all $(t, x) \in [0, T] \times \R^d$ and which converges weakly to some $g \in L^2([0, T]; L^2_{\textup{loc}}(\R^d; \R^{d_G}))$;
    \item $(w^n)_n$ is a convergent sequence in $C([0, T]; \R^d)$ with limit $w$;
    \item $F \define [0, T] \times \R^d \times \M(\R^d) \times G \to \R$ is measurable and bounded such that $v \mapsto F(t, x, v, h)$ is continuous uniformly in $h \in G$ for all $(t, x) \in [0, T] \times \M(\R^d)$;
    \item $\varphi \in C_c([0, T] \times \R^d)$.
\end{itemize}
Set $\tilde{\mu}^n_t = (\id + w^n_t)^{\#}\mu^n_t$ and $\tilde{\mu}_t = (\id + w_t)^{\#}\mu_t$. Then, 
\begin{enumerate}[noitemsep, label = \textup{(\roman*)}]
    \item \label{it:conv_1} if $F$ is affine in $h$, it follows that
    \begin{align*}
        \lim_{n \to \infty}\int_0^T \Bigl\langle \mu^n_t, \varphi(t, \cdot) F\bigl(t, \cdot + w^n_t, \tilde{\mu}^n_t, g^n_t\bigr) \Bigr\rangle \, \d t = \int_0^T \Bigl\langle \mu_t, \varphi(t, \cdot) F\bigl(t, \cdot + w_t, \tilde{\mu}_t, g_t\bigr) \Bigr\rangle \, \d t;
    \end{align*}
    \item \label{it:conv_2} if $h \mapsto F(t, x, h, v)$ is convex for all $(t, x, v) \in [0, T] \times \R^d \times \M(\R^d)$, it follows that
    \begin{equation*}
        \liminf_{n \to \infty}\int_0^T \bigl\langle \mu^n_t, F\bigl(t, \cdot + w^n_t, \tilde{\mu}^n_t, g^n_t\bigr)\bigr\rangle \, \d t \geq \int_0^T \bigl\langle \mu_t, F\bigl(t, \cdot + w_t, \tilde{\mu}_t, g_t\bigr)\bigr\rangle \, \d t.
    \end{equation*}
    \item \label{it:conv_3} if $(g^n)_n$ converges in the strong topology on $L^2([0, T]; L^2_{\textup{loc}}(\R^d; \R^{d_G}))$ and the map $h \mapsto F(t, x, v, h)$ is continuous for all $(t, x, v) \in [0, T] \times \R^d \times \M(\R^d)$, it follows that
    \begin{equation*}
        \lim_{n \to \infty}\int_0^T \Bigl\langle \mu^n_t, \varphi(t, \cdot) F\bigl(t, \cdot + w^n_t, \tilde{\mu}^n_t, g^n_t\bigr) \Bigr\rangle \, \d t = \int_0^T \Bigl\langle \mu_t, \varphi(t, \cdot) F\bigl(t, \cdot + w_t, \tilde{\mu}_t, g_t\bigr)\Bigr\rangle \, \d t;
    \end{equation*}
    \item \label{it:conv_4} if $(g^n)_n$ converges in the strong topology on $L^2([0, T]; L^2_{\textup{loc}}(\R^d; \R^{d_G}))$ and the map $h \mapsto F(t, x, v, h)$ is upper semicontinuous for all $(t, x, v) \in [0, T] \times \R^d \times \M(\R^d)$, it follows that
    \begin{equation*}
        \limsup_{n \to \infty}\int_0^T \bigl\langle \mu^n_t, F\bigl(t, \cdot + w^n_t, \tilde{\mu}^n_t, g^n_t\bigr) \bigr\rangle \, \d t \leq \int_0^T \bigl\langle \mu_t, F\bigl(t, \cdot + w_t, \tilde{\mu}_t, g_t\bigr)\bigr\rangle \, \d t.
    \end{equation*}
\end{enumerate}
\end{lemma}

\begin{proof}
Let us define $\tilde{g}^n$, $\tilde{g} \in L^2([0, T]; L^2_{\text{loc}}(\R^d; \R^{d_G}))$ by $\tilde{g}^n_t(x) = g^n_t(x - w^n_t)$ and $\tilde{g}_t(x) = g_t(x - w_t)$ and, similarly, set $\tilde{\varphi}^n_t(x) = \varphi(t,x - w^n_t)$ and $\tilde{\varphi}_t(x) = \varphi(t, x - w_t)$. Note that, like $(\mu^n_t)_n$, the sequence $(\tilde{\mu}^n)_n$ is bounded in $L^2([0, T]; L^2_{\eta}(\R^d))$ and converges to $\tilde{\mu}$ in $L^2([0, T]; L^2_{\eta, \text{loc}}(\R^d))$, and, like $(g^n)_n$, the sequence $(\tilde{g}^n)_n$ converges weakly to $\tilde{g}$ in the space $L^2([0, T]; L^2_{\text{loc}}(\R^d; \R^{d_G}))$. Thus, by Lemma \ref{lem:loc_l2_to_d0} it holds that 
\begin{equation} \label{eq:pushfwd_conv}
    \int_0^T \sup_{\lVert \varphi\rVert_{\infty} \leq 1} \langle \tilde{\mu}^n_t - \tilde{\mu}_t, \varphi\rangle \, \d t \to 0
\end{equation}
and, in particular, $\int_0^T d_0^2(\tilde{\mu}^n_t, \tilde{\mu}_t) \, \d t \to 0$. Next, since $G$ is convex and $(g^n)_n$ converges weakly to $g$ in $L^2([0, T]; L^2_{\text{loc}}(\R^d; \R^{d_G}))$, $g^n_t(x) \in G$ for all $(t, x) \in [0, T] \times \R^d$ implies the same is true for $g$, i.e.\@ $g_t(x) \in G$ for all $(t, x) \in [0, T] \times \R^d$. In the proof of the four statements, we use the decomposition
\begin{align} \label{eq:conv_1}
\begin{split}
    \int_0^T &\Bigl(\Bigl\langle \mu^n_t, \varphi(t, \cdot) F\bigl(t, \cdot + w^n_t, \tilde{\mu}^n_t, g^n_t\bigr) \Bigr\rangle - \Bigl\langle \mu_t, \varphi(t, \cdot) F\bigl(t, \cdot + w_t, \tilde{\mu}_t, g_t\bigr) \Bigr\rangle\Bigr) \, \d t \\
    &= \int_0^T \bigl\langle \tilde{\mu}^n_t - \tilde{\mu}_t, \tilde{\varphi}^n_t F\bigl(t, \cdot, \tilde{\mu}^n_t, \tilde{g}^n_t\bigr) \bigr\rangle \, \d t \\
    &\ \ \ + \int_0^T \Bigl\langle \tilde{\mu}_t, (\tilde{\varphi}^n_t - \tilde{\varphi}_t)F\bigl(t, \cdot, \tilde{\mu}^n_t, \tilde{g}^n_t \bigr) \Bigr\rangle \, \d t \\
    &\ \ \ + \int_0^T \Bigl\langle \tilde{\mu}_t,\tilde{\varphi}_t \Bigl(F\bigl(t, \cdot, \tilde{\mu}^n_t, \tilde{g}^n_t \bigr) - F\bigl(t, \cdot, \tilde{\mu}_t, \tilde{g}^n_t \bigr)\Bigr) \Bigr\rangle \, \d t \\
    &\ \ \ + \int_0^T \Bigl\langle \tilde{\mu}_t, \tilde{\varphi}_t \Bigl(F\bigl(t, \cdot, \tilde{\mu}_t, \tilde{g}^n_t\bigr) - F\bigl(t, \cdot, \tilde{\mu}_t, \tilde{g}_t \bigr)\Bigr) \Bigr\rangle \, \d t.
\end{split}
\end{align}
Note that for Item \ref{it:conv_2} and \ref{it:conv_4}, the appearances of $\varphi$ and the functions derived from it should be removed, in which case the second term on the right-hand side above equals zero. Now, the first three terms on the right-hand side above vanish as $n \to \infty$ independently of the specific assumptions made in Items \ref{it:conv_1} to \ref{it:conv_4}. Indeed, the first expression vanishes by \eqref{eq:pushfwd_conv}. Next, for the second term in \eqref{eq:conv_1} (relevant for Items \ref{it:conv_1} and \ref{it:conv_3}), we have by assumption that
\begin{equation*}
    \Bigl\lvert (\tilde{\varphi}^n_t(x) - \tilde{\varphi}_t(x))F\bigl(t, x, \tilde{\mu}^n_t, \tilde{g}^n_t(x) \bigr) \Bigr\rvert \leq \lVert F\rVert_{\infty} \lvert \tilde{\varphi}^n_t(x) - \tilde{\varphi}_t(x)\rvert \to 0
\end{equation*}
since $w^n_t \to w_t$ and by continuity of $\varphi$. So convergence follows from the dominated convergence theorem. Lastly, since $F$ is continuous in the measure argument uniformly in the control, we have
\begin{equation*}
    \Bigl\lvert F\bigl(t, x, \tilde{\mu}^n_t, \tilde{g}^n_t(x) \bigr) - F\bigl(t, x, \tilde{\mu}^n_t, \tilde{g}_t(x) \bigr)\Bigr\rvert \to 0
\end{equation*}
as $n \to \infty$. Thus, convergence of the third expression on the right-hand side of \eqref{eq:conv_1} is again a consequence of the dominated convergence theorem. 

The behaviour as $n \to \infty$ of the remaining terms on the right-hand side will be determined based on the specific assumptions made in Items \ref{it:conv_1} to \ref{it:conv_1}. We proceed in sequence.

\textit{Item \ref{it:conv_1}}: The last term on the right-hand side of \eqref{eq:conv_1} vanishes as $n \to \infty$ since $\tilde{\varphi}$ has a compact support, $F$ is affine in the control argument, and the sequence $(\tilde{g}^n)_n$ converges weakly to $\tilde{g}$ in $L^2([0, T]; L^2_{\text{loc}}(\R^d; \R^{d_G}))$.

\textit{Item \ref{it:conv_2}}: Note that the convexity of $F$ in the control argument implies convexity of the function $L^2([0, T]; L^2_{\text{loc}}(\R^d; G)) \ni h \mapsto \int_0^T \langle \tilde{\mu}_t, F(t, \cdot, \tilde{\mu}_t, h_t)\rangle \, \d t$ and, therefore, its weak lower semicontinuity. Since $(\tilde{g}^n)_n$ converges weakly to $\tilde{g}$ in $L^2([0, T]; L^2_{\text{loc}}(\R^d; \R^{d_G}))$, we get that
\begin{equation*}
    \liminf_{n \to \infty} \int_0^T \Bigl\langle \tilde{\mu}_t, F(t, \cdot, \tilde{\mu}_t, \tilde{g}^n_t) - F(t, \cdot, \tilde{\mu}_t, \tilde{g}_t)\Bigr\rangle \, \d t \geq 0.
\end{equation*}

\textit{Item \ref{it:conv_3}}: This is again a straightforward consequence of the dominated convergence theorem.

\textit{Item \ref{it:conv_4}}: Let us pick a subsequence $(n_k)_k$ that achieves the limit superior of the expression
\begin{equation*}
    \int_0^T \Bigl\langle \tilde{\mu}_t, F(t, \cdot, \tilde{\mu}_t, \tilde{g}^n_t) - F(t, \cdot, \tilde{\mu}_t, \tilde{g}_t)\Bigr\rangle \, \d t.
\end{equation*}
By choosing a further subsequence if necessary (which we still denote by $(n_k)_k$), we may assume that $\tilde{g}^{n_k}_t(x)$ converges to $\tilde{g}_t(x)$ for a.e.\@ $(t, x) \in [0, T] \times \R^d$. Then upper semicontinuity of $F$ in the control argument implies that
\begin{equation*}
    \limsup_{k \to \infty} F\bigl(t, x, \tilde{\mu}_t, \tilde{g}^{n_k}_t(x)\bigr) \leq F\bigl(t, x, \tilde{\mu}_t, \tilde{g}_t(x)\bigr).
\end{equation*}
Multiplying both sides by $\tilde{\mu}_t(x)$, integrating over $[0, T] \times \R^d$, and then applying the reverse version of Fatou's lemma gives
\begin{align*}
    \limsup_{n \to \infty} \int_0^T \Bigl\langle \tilde{\mu}_t, F&(t, \cdot, \tilde{\mu}_t, \tilde{g}^n_t) - F(t, \cdot, \tilde{\mu}_t, \tilde{g}_t)\Bigr\rangle \, \d t \\
    &= \lim_{k \to \infty} \int_0^T \Bigl\langle \tilde{\mu}_t, F(t, \cdot, \tilde{\mu}_t, \tilde{g}^{n_k}_t) - F(t, \cdot, \tilde{\mu}_t, \tilde{g}_t)\Bigr\rangle \, \d t \\
    &\leq \int_0^T \limsup_{k \to \infty} \Bigl\langle \tilde{\mu}_t, F(t, \cdot, \tilde{\mu}_t, \tilde{g}^{n_k}_t) - F(t, \cdot, \tilde{\mu}_t, \tilde{g}_t)\Bigr\rangle \, \d t \\
    &\leq 0.
\end{align*}
This establishes Item \ref{it:conv_4} and concludes the proof.
\end{proof}

%% file: biblio.bib
@book{bogachev_measure_theory_vol_2_2007,
    author = {Vladimir I. Bogachev},
    doi = {10.1007/978-3-540-34514-5},
    publisher = {Springer Berlin, Heidelberg},
    title = {Measure Theory},
    volume = {2},
    year = {2007}
}

@book{barbu_ana_contr_1993,
  title={Analysis and Control of Nonlinear Infinite Dimensional Systems},
  author={Viorel Barbu},
  series={Mathematics in Science and Engineering},
  year={1993},
  volume={190},
  publisher={Academic Press, Inc.}
}

@article{daudin_oc_const_2023,
title = {Optimal control of the Fokker--Planck equation under state constraints in the Wasserstein space},
journal = {J.\@ Math.\@ Pures Appl.\@},
volume = {175},
pages = {37--75},
year = {2023},
doi = {10.1016/j.matpur.2023.05.002},
author = {Samuel Daudin}
}

@article{barbu_nonlin_fp_ctrl_2023,
author = {Viorel Barbu},
title = {Existence of Optimal Control for Nonlinear Fokker--Planck Equations in \(\boldsymbol{L^1(\mathbb{R}^d)}\).},
journal = {SIAM J.\@ Control},
volume = {61},
number = {3},
pages = {1213--1230},
year = {2023},
doi = {10.1137/22M1485243}
}

@article{anita_nonlin_fp_2024,
title = {Controlling a generalized Fokker--Planck equation via inputs with nonlocal action},
journal = {Nonlinear Anal.\@},
volume = {241},
pages = {113476},
year = {2024},
doi = {10.1016/j.na.2023.113476},
author = {\c{S}tefana-Lucia Ani\c{t}a}
}

@article{carrillo_mfc_2020,
title = {On a mean field optimal control problem},
journal = {Nonlinear Anal.\@},
volume = {199},
pages = {112039},
year = {2020},
doi = {10.1016/j.na.2020.112039},
author = {Jos\'e A.\@ Carrillo and Edgard A.\@ Pimentel and Vardan K.\@ Voskanyan}
}

@misc{carmona_nssc_2023,
    title={Non-standard Stochastic Control with Nonlinear Feynman-Kac Costs},
    author={Ren\'e Carmona and Mathieu Lauri\`ere and Pierre-Louis Lions},
    year={2023},
    eprint={2312.00908},
    archivePrefix={arXiv}
    % primaryClass={math.PR}
}

@Article{hambly_spde_model_2019,
  author={Ben Hambly and Andreas S{\o}jmark},
  title={{An SPDE model for systemic risk with endogenous contagion}},
  journal={Financ.\@ Stoch.\@},
  year=2019,
  volume={23},
  number={3},
  pages={535--594},
  month={6},
  doi={10.1007/s00780-019-00396-1}
}

@article{burzoni_mean_field_absorption_2023,
title = {Mean field games with absorption and common noise with a model of bank run},
journal = {Stoch.\@ Process.\@ Their Appl.\@},
volume = {164},
pages = {206--241},
year = {2023},
doi = {10.1016/j.spa.2023.07.007},
author = {Matteo Burzoni and Luciano Campi}
}

@article{pham_dynamic_programming_mkv_2017,
author = {Huy\^{e}n Pham and Xiaoli Wei},
title = {Dynamic Programming for Optimal Control of Stochastic McKean--Vlasov Dynamics},
journal = {SIAM J.\@ Control Optim.\@},
volume = {55},
number = {2},
pages = {1069--1101},
year = {2017},
doi = {10.1137/16M1071390}
}

@book{liptser_statistics_rp_1977,
  title={Statistics of Random Processes I: General Theory},
  author={Robert S.\@ Liptser and Albert N.\@ Shiryaev},
  isbn={978-1-4757-1665-8},
  series={Stochastic Modelling and Applied Probability},
  year={1977},
  doi={10.1007/978-1-4757-1665-8},
  publisher={Springer New York}
}

@article{bensoussan_smp_spde_1983,
title = {Stochastic maximum principle for distributed parameter systems},
author={Alain Bensoussan},
journal = {J.\@ Franklin Inst.\@},
volume = {315},
number = {5},
pages = {387--406},
year = {1983},
doi = {10.1016/0016-0032(83)90059-5}
}

@article{lasry_mfg_st_2006,
title = {Jeux à champ moyen. I – Le cas stationnaire},
journal = {C.\@ R.\@ Math.\@},
volume = {343},
number = {9},
pages = {619--625},
year = {2006},
doi = {10.1016/j.crma.2006.09.019},
author = {Jean-Michel Lasry and Pierre-Louis Lions}
}

@book{gall_bm_2016,
  title={Brownian Motion, Martingales, and Stochastic Calculus},
  author={Jean-Fran\c{c}ois {Le Gall}},
  series={Graduate Texts in Mathematics},
  year={2016},
  doi={doi.org/10.1007/978-3-319-31089-3},
  publisher={Springer International Publishing}
}

@article{kurtz_weak_strong_2014,
author = {Thomas Kurtz},
title = {{Weak and strong solutions of general stochastic models}},
volume = {19},
journal = {Electron.\@ Commun.\@ Probab.\@},
publisher = {Institute of Mathematical Statistics and Bernoulli Society},
pages = {1 -- 16},
year = {2014},
doi = {10.1214/ECP.v19-2833}
}

@article{djete_mkv_control_limit_2022,
author = {Mao Fabrice Djete and Dylan Possama\"i and Xiaolu Tan},
title = {McKean–Vlasov Optimal Control: Limit Theory and Equivalence Between Different Formulations},
journal = {Math.\@ Oper.\@ Res.\@},
volume={47},
number={4},
year = {2022},
pages = {2547--3399},
doi = {10.1287/moor.2021.1232}
}

@book{lions_bvp_1972,
  title={Non-Homogeneous Boundary Value Problems and Applications},
  subtitle={Vol.\@ 1},
  author={Jacques-Louis Lions and Enrico Magenes},
  isbn={978-3-642-65161-8},
  series={Grundlehren der mathematischen Wissenschaften},
  edition={1},
  year={1972},
  publisher={Springer-Verlag, Berlin Heidelberg},
  doi={10.1007/978-3-642-65161-8}
}

@article{lacker_mimicking_2020,
  title={Superposition and mimicking theorems for conditional McKean–Vlasov equations},
  author={Daniel Lacker and Mykhaylo Shkolnikov and Jiacheng Zhang},
  journal={J.\@ Eur.\@ Math.\@ Soc.\@ },
  year={2023},
  volume={25},
  number={8},
  pages={3229--3288},
  doi={10.4171/JEMS/1266}
}

@article{achdou_mfc_pde_2015,
title = {On the system of partial differential equations arising in mean field type control},
journal = {Discrete Contin.\@ Dyn.\@ Syst.\@},
volume = {35},
number = {9},
pages = {3879--3900},
year = {2015},
doi = {10.3934/dcds.2015.35.3879},
author = {Yves Achdou and Mathieu Lauri\`ere}
}

@book{bensoussan_mfg_mfc_2013,
  title={Mean Field Games and Mean Field Type Control Theory},
  author={Alain Bensoussan and Jens Frehse and Phillip Yam},
  edition={1},
  series = {SpringerBriefs in Mathematics},
  year={2013},
  isbn={978-1-4614-8508-7},
  publisher={Springer, New York, NY},
  doi={10.1007/978-1-4614-8508-7}
}

@article{hambly_mckean_vlasov_blow_up_2019,
    title={A McKean-Vlasov Equation with Positive Feedback and Blow-Ups},
    author={Ben Hambly and Sean Ledger and Andreas S\o{}jmark},
    year={2019},
    journal={Ann.\@ Appl.\@ Probab.\@},
    volume={29},
    number={4},
    pages={2338--2373},
    doi={10.1214/18-AAP1455}
}

@article{nadtochiy_mean_field_network_2020,
title = "Mean field systems on networks, with singular interaction through hitting times",
author = "Sergey Nadtochiy and Mykhaylo Shkolnikov",
year = "2020",
doi = "10.1214/19-AOP1403",
volume = "48",
pages = "1520--1556",
journal = "Ann.\@ Probab.\@",
fjournal = "Annals of Probability",
issn = "0091-1798",
publisher = "Institute of Mathematical Statistics",
number = "3",
}

@article{nadtochiy_ps_singular_inter_2019,
author = {Sergey Nadtochiy and Mykhaylo Shkolnikov},
title = {Particle systems with singular interaction through hitting times: Application in systemic risk modeling},
volume = {29},
journal = {Ann.\@ Appl.\@ Probab.\@},
number = {1},
publisher = {Institute of Mathematical Statistics},
pages = {89 -- 129},
year = {2019},
doi = {10.1214/18-AAP1403}
}

@article{hambly_mckean_vlasov_absorbing_2017,
    title={A stochastic McKean–Vlasov equation for absorbing diffusions on the half-line},
    author={Ben Hambly and Sean Ledger},
    year={2017},
    journal={Ann.\@ Appl.\@ Probab.\@},
    fjournal = {The Annals of Applied Probability},
    volume={27},
    number={5},
    pages={2698--2752},
    doi={10.1214/16-AAP1256}
}

@article{hu_semilinear_bspde_1991,
author = {Ying Hu and Shige Peng},
title = {Adapted solution of a backward semilinear stochastic evolution equation},
journal = {Stoch.\@ Anal.\@ Appl.\@},
volume = {9},
number = {4},
pages = {445--459},
year = {1991},
publisher = {Taylor \& Francis},
doi = {10.1080/07362999108809250}
}

@article{zhou_duality_1992,
title = {A duality analysis on stochastic partial differential equations},
journal = {J.\@ Funct.\@ Anal.\@},
volume = {103},
number = {2},
pages = {275--293},
year = {1992},
doi = {10.1016/0022-1236(92)90122-Y},
author = {Xun Yu Zhou}
}

@article{zhou_nec_spde_1993,
author = {Xun Yu Zhou},
title = {On the Necessary Conditions of Optimal Controls for Stochastic Partial Differential Equations},
journal = {SIAM J.\@ Control Optim.\@},
volume = {31},
number = {6},
pages = {1462--1478},
year = {1993},
doi = {10.1137/0331068}
}

@article{pardoux_bspde_1980,
author = {\'Etienne Pardoux},
title = {Stochastic partial differential equations and filtering of diffusion processes},
journal = {Stoch.\@},
fjournal={Stochastics},
volume = {3},
number = {1--4},
pages = {127--167},
year  = {1980},
publisher = {Taylor \& Francis},
doi = {10.1080/17442507908833142}
}

@book{zheng_nonlinear_2004,
   title =     {Nonlinear Evolution Equations},
   author =    {Songmu Zheng},
   publisher = {Chapman \& Hall/CRC},
   year =      {2004},
   edition={1},
   doi={10.1201/9780203492222},
   series =    {Monographs and Surveys in Pure and Applied Math}
  }

@article{peng_stochastic_hjb_1992,
author = {Shige Peng},
title = {Stochastic Hamilton–Jacobi–Bellman Equations},
journal = {SIAM J.\@ Control Optim.\@},
volume = {30},
number = {2},
pages = {284--304},
year = {1992},
doi = {10.1137/0330018}
}

@article{al_hussein_bspde_2009,
author = {Abdul Rahman {Al-Hussein}},
sortkey  = {Alhussein},
title = {Backward stochastic partial differential equations driven by infinite-dimensional martingales and applications},
journal = {Stoch.\@},
volume = {81},
number = {6},
pages = {601--626},
year  = {2009},
publisher = {Taylor \& Francis},
doi = {10.1080/17442500903370202}
}

@book{carmona_mfg_2018, 
    title = {Probabilistic Theory of Mean Field Games with Applications I},
    author = {Ren\'e Carmona and Fran\c{c}ois Delarue},
    year = {2018},
    series = {Probability Theory and Stochastic Modelling},
    publisher = {Springer International Publishing},
    edition = {1},
    volume = {83},
    isbn = {978-3-319-58920-6},
    doi = {10.1007/978-3-319-58920-6}
}

@book{carmona_mfg_ii_2018, 
    title = {Probabilistic Theory of Mean Field Games with Applications II},
    author = {Ren\'e Carmona and Fran\c{c}ois Delarue},
    year = {2018},
    series = {Probability Theory and Stochastic Modelling},
    publisher = {Springer International Publishing},
    edition = {1},
    volume = {84},
    isbn = {978-3-319-56436-4},
    doi = {10.1007/978-3-319-56436-4}
}

@article{benoussan_mp_dp_1983,
author = {Alain Bensoussan},
title = {Maximum principle and dynamic programming approaches of the optimal control of partially observed diffusions},
journal = {Stoch.\@},
volume = {9},
number = {3},
pages = {169--222},
year = {1983},
publisher = {Taylor \& Francis},
doi = {10.1080/17442508308833253}
}

@article{lauriere_dp_mfc_2014,
title = {Dynamic programming for mean-field type control},
journal = {C.\@ R.\@ Math.\@},
fjournal = {Comptes Rendus Mathematique},
volume = {352},
number = {9},
pages = {707--713},
year = {2014},
issn = {1631-073X},
doi = {10.1016/j.crma.2014.07.008},
author = {Mathieu Lauri\`ere and Olivier Pironneau}
}

@article{campi_mfg_absorption_2018,
author = {Luciano Campi and Markus Fischer},
title = {$N$-player games and mean-field games with absorption},
volume = {28},
journal = {Ann.\@ Appl.\@ Probab.\@},
fjournal = {The Annals of Applied Probability},
number = {4},
publisher = {Institute of Mathematical Statistics},
pages = {2188 -- 2242},
year = {2018},
doi = {10.1214/17-AAP1354}
}

@article{campi_mfg_hitting_2021,
author = {Luciano Campi and Maddalena Ghio and Giulia Livieri},
title = {N-Player games and mean-field games with smooth dependence on past absorptions},
volume = {57},
journal = {Ann.\@ Inst.\@ H.\@ Poincar\'e Probab.\@ Statist.\@},
number = {4},
publisher = {Institut Henri Poincar\'e},
pages = {1901--1939},
year = {2021},
doi = {10.1214/20-AIHP1138}
}

@article{lasry_mfg_2007,
author = {Jean-Michel Lasry and Pierre-Louis Lions},
title = {Mean field games},
journal = {Jpn.\@ J.\@ Math.\@},
fjournal = {Japanese Journal of Mathematics},
volume = {2},
number={1},
year={2007},
doi ={10.1007/s11537-007-0657-8},
pages = {229--260}
}

@article{cardaliaguet_mfg_2014,
author = {Pierre Cardaliaguet and P.\@ Jameson Graber and Alessio Porretta and Daniela Tonon},
title = {Second order mean field games with degenerate diffusion and local coupling},
journal = {Nonlinear Differ.\@ Equ.\@ Appl.\@},
volume={22},
year={2014},
doi ={10.1007/s00030-015-0323-4},
pages = {1287--1317}
}

@misc{lions_cond_proc_2016,
      title={\'Equations de HJB et extensions de la th\'eorie classique du contr\^ole stochastique}, 
      author={Pierre-Louis Lions},
      howpublished={Coll\`ege de France},
      year={2016},
      % url={https://www.college-de-france.fr/fr/agenda/cours/equations-de-hjb-et-extensions-de-la-theorie-classique-du-controle-stochastique}
}

@article{fleig_oc_fpe_2017,
author = {Arthur Fleig and Roberto Guglielmi},
title = {Optimal Control of the Fokker--Planck Equation with Space-Dependent Controls},
volume = {174},
journal = {J.\@ Optim.\@ Theory Appl.\@},
pages = {408 -- 427},
year = {2017},
doi = {10.1007/s10957-017-1120-5}
}

@article{breitenbach_mp_fpe_2020,
author = {Tim Breitenbach and Alessio Borz\`i},
title = {The Pontryagin maximum principle for solving Fokker–Planck optimal control problems},
volume = {76},
journal = {Comput.\@ Optim.\@ Appl.\@},
pages = {499 -- 533},
year = {2020},
doi = {10.1007/s10589-020-00187-x}
}

@article{carmona_fbsde_smp_2015,
author = {Ren\'e Carmona and Fran\c{c}ois Delarue},
title = {Forward–backward stochastic differential equations and controlled McKean–Vlasov dynamics},
volume = {43},
journal = {Ann. Probab.},
number = {5},
publisher = {Institute of Mathematical Statistics},
pages = {2647--2700},
year = {2015},
doi = {10.1214/14-AOP946}
}

@misc{cardaliaguet_rate_mfc_2023,
      title={An algebraic convergence rate for the optimal control of McKean--Vlasov dynamics}, 
      author={Pierre Cardaliaguet and Samuel Daudin and Joe Jackson and Panagiotis Souganidis},
      year={2023},
      eprint={2203.14554},
      archivePrefix={arXiv}
}

@article{anita_sde_fpe_2021,
  title={Optimal Control of Stochastic Differential Equations via Fokker--Planck Equations},
  author={\c{S}tefana-Lucia Ani\c{t}a},
  journal={Appl.\@ Math.\@ Opt.\@},
  year={2021},
  volume={84},
  pages={1555 -- 1583},
  doi={10.1007/s00245-021-09804-5}
}

@book{metivier_semimartingales_1982,
title = {Semimartingales},
author = {Michel M\'etivier},
publisher = {De Gruyter},
address = {Berlin, New York},
doi = {doi:10.1515/9783110845563},
isbn = {9783110845563},
year = {1982},
volume = {2},
series = {De Gruyter Studies in Mathematics}
}

@misc{hambly_mvcp_arxiv_2023,
      title={Control of McKean--Vlasov SDEs with Contagion Through Killing at a State-Dependent Intensity}, 
      author={Ben Hambly and Philipp Jettkant},
      year={2023},
      eprint={2310.15854},
      archivePrefix={arXiv}
      % primaryClass={math.PR}
}

@article{al_hussein_bspde_2006,
title = {Backward stochastic partial differential equations in infinite dimensions},
author = {Abdul Rahman {Al-Hussein}},
sortkey  = {Alhussein},
pages = {1--22},
volume = {14},
number = {1},
journal = {Random Oper.\@ and Stoch.\@ Equ.\@},
doi = {10.1515/156939706776138020},
year = {2006}
}

@article{jin_1999_bspde,
    author = {Jin Ma and Jiongmin Yong},
    doi = {10.1007/s004400050205},
    journal = {Probab.\@ Theory Relat.\@ Fields},
    number = {2},
    pages = {135--170},
    title = {On linear, degenerate backward stochastic partial differential equations},    
    volume = {113},
    year = {1999}
}

@article{gomes_mfg_2016,
title = {Local Regularity for Mean-Field Games in the Whole Space},
journal = {Minimax Theory Appl.\@},
volume = {1},
number = {1},
pages = {65--82},
year = {2016},
author = {Diogo A.\@ Gomes and Edgard Pimentel}
}

@article{lasry_mfg_fh_2006,
title = {Jeux \`a champ moyen. II – Horizon fini et contr\^ole optimal},
journal = {C.\@ R.\@ Math.\@},
volume = {343},
number = {10},
pages = {679--684},
year = {2006},
doi = {10.1016/j.crma.2006.09.018},
author = {Jean-Michel Lasry and Pierre-Louis Lions}
}

@article{cardaliaguet_fo_mgf_2022,
author = {Pierre Cardaliaguet and Panagiotis E. Souganidis},
title = {On first order mean field game systems with a common noise},
volume = {32},
journal = {Ann.\@ Appl.\@ Probab.\@},
number = {3},
publisher = {Institute of Mathematical Statistics},
pages = {2289--2326},
year = {2022},
doi = {10.1214/21-AAP1734}
}

@book{cardaliaguet_master_2019,
  title={The Master Equation and the Convergence Problem in Mean Field Games},
  author={Pierre Cardaliaguet and Fran\c{c}ois Delarue and Jean-Michel Lasry and Pierre-Louis Lions},
  series={Annals of Mathematics Studies},
  year={2019},
  isbn={978-0-691-19070-9},
  publisher={Princeton University Press}
}

@misc{cardaliaguet_mfg_common_2022,
    title={Mean field games with common noise and degenerate idiosyncratic noise},
    author={Pierre Cardaliaguet and Benjamin Seeger and Panagiotis Souganidis},
    year={2022},
    eprint={2207.10209},
    archivePrefix={arXiv}
}

@article{buckdahn_mfc_smp_2011,
    author = {Rainer Buckdahn and Boualem Djehiche and Juan Li},
    doi = {10.1007/s00245-011-9136-y},
    journal = {Appl. Math. Optim.},
    mrnumber = {2822408},
    number = {2},
    pages = {197--216},
    title = {A General Stochastic Maximum Principle for SDEs of Mean-field Type},
    volume = {64},
    year = {2011}
}

@article{andersson_smp_mfc_2011,
    author = {Daniel Andersson and Boualem Djehiche},
    doi = {10.1007/s00245-010-9123-8},
    journal = {Appl. Math. Optim.},
    mrnumber = {2011},
    number = {3},
    pages = {341--356},
    title = {A Maximum Principle for SDEs of Mean-Field Type},
    volume = {63},
    year = {2011}
}

@article{acciaio_smp_mfc_2019,
    author = {Beatrice Acciaio and Julio Backhoff-Veraguas and Ren\'e Carmona},
    doi = {10.1137/18M1196479},
    journal = {SIAM J. Control Optim.},
    mrnumber = {4029806},
    number = {6},
    pages = {3666--3693},
    title = {Extended Mean Field Control Problems: Stochastic Maximum Principle and Transport Perspective},
    volume = {57},
    year = {2019}
}

@article{li_smp_mfc_2012,
    author = {Juan Li},
    doi = {10.1016/j.automatica.2011.11.006},
    journal = {Automatica},
    number = {2},
    pages = {366--373},
    title = {Stochastic maximum principle in the mean-field controls},
    volume = {48},
    year = {2012}
}

@article{jiequn_deep_bsde_2018,
author = {Han Jiequn and Arnulf Jentzen and E. Weinan},
title = {Solving high-dimensional partial differential equations using deep learning},
journal = {Proceedings of the National Academy of Sciences},
volume = {115},
number = {34},
pages = {8505--8510},
year = {2018},
doi = {10.1073/pnas.1718942115}
}
